\documentclass[a4paper;11pt]{amsart}
\usepackage{mathrsfs}\usepackage{graphicx}
 \usepackage[arrow,matrix]{xy}
\usepackage{amsmath,amssymb,amscd,bbm,amsthm,mathrsfs}
\usepackage{color,xcolor}
\usepackage{graphicx}
\usepackage{manfnt}
\usepackage{appendix}\usepackage{subfigure}

\newtheorem{thm}{Theorem}[section]
\newtheorem{lem}{Lemma}[section]
\newtheorem{cor}{Corollary}[section]
\newtheorem{prop}{Proposition}[section]
\newtheorem{rem}{Remark}[section]

\theoremstyle{definition}

\begin{document}
\numberwithin{equation}{section}

 \title[Flag-like  singular integrals and    associated  Hardy spaces ]{Flag-like  singular integrals and   associated Hardy spaces on  a kind of     nilpotent Lie
 groups of
step two}
\author {Wei Wang${}^\dag$ and Qingyan Wu${}^\ddag$}
\thanks{$  \dag$ School of Mathematical Science, Zhejiang University (Zijingang campus),  Zhejiang 310058, China, Email: wwang@zju.edu.cn;}
\thanks{
$ \ddag$ Department of Mathematics,
         Linyi University,
         Shandong  276005, China, Email: qingyanwu@gmail.com}


\begin{abstract} The Cauchy-Szeg\H o singular integral is a fundamental tool in the study of holomorphic $H^p$ Hardy space. But for a kind of Siegel domains, the Cauchy-Szeg\H o kernels are neither product ones nor flag ones  on the  Shilov boundaries, which have the structure of  nilpotent Lie groups $\mathscr N $ of step two.  We use the lifting method to investigate flag-like singular integrals on $\mathscr N $, which includes  these  Cauchy-Szeg\H o ones as a special case.The lifting group is the product $\tilde {\mathscr N }$ of three Heisenberg groups, and   naturally geometric or analytical objects on  $\mathscr N $ are the projection of those on $\tilde {\mathscr N } $.   As in  the   flag case, we introduce  various notions on $\mathscr N $ adapted to  geometric feature of these kernels, such as tubes,  nontangential regions, tube maximal functions,   Littlewood-Paley  functions, tents, shards   and atoms    etc.  They have the feature of tri-parameters, although the second step of the  group  $\mathscr N$ is only $2$-dimensional, i.e. there exists a hidden parameter as in  the   flag case. We also establish  the  corresponding Calder\'on reproducing formula,   characterization of $ L^p(\mathscr N)$ by Littlewood-Paley  functions,    $ L^p$-boundedness of  tube  maximal functions and flag-like singular integrals  and   atomic decomposition of  $H^1$ Hardy space  on $  {\mathscr N } $.
\end{abstract}
\thanks{The first author  is supported
by the National Nature Science Foundation in China (NNSF) (No. 12371082). The second author   is supported by NNSF
(Nos. 12171221, 12071197), the Natural Science Foundation of Shandong Province (Nos. ZR2021MA031,
No. 2020KJI002)}
 \maketitle
 \tableofcontents
\section{Introduction }
A {\it  Siegel domain} is  a  domain given by
\begin{equation}\label{eq:Siegel-domain}
   {\mathcal D }: = \left\{ \zeta= (\zeta',\zeta'') \in\mathbb{C}^{N }\times \mathbb{C}^{m};  \operatorname{Im}
   \zeta'' - \Psi(\zeta' ,\zeta' )\in \Omega\right\},
\end{equation}for some regular cone $\Omega\subset    \mathbb{R}^{m}$   and $\Omega$-positive
 Hermitian form
 $\Psi: \mathbb{C}^{N}\times \mathbb{C}^{N }\rightarrow \mathbb{C}^{m}$. The holomorphic  Hardy space $H^p(\mathcal D )$
consists of all holomorphic functions $f $ on $\mathcal D $ such that
\begin{equation}\label{eq:H2}
   \|f\|_{H^p(\mathcal D )}^p=\sup_{y \in \Omega}\int_{ \mathbb{C}^{N }\times\mathbb{R}^{m}}\left|f(\zeta',  x +\mathbf{i}y
   +\mathbf{i}\Psi(\zeta',\zeta' )  )\right|^pdx
   d\zeta'  <\infty.
\end{equation}

 The bidisc is the simplest  Siegel domain   with non-smooth boundary, but compared to the disc,   the
boundary behavior of holomorphic functions and  holomorphic
Hardy space  become much more complicated   by Malliavin-Malliavin
\cite{MM} and Gundy-Stein \cite{GuS} in  1970s. This phenomenon has stimulated  the development of multi-parameter  harmonic
analysis
 since then (cf. e.g. \cite{CF80} \cite{CF85}  \cite{J} \cite{P}). In particular, the definition of a  multi-parameter  atom
 is more complicated than the  one-parameter one.

The next step  is consider the product $ \mathcal U:=    {\mathcal U}_1\times   {\mathcal U}_2$ of two Siegel upper half spaces
\begin{equation}\label{eq:U-alpha}
    {\mathcal U}_\alpha:=\left\{( {\mathbf{z}}_\alpha,  w_\alpha)\in
\mathbb{C}^{n_\alpha } \times \mathbb{C};\rho_\alpha( { \mathbf z}_\alpha
  ,  w_\alpha
):=\operatorname{Im}  {w}_\alpha-|  {\mathbf z}_\alpha|^2>0 \right\}
 ,\qquad \alpha=1,2.
\end{equation}   Its Shilov boundary,    defined by $\rho_1=\rho_2=0$,   has the
structure of the product $ \mathscr H_1\times\mathscr H_2$ of two Heisenberg groups.
In \cite{WW}, the authors proved that  any  function $f $ belonging to  holomorphic  Hardy space $H^1( \mathcal{ U})$ has a boundary distribution $f^b$ in the bi-parameter
Hardy space $  H^1 (\mathscr H_1\times \mathscr H_2)$ on its Shilov boundary. Since  $f^b$ has an atomic decomposition  by recently developed bi-parameter
harmonic analysis (cf. e.g. \cite{CDLWY} \cite{HLPW} \cite{HLW}), the    Cauchy-Szeg\H o projection can be used to produce holomorphic atoms and to decompose a  holomorphic $H^1( \mathcal{ U})$ function into a sum
of holomorphic atoms.

Harmonic analysis  on the Shilov boundaries of general Siegel domains has the feature of multi-parameters and   plays  an
important
role in the understanding of
boundary behavior of holomorphic functions  and holomorphic   Hardy spaces on  these domains. Notably, the  Cauchy-Szeg\H o kernels on Siegel domains are usually   complicated singular integral kernels, some of which do not fall within the scope of known classes of kernels.   It is known that   for a tube domain  (i.e. Siegel domains \eqref{eq:Siegel-domain} with $m=0$), it is the sum of flag kernels \cite{NRS}.
In this paper, we will consider a family of more complicated
Siegel domains:
\begin{equation}\label{eq:U}
    {\mathcal D} =\Big\{(\mathbf {z},{\mathbf w})\in
\mathbb{C}^{n_1+ n_2+ n_3}\times\mathbb{C}^2 ;\rho_\alpha(
 { \mathbf z}_\alpha,   w_\alpha
):=\operatorname{Im}  {w}_\alpha-|  {\mathbf z}_\alpha|^2-| {\mathbf z}_3 |^2>0, \alpha=1,2\Big\},
\end{equation} where $   {\mathbf{z}}=(  {\mathbf{z}}_1,  {\mathbf{z}}_2,  {\mathbf{z}}_3)\in \mathbb{C}^{n_1 }\times \mathbb{C}^{ n_2 }\times \mathbb{C}^{  n_3}$,
$\mathbf  w=(w_1,w_2)\in \mathbb{C}^2$.
Its Shilov boundary,   defined by $\rho_1=\rho_2=0$,  has the structure of   a nilpotent Lie group $\mathscr N$ of
step two, which  is  $
      \mathbb{C}^{N}\times \mathbb{R}^2$, $N:=n_1+ n_2+ n_3$,  with the multiplication given by
 \begin{equation}\label{eq:multiplication1}
   ( \mathbf{z} , \mathbf{t} ) ( \mathbf{z} ' ,\mathbf{t}  ')= \Big(\mathbf{z}  +\mathbf{z} ',t_1 +t_1 '+ \Phi_1(
   \mathbf{z}_1,  \mathbf{z}_1
   ')+  \Phi_3(
   \mathbf{z}_3,  \mathbf{z}_3
   ' ),t_2 +t_2 '+ \Phi_2(
   \mathbf{z}_2 ,  \mathbf{z}_2
   ')+  \Phi_3(
   \mathbf{z}_3,  \mathbf{z}_3
   ' )  \Big),
\end{equation}  where $\mathbf{t}=(t_1,t_2)$, $\mathbf{t}'=(t_1',t_2')\in \mathbb{R}^2$, and
\begin{equation*}
   \Phi_\mu(\mathbf{z}_\mu,\mathbf{z}_\mu')=2 {\rm Im}  \langle
   \mathbf{z}_\mu,  \mathbf{z}_\mu
   '\rangle .
\end{equation*}Here $\langle \cdot ,  \cdot\rangle $ is the standard Hermitian inner product on $\mathbb{C}^{
n_\mu}$, $\mu=1,2,3$.

When restricted to the group $\mathscr N$, the  Cauchy-Szeg\H o    kernel (cf. Corollary \ref{cor:flat-Szego1}) is
\begin{equation}\begin{split}\label{eq:singular-kernel}
      \frac { 1 }{ 4^2(\frac \pi
 2)^{N+2}} \sum_{k=0}^{n_3}\binom {n_3}{k}
      \frac {( n_1+k)!}{ (|{\mathbf z}_1 |^2+|{\mathbf z}_3 |^2  - \mathbf{i}  t_1  ) ^{n_1+k+1}    }  \cdot \frac {( n_2+n_3-k)!}{
      (|{\mathbf z}_2 |^2+|{\mathbf z}_3 |^2 - \mathbf{i}  t_2  ) ^{n_2+ n_3-k+1}    }.
\end{split} \end{equation}
To see it belonging to a new kind of singular integral kernels, recall the definition  of   product kernels
 and  flag kernels \cite{NS}. Consider   a
decomposition $\mathbb{R}^M=\mathbb{R}^{m_1}\times\cdots \times\mathbb{R}^{m_n}$ into $n$ homogeneous subspaces with given weighted
dilations, and denote the elements of $\mathbb{R}^M$ by $n$-tuples $\mathbf{x}=(\mathbf{x}_1 , \cdots, \mathbf{x}_n)$.
A {\it product kernel on $\mathbb{R}^M$}  relative to this decomposition
  is a distribution $K$ on $\mathbb{R}^M$ which coincides with a smooth function
away from the coordinate subspaces $\mathbf{x}_j=\mathbf{0}$, and satisfies
the size
estimates:    for each multi-index $\alpha:=(\alpha_1 , \cdots, \alpha_n)$,
there is a constant $C_\alpha$ so that
\begin{equation}\label{eq:size-estimates-product}
   |\partial_{\mathbf{x}_1}^{\alpha_1} \cdots  \partial_{\mathbf{x}_n}^{\alpha_n}K(\mathbf{x} )|\leq   C_\alpha  | \mathbf{x}_1  |^{-Q_1-|{\alpha_1}| }\cdots   |
   \mathbf{x}_n  |^{-Q_n- |{\alpha_n}| }
\end{equation}away from the coordinate subspaces,
and also satisfies suitable cancellation conditions, where $ Q_j$ is  the homogeneous dimension of $\mathbb{R}^{m_j}$ and   $|\mathbf{x}_j |$ is
a smooth homogeneous norm on $\mathbb{R}^{m_j}$. Relative to the {\it flag}
\begin{equation*}
   0 \subset V_1\subset\cdots V_{n-1}\subset  \mathbb{R}^M \qquad {\rm with} \quad  V_j=\mathbb{R}^{m_1}\times\cdots \times\mathbb{R}^{m_j},
\end{equation*}
a {\it flag kernel}  is a distribution $K$ on $\mathbb{R}^M$ which coincides with a smooth function
away from   $\mathbf{x}_n=\mathbf{0}$, and satisfies
the size
estimates:
\begin{equation}\label{eq:size-estimates-flag}
   |\partial_{\mathbf{x}_1}^{\alpha_1} \cdots  \partial_{\mathbf{x}_n}^{\alpha_n}K(\mathbf{x} )|\leq   C_\alpha  (| \mathbf{x}_1  |+\cdots+ | \mathbf{x}_n
   |)^{-Q_1-|{\alpha_1}| }\cdots   (| \mathbf{x}_{n -1}
   |+| \mathbf{x}_n  |)^{-Q_{n -1}- |{\alpha_{n -1}}| } | \mathbf{x}_n  |^{-Q_n- |{\alpha_n}| }
\end{equation}for $\mathbf{x}_n \neq \mathbf{0}$, and also satisfies suitable cancellation conditions.

The
Cauchy-Szeg\H o  kernel in \eqref{eq:singular-kernel} is neither a  product one, nor a flag one according to their definitions. Note that
two factors in \eqref{eq:singular-kernel} are  homogenous singular kernels on subspaces
\begin{equation}\label{eq:multiplication_V}
    {\mathscr V}_1 :=   \mathbb{C}^{n_1+   n_3}\times \mathbb{R}_{t_1}, \qquad  {\mathscr V}_2 :=   \mathbb{C}^{n_2+   n_3}\times \mathbb{R}_{t_2},
\end{equation}  with singularities on $\mathbb{C}^{   n_2}$ and $\mathbb{C}^{   n_1}$, respectively.
${\mathscr V}_1$ and  ${\mathscr V}_2$  both have structures of the Heisenberg groups, while their interaction ${\mathscr V}_1\cap {\mathscr V}_2=\mathbb{C}^{
n_3}$ is
nonempty. But  the powers in \eqref{eq:singular-kernel} depend on $k$, and so are different from the homogeneous dimensions $Q_\alpha=2n_\alpha+
2n_3+2$, $\alpha=1,2 $.
To understand the behavior of Cauchy-Szeg\H o singular integral and holomorphic  Hardy space on this class of  Siegel  domains, we    develop  necessary
tools of harmonic
analysis  to handle   singular integral operators with this new type   kernels and associated new Hardy spaces on the group $\mathscr N$. Further applications to
holomorphic  Hardy spaces
on this kind of  Siegel  domains will be discussed in a separate  paper.

The Cauchy-Szeg\H o  kernel in \eqref{eq:singular-kernel}  involves  two   metrics on the Heisenberg groups  ${\mathscr V}_1$ and $ {\mathscr V}_2$, respectively.
The general idea to handle   complicated singular integral operators involving the conflicting metrics  is   ``lifting" to a
 product  or ``simpler" situation.   This powerful idea
has already appeared in different forms in the study of the sub-Laplacian associated to H\"ormander's vector fields by Rothschild-Stein
\cite{RS},  in the study of  Marcinkiewicz multipliers on the Heisenberg group by M\"uller-Ricci-Stein \cite{MRS,MRS2},
  and then in the study of  Kohn-Laplacian $\Box_b$ on quadratic CR manifolds of
higher-codimensions  by Nagel-Ricci-Stein
\cite{NRS} and on     rigid decoupled hypersurfaces  of finite type in $\mathbb{C}^n$ by Nagel-Stein
 \cite{NS} etc. By lifting   the Heisenberg group  $\mathscr  H$ to the product of    $\mathscr  H \times \mathbb{R}$, M\"uller-Ricci-Stein \cite{MRS} introduced
 the flag
 kernels  on  $\mathscr  H$ as the the projection of  standard convolution   kernels of bi-parameters on $\mathscr  H \times
 \mathbb{R}$. They also characterized
flag  kernels directly in terms of the size
estimates and the cancellation conditions on $\mathscr  H  $,  and proved their $L^p$-boundedness for $1<p<\infty$.  The notions of  flag  maximal functions,  flag
Littlewood-Paley  functions,   flag  Calder\'on
reproducing formula, and   flag  $H^p$ Hardy spaces etc. were introduced and developed by
 Han-Lu-Sawyer \cite{HLS}. Based on fractal tiling   of    the Heisenberg  group,
Chen-Cowling-Lee-Li-Ottazzi \cite{CCLLO} constructed shards, as stacks of tiles, to obtain
a ``dyadic decomposition" of  the Heisenberg group   that is adapted to flag  singular integrals, by which they could establish the  atomic decomposition of  flag
$H^1$ Hardy space  and the equivalence of various characterizations of  flag   Hardy space. Recently, the theory of flag singular integrals has developed   rapidly
(cf. \cite{CHW,CCLLO,DLOPW,HLLW,HLS,HLW19,NRSW12,NRSW} and the references therein).

In our case, the lifting group   of  $\mathscr N$ is the product
$
    \tilde{\mathscr  N}:=  {\mathscr H}_1\times {\mathscr H}_2\times {\mathscr H}_3
$ of three
    Heisenberg groups, where  $ {\mathscr H}_\mu:=\mathbb{C}^{ n_\mu }\times  \mathbb{R}$  with the
multiplication given by
\begin{equation}\label{eq:multiplication3}
   (\mathbf{z}_\mu ,t_\mu ) ( \mathbf{z}_\mu' ,t_\mu ')= \left(\mathbf{z}_\mu +\mathbf{z}_\mu', t_\mu +t_\mu '+  \Phi_\mu(
   \mathbf{z}_\mu,  \mathbf{z}_\mu
   ')     \right),
\end{equation}$\mu=1,2,3$.
 To pass   objects   on the lifting  group $\tilde{\mathscr  N} $  to   ones on the group $\mathscr  N$, we use
the   projection $ \pi: \tilde{\mathscr  N}=\mathbb{ C} ^N \times\mathbb{R}^3\rightarrow\mathscr  N=\mathbb{ C} ^N \times\mathbb{R}^2$ defined by
\begin{equation}\label{eq:pi}\begin{split}
   (\mathbf{z},\mathbf{u}) \mapsto (\mathbf{z},u_1+u_3,u_2+u_3),
\end{split} \end{equation}for  $  \mathbf{z} \in \mathbb{ C} ^N,$ $ \mathbf{u}=(u_1 ,u_2,u_3)\in\mathbb{R}^3 $ and $  N=  n_1+  n_2
   + n_3 $.
   The fiber of the  projection $\pi$ over the point $(\mathbf{z},\mathbf{t})\in \mathscr  N$  is the straight line
 \begin{equation}\label{eq:fiber}
    \pi^{-1}(\mathbf{z},\mathbf{t})=\left\{(\mathbf{z}, t_1-u,t_2-u,u ); u\in \mathbb{R}\right\}.
 \end{equation}
 In particular, $\pi$ is a homomorphism of groups
   with the kernel to be the $1$-dimensional Abelian subgroup $
 \pi^{-1}(\mathbf{0}_{\mathbf{z}},\mathbf{0}_{\mathbf{t}})=\{(\mathbf{0}_{\mathbf{z}},-u,-u,u); u\in\mathbb{ R}\}$.

For an   $  L^1 $-function   $F$ on $\tilde{\mathscr  N}$, we define the {\it push-forward function} $\pi_*F$ on $\mathscr  N$
simply to be the integral of $F$ along the fiber as
 \begin{equation}\label{eq:transfer-F}
  \left ( \pi_*F\right)(\mathbf{z},\mathbf{t}):=\int_{ \mathbb{R} }F(\mathbf{z}, t_1-u,t_2-u,u)du.
 \end{equation}
 We introduce the lifting given by  \eqref{eq:pi} because the projection of the Cauchy-Szeg\H o  kernel on $\tilde{\mathscr  N}$ by $\pi$ is exactly
 the Cauchy-Szeg\H o  kernel   \eqref{eq:singular-kernel} on the group $\mathscr  N$ (see Section 2.3 for the details).

Since
  $\tilde{\mathscr  N}$ has three commutative dilations, a natural ball
  is the product of three balls in the Heisenberg groups $ {\mathscr H}_\mu$'s, respectively. Thus a natural ball in the group $\mathscr  N$ is the image of such a
  product under the projection   $\pi$. Motivated by this, we introduce the notion of a tube $ T(\mathbf{g},\mathbf{r})$
for $\mathbf{g}\in {\mathscr  N}$ and $  \mathbf{r} :=\left({r_1} ,
          {r_2} ,r_3 \right)\in \mathbb{  R}^3_+$. It plays the role of a ball  for $\mathscr  N$,  and has the feature of tri-parameters, although the second step
          of the  group  $\mathscr N$ is only $2$-dimensional. Namely, there exists a hidden parameter
as in the theory of flag   singular integrals.
Define the {\it tube maximal function}   as
\begin{equation*}
   M   (f)(\mathbf{g})=\sup_{\mathbf{r}\in \mathbb{R}^3_+} \frac 1{| T(\mathbf{g},\mathbf{r})|}\int_{  T(\mathbf{g},\mathbf{r})}|f (\mathbf{h})|d\mathbf{h}.
\end{equation*}
\begin{thm}\label{thm:maximal} For $1<p<\infty$,    tube  maximal function  $ M  $ is bounded from $L^p(\mathscr N)$ to $L^p(\mathscr N)$.
 \end{thm}

 For an   $  L^1 $-function   $\varphi^{
(\mu)}$
on $ {\mathscr
  {H}}_\mu$, let $\varphi_{r_\mu}^{
(\mu)}$ be
the normalised dilates  and let $\chi_{r_\mu}^{
(\mu)}$ be the normalised characteristic function of the ball $  {B}_\mu(\mathbf{0}_\mu,r_\mu) $ of $ {\mathscr
  {H}}_\mu$. Set
 \begin{equation}\label{eq:varphi-r0}
    \varphi_{\mathbf{r} }:= \pi_* \left(\varphi_{r_1}^{
(1)}   \varphi_{r_2}^{
(2)}\varphi_{r_3}^{
(3)}\right),\qquad \chi_{\mathbf{r} }:= \pi_*  \left(\chi_{r_1}^{
(1)}   \chi_{r_2}^{
(2)}\chi_{r_3}^{
(3)}\right ) ,
 \end{equation}for  $\mathbf{r} :=\left({r_1} ,
          {r_2} ,r_3 \right)\in \mathbb{  R}^3_+  $. Denote by  $ {*}$ and $\tilde{*}$ the convolutions of functions on the groups $  \mathscr N  $ and $ \tilde{{\mathscr N}}
$, respectively.
  We establish
the Calder\'on
reproducing formula on $\mathscr N $.
 \begin{thm}\label{prop:reproducing}
    Suppose that   $\varphi^{
(\mu)}$
is Poisson bounded  on $ {\mathscr H}_\mu$ and w-invertible   with w-inverses $\psi^{
(\mu)}$, $\mu=1,2,3$. Then for $ f \in    L
^1\cap L
^2
( {\mathscr N})$, we have
\begin{equation}\label{eq:reproducing} f =
\int_{\mathbb{R}^3_+} f*\varphi_ {\mathbf{r}}*\psi_ {\mathbf{r}} \frac { d\mathbf{r}}{\mathbf{r}},\qquad \frac {d\mathbf{r}}{\mathbf{r}}=\frac
{dr_1}{r_1}\frac {dr_2}{r_2}\frac {dr_3}{r_3}.
 \end{equation}
 \end{thm}
 The key tool to prove this formula is the following commutativity of convolutions with  $\pi_*$.
 \begin{lem}\label{lem:convolution} For $ F,G  \in L
^1
(\tilde{\mathscr N})$, we have
$\pi_*(F \tilde{*} G)
  = \pi_*(F)*\pi_*(G).
$
     \end{lem}

      For $f \in L
^p
(\mathscr N)$, we define
the {\it    Littlewood-Paley  function} of $f$
as
\begin{equation*}
   g_{\boldsymbol\varphi }(f)(\mathbf{g}):=\left(\int_{\mathbb{R}^3_+}|f* \varphi_{\mathbf{r} }(\mathbf{g})|^2\frac {d\mathbf{r}}{\mathbf{r}}\right)^{\frac 12},
\end{equation*}
for all $\mathbf{g} \in \mathscr N $, where      $\varphi_{\mathbf{r} }(\mathbf{g}) $ is given by    \eqref{eq:varphi-r0}.

\begin{thm} \label{thm:g-function}  Suppose that  Poisson bounded function $\varphi^{
(\mu)}$
  on $ {\mathscr H}_\mu$ has mean value zero, $\mu=1,2,3$.  For    $f \in L
^p
(\mathscr N)$ ($1<p<\infty$), we have
$
      \|g_{\boldsymbol\varphi }(f)\|_p\lesssim \|f\|_p
$. If they are also w-invertible, then $
      \|g_{\boldsymbol\varphi }(f)\|_p\approx \|f\|_p
$.
\end{thm}

Recall that
a flag singular integral is a  convolution operator
with a distribution  kernel, which is exactly the    push-forward   distribution  of  a distributional convolution kernel  of bi-parameters  on $\mathscr  H\times
\mathbb{R}$ \cite{MRS}, by the projection  $\mathscr  H\times \mathbb{R}\rightarrow \mathscr  H$ given by $(z,t,u)\mapsto(z,t+u)$. Motivated by this, we show that for a
distributional    convolution kernel
$K$ of tri-parameters  on $ \tilde{{\mathscr N}} $, defined by the size
estimates and the cancellation conditions,
the   {\it push-forward   distribution} $ K^\flat$ on $ {\mathscr  N}$   exists (cf. Section 5.2), and  define  a    convolution
operator   on ${\mathscr N}$ by
\begin{equation}\label{eq:T-K}
   T_{K^\flat}(f)= f * K^\flat.
\end{equation}Following  the   flag case \cite{MRS},  we call $ K^\flat$   a {\it flag-like convolution kernel} and  $T_{K^\flat} $    a {\it flag-like
singular integral}.
 The Cauchy-Szeg\H o  kernel \eqref{eq:singular-kernel} on the group $\mathscr  N$ is a   flag-like convolution kernel  by Proposition \ref{prop:CS-flag-like}.

\begin{thm} \label{thm:Kb} For a  flag-like     convolution kernel $ {K^\flat}$ on $ {\mathscr N} $, the operator
$
   T_{K^\flat} $ is bounded from $  L
^p
(\mathscr N)$ to $  L
^p
(\mathscr N)$ for $1<p<\infty$.
\end{thm}

   For $\mathbf{g} \in \mathscr N$,
   define the {\it nontangential region}
  $
     \Gamma (\mathbf{g}):=\{(\mathbf{g}',\mathbf{r})\in \mathscr N \times\mathbb{R}_+^3;\mathbf{g}'\in T(\mathbf{g}, \mathbf{r}) \}
 $.
For $f \in L
^1
(\mathscr N)$, we define
  the {\it Lusin-Littlewood-Paley area function}
\begin{equation}\label{eq:area-function}
   S_{area, \boldsymbol \varphi}(f)(\mathbf{g})=\left( \int_{\Gamma (\mathbf{g})}|f* \varphi_{\mathbf{r} }(\mathbf{h}) |^2 \frac {d\mathbf{h}}{|
   T(\mathbf{0},\mathbf{r})|} \frac
   {d\mathbf{r}}{\mathbf{r}}\right)^{\frac 12}.
\end{equation}
The {\it square function Hardy space} $H_{area,\boldsymbol \varphi}^
1(\mathscr N)$
    is defined to be the
set  of all $f \in L
^1
(\mathscr N)$ such that $   S_{area, \boldsymbol \varphi}(f) \in L
^1
(\mathscr N)$  with the norm
$
   \|f\|_{H_{area,\boldsymbol\varphi}^
1(\mathscr N)}:=\|S_{area,\boldsymbol  \varphi}(f)\|_{L
^1
(\mathscr N)}
$.

By generalizing   Chen-Cowling-Lee-Li-Ottazzi's construction \cite{CCLLO}   for
flag singular integrals on the Heisenberg group, we use suitable products of fractal tiles   in   the Heisenberg subgroups $\mathscr  H_1$ and $\mathscr  H_2$ to
construct shards to obtain a partition  of
    $\mathscr  N$ for each scale. Shards   have the size comparable to    tubes   and play the role of dyadic rectangles in the product spaces. They can be used to
    define
atoms on $\mathscr  N$.

We say that $f \in L^1
(\mathscr N)$ has an {\it atomic decomposition} if we may write $f$ as a sum $\sum_{j\in \mathbb{N}} \lambda_j a_j$,
converging in $ L^1
(\mathscr N)$ with $\sum_{j\in \mathbb{N}} |\lambda_j|<\infty$  and each $a_j $ is an atom. The {\it atomic Hardy space} $ H^1_{atom}(\mathscr N)$   is defined to
be
the completion of the linear space of all $f \in L^1
(\mathscr N)$ that have atomic decompositions, with the norm
\begin{equation*}
  \|f\|_{H^1_{atom}(\mathscr N)}=\inf\left\{\sum_{j\in \mathbb{N}} |\lambda_j|; f\sim \sum_{j\in \mathbb{N}} \lambda_j a_j\right\}.
\end{equation*}
 \begin{thm} \label{thm:atomic-decomposition}
Suppose that $M  $ is positive integer
  and that $\varphi^{
(\mu)}$
on $ {\mathscr
  {H}}_\mu$
are Poisson bounded  and w-invertible   with w-inverses $\psi^{
(\mu)}$
 of the form $  \widetilde{\triangle}_\mu^M \dot{ \psi}^{
(\mu)}$
for some w-invertible Poisson bounded  $   \dot{ \psi}^{
(\mu)}$ with support in the unit ball in $ {\mathscr
  {H}}_\mu$, where $
    \widetilde{ \triangle}_{\mu}
$ is the   sub-Laplacian  on     $ { \mathscr H}_\mu$, $\mu=1,2,3$.  Then there is a constant $C$, depending on $\mathscr N$, $\varphi^{
(\mu)}$ and $\psi^{
(\mu)}$, such that for all $f\in H_{area, \boldsymbol\varphi}^
1(\mathscr N)$,
there exist numbers $\lambda_j$ and    atoms $ a_j  $ such that $f \sim \sum_j\lambda_j   a_j  $, and
\begin{equation*}
   \|f\|_{H_{area,\boldsymbol \varphi}^
1(\mathscr N)}\leq\sum_j|\lambda_j | \leq C\|f\|_{H_{area,\boldsymbol \varphi}^
1(\mathscr N)}.
\end{equation*}
\end{thm}

The paper is organized as follows. In Section 2, we discuss the lifting from the   group  $\mathscr N$ to the   group $\tilde{\mathscr N}$,  prove a general
transference theorem, and then  show   the coincidence of the  Cauchy-Szeg\H o kernel on  $\mathscr N$ with the push-forward of the  Cauchy-Szeg\H o kernel on
$\tilde{\mathscr N}$. In Section 3, we introduce the notions  of a tube, which is the natural ball  of tri-parameters on $\mathscr N$,  and the tube maximal function.
In Section 4, we show  the   commutativity of convolutions with the projection $\pi_*$ and use it to establish the Calder\'on
reproducing formula. The characterization of $ L
^p
(\mathscr N)$ by the Littlewood-Paley  function  in Theorem \ref{thm:g-function} is proved. In Section 5, we discuss a general class of  convolution kernels
 of tri-parameters  on $ \tilde{{\mathscr N}} $,  which are distributions satisfying the size
estimates and the cancellation conditions in terms of normalized bump functions. Our flag-like singular integrals are convolution operators with   kernels to be the
push-forward  of  convolution kernels
 of tri-parameters on $ \tilde{{\mathscr N}} $. Their boundedness on $ L
^p
(\mathscr N)$  is established. In Section 6, we use  tiles  in     the Heisenberg  groups to construct shards in the   group  $\mathscr N$, which   have  sizes comparable to tubes and constitute a    partition of $ {\mathscr N} $ for each scale. In Section 7, we introduce the notion
of the  tent of a  shard   and use it to give a    partition of   $\mathscr N \times\mathbb{R}_+^3$, based on which    atomic decomposition of  $H^1$ Hardy space
is established.
In the appendix, we collect basic facts about   the  Cauchy-Szeg\H o kernels on the  Siegel domains, and use general
Gindikin's formula to deduce the ones on domains \eqref{eq:U} and the product   of three Siegel upper half spaces.
 \section{Lifting  and transference   }
\subsection{The lifting group $\tilde{\mathscr N}$ }
We write a point of $ {\mathscr H}_\mu$ as    $\mathbf{g}_\mu := (\mathbf{z}_\mu, t_\mu )  $ with
$t_\mu\in \mathbb{R}$ and
    \begin{equation} \label{eq:z-x}
      \mathbf{z}_\mu=({z}_{\mu 1},\ldots,{z}_{\mu n_\mu})\in\mathbb{C}^{ n_\mu} ,\qquad {\rm where}\qquad {z}_{\mu j}=
   x_{\mu j}+\mathbf{i}x_{\mu (n_\mu+j)} ,
   \end{equation}
   $j=1,\ldots, n_\mu$. Since
${\rm Im} \langle \mathbf{z}_\mu,  \mathbf{z}_\mu
   '\rangle= \sum_{j=1}^{n_\mu}(x_{\mu(n_\mu+j)}x'_{\mu j}-x_{\mu j} x'_{\mu(n_\mu+j)} )$,
 \begin{equation} \label{eq:left-invariant}\begin{split}
    Y_{\mu j} =\frac \partial{\partial x_{\mu j}}+2  x_{\mu(n_\mu+j)}\frac \partial{\partial t_\mu},
    \qquad \qquad   Y_{\mu(n_\mu+j)} =\frac
   {\partial\hskip 8mm}{\partial
    x_{\mu(n_\mu+j)}}-2  x_{\mu j}\frac \partial{\partial t_\mu },
 \end{split}  \end{equation}$j=1,\ldots,n_\mu$,  are left invariant vector fields on the  Heisenberg group $ {\mathscr
  {H}}_\mu$ for  $\mu=1,2,3$. Then
 \begin{equation} \label{eq:brackets}
   \left [Y_{\mu j},Y_{\mu(n_\mu+j)}\right]=-4 \frac \partial{\partial t_\mu },
 \end{equation}
 and all other brackets vanish. The {\it sub-Laplacian} on  the Heisenberg  group  $ { \mathscr H}_\mu$ is
$
   \widetilde{ { \triangle}}_{\mu}:=  - \sum_{j=1}^{2n_\mu}   Y_{\mu j} ^2 .
$ The dilation of  $ { \mathscr H}_\mu$ is denoted as
$\delta_a(\mathbf{z}_\mu, t_\mu )=(a\mathbf{z}_\mu, a^2 t_\mu )$ for $a>0$. Sometimes we write it briefly as $ a(\mathbf{z}_\mu, t_\mu )  $.

We will use the following  {\it  pseudodistance}  on    the Heisenberg  group $ {\mathscr H}_\mu$:
 \begin{equation*}
    d_\infty(\mathbf{h}_\mu, \mathbf{g}_\mu )=\|\mathbf{h}_\mu
 ^{-1}
\mathbf{g}_\mu\| _\infty
 \end{equation*}
 where $\|\cdot\|_\infty$ is given by
$
   \|(\mathbf{z}_\mu ,t_\mu )\|_\infty=\max \{|  x_{\mu1} |,\ldots, | x_{\mu(2n_\mu)}|,  |t_\mu|^{\frac 12}  \}
$ for $\mathbf{z}_\mu$ given by \eqref{eq:z-x}.
Balls  on    the Heisenberg  group $ {\mathscr H}_\mu$  are $ {B}_\mu(\mathbf{g}_\mu,r_\mu):=\{ {\mathbf{h}}_\mu\in  {\mathscr H} _ \mu; \|
{\mathbf{h}}_\mu^{-1}\mathbf{g}_\mu\|_\infty <  {r}_\mu \}$ for some $\mathbf{g}_\mu\in {\mathscr H}_\mu$ and  $  {r}_\mu>0$. If   denote by $\Box_{a}^{ (\mu) }$
the cube
$
  (-a,a)^{ 2n_\mu
}:=\{\mathbf{z}_\mu\in\mathbb{C}^{n_\mu};  | {x}_{\mu j}| <a\}
$ in $\mathbb{C}^{n_\mu}$,
  and    by $I_b$  the interval $(-b,b)$, then
 \begin{equation*}
    {B}_\mu(\mathbf{0}_\mu,r_\mu)=\Box_{r_\mu}^{ (\mu) }\times I_{r_\mu^2}
 \end{equation*}
is a cuboid, where $\mathbf{0}_\mu$ is the origin of the Heisenberg group $ {\mathscr H} _ \mu$.
Let $ \tau_\mu: {\mathscr H}_\mu\rightarrow
{\mathscr  N}$        be the  (embedding) homomorphism   from the Heisenberg group       $ {\mathscr H}_\mu$  ($\mu=1,2,3$) into $\mathscr N$:
 \begin{equation}\label{eq:embedding} \begin{split}
       \tau_1(\mathbf{z}_1, t_1)&: =(\mathbf{z}_1,\mathbf{0}_2,\mathbf{0}_3, t_1,0),\\
         \tau_2(\mathbf{z}_2, t_2)&: =(\mathbf{0}_1,\mathbf{z}_2,\mathbf{0}_3,0, t_2 ),\\
           \tau_3(\mathbf{z}_3, t_3)&: =(\mathbf{0}_1,\mathbf{0}_2,\mathbf{z}_3, t_3,t_3).
 \end{split}  \end{equation}

 \begin{prop} \label{prop:homomorphism} (1) $\pi$ in \eqref{eq:pi} is a homomorphism of groups.
 \\
 (2) $ \tau_\mu({\mathscr H}_\mu)$ commutes with $ \tau_\nu({\mathscr H}_\nu)$ for $\mu\neq\nu$.
     \end{prop}
  \begin{proof} (1) It is direct to see $ \pi(\mathbf{z} , \mathbf{u} )\pi(\mathbf{z} ', \mathbf{u} ')  = \pi((\mathbf{z} , \mathbf{u} ) (\mathbf{z} ', \mathbf{u}
  '))
  $ by the multiplication  laws   of $\mathscr N$  and $\tilde{\mathscr N}$.

 (2) It is direct to check $ \tau_\mu(\mathbf{z}_\mu, t_\mu)\tau_\nu(\mathbf{z}_\nu, t_\nu)= \tau_\nu(\mathbf{z}_\nu, t_\nu) \tau_\mu(\mathbf{z}_\mu, t_\mu)$ for
 $\mu\neq\nu$ by the multiplication  law \eqref{eq:multiplication1} of $\mathscr N$ and the definition  of $\tau_\mu$'s in \eqref {eq:embedding}.
   \end{proof}
  $ \tau_1({\mathscr H}_1)$ is a normal subgroup  of $ {\mathscr  N}$ by
  \begin{equation*}
     (\mathbf{z}', \mathbf{t}') (\mathbf{z}_1,\mathbf{0}_2,\mathbf{0}_3, t_1,0)(\mathbf{z}', \mathbf{t}')^{-1}=(\mathbf{z}_1,\mathbf{0}_2,\mathbf{0}_3, t_1+2
     \Phi_1(
   \mathbf{z}_1 ',  \mathbf{z}_1
),0)\in \tau_1({\mathscr H}_1),
  \end{equation*} for any $(\mathbf{z}', \mathbf{t}')\in {\mathscr  N}$. Similarly, $ \tau_2({\mathscr H}_2)$ and $ \tau_3({\mathscr H}_3)$ are also normal subgroups of $ {\mathscr  N}$.
  Thus, $ {\mathscr V}_\alpha:=\mathbb{C}^{ n_\alpha+ n_3 }\times\mathbb{R}$, $\alpha=1,2$, are quotient groups:
 \begin{equation*}
    {\mathscr  N}/\tau_1({\mathscr H}_1)\cong {\mathscr V}_2,\qquad{\mathscr  N}/\tau_2({\mathscr H}_2)\cong {\mathscr V}_1,
 \end{equation*}which are isomorphic to  the Heisenberg groups with
  multiplications   given by
\begin{equation}\label{eq:multiplication2}
   (\mathbf{z}_\alpha,{\mathbf{z}}_3 , t_\alpha ) ( \mathbf{z}_\alpha',{\mathbf{z}}_3',t_\alpha ')= \Big( \mathbf{z}_\alpha +\mathbf{z}_\alpha', {\mathbf{z}}_3+
   {\mathbf{z}}_3',t_\alpha +t_\alpha '+  \Phi_\alpha(
   \mathbf{z}_\alpha,  \mathbf{z}_\alpha
   ')+  \Phi_3(
   \mathbf{z}_3,  \mathbf{z}_3
   ' )\Big).
\end{equation}

The embedding  $ \tau_\mu$ maps left invariant vector fields $  Y_{\mu j}$'s in \eqref{eq:left-invariant} on the group $ {\mathscr
  {H}}_\mu$ to left invariant vector fields $  X_{\mu j}$'s  on the group $  \mathscr
  N$, which are given by
\begin{equation} \label{eq:left-invariant2}\begin{split}
    X_{\mu j} &=\frac \partial{\partial x_{\mu j}}+2  x_{\mu(n_\mu+j)}\frac \partial{\partial t_\mu},
    \qquad \qquad \qquad  X_{\mu(n_\mu+j)} =\frac
   {\partial\hskip 8mm}{\partial
    x_{\mu(n_\mu+j)}}-2  x_{\mu j}\frac \partial{\partial t_\mu },\qquad \mu=1,2 ,\\
      X_{3 j} &=\frac \partial{\partial x_{3j}}+2  x_{3(n_3+j)}\left(\frac \partial{\partial t_1}+\frac \partial{\partial t_2}\right),
    \qquad X_{3(n_3+j)} =\frac
   {\partial\hskip 8mm}{\partial
    x_{3(n_3+j)}}-2  x_{3 j}\left(\frac \partial{\partial t_1}+\frac \partial{\partial t_2}\right).
 \end{split}  \end{equation}The   sub-Laplacian  on  the  group  $ { \mathscr H}_\mu$ induces an operator $
 { \triangle}_{\mu}:=  - \sum_{j=1}^{2n_\mu}   X_{\mu j} ^2
$ on
  the    group  $   \mathscr N$.

The normalized Haar measure  on  $\mathscr  N$
is the Lebesgue measure, which is denoted by $d\mathbf{g}$. The
  convolution  on  $\mathscr  N$ is defined as
  \begin{equation}\label{eq:convolution-def}
     f_1*f_2(\mathbf{g}):=\int_{\mathscr  N}f_1(\mathbf{h})f_2(\mathbf{h}^{-1}\mathbf{g})d\mathbf{h}=\int_{\mathscr
     N}f_1(\mathbf{g}\mathbf{h}^{-1})f_2(\mathbf{h})d\mathbf{h},
  \end{equation}
  by taking coordinates transformation $\mathbf{h}^{-1}\mathbf{g}\rightarrow \mathbf{h}$, which preserving the Haar measure. The same formulae hold  for  $
  {\mathscr
  H}_\mu$. For a function $f$ on a nilpotent group, denote   \begin{equation*}
    \breve{f} (\mathbf{g} ):=f (\mathbf{g}^{-1}).
 \end{equation*}

An  $  L^p $-function on $   {\mathscr N} $ can be lifted to  an $  L^p $-function on $  \tilde{\mathscr N} $ by the following lemma.

 \begin{lem} \label{lem:lift-function} For $f\in L^p(\mathscr N)$, define
 \begin{equation}\label{f-sharp}
    f^{\sharp}(\mathbf{z}, t_1,t_2,u): =2 \chi( t_1+t_2)f (\mathbf{z}, t_1+u,t_2+u)
 \end{equation}
    where $\chi$ is a non-negative smooth function supported on $[1/2, 1]\cup [-1,-1/2 ]$ such that $\chi(t)=\chi(-t)$ and
$\int_{
      \mathbb{R} }  \chi = 1$. Then $ f^{\sharp}\in L^p(\tilde{\mathscr N})$ and $\pi_*f^{\sharp}=f$. Moreover, $ f^{\sharp}\in C_c^\infty (\tilde{\mathscr N})$ if
      $ f
\in C_c^\infty ( {\mathscr N})$.
 \end{lem}
 \begin{proof} This is because
 \begin{equation*}\begin{split}
     (\pi_* f^{\sharp})(\mathbf{z},\mathbf{t})& =f (\mathbf{z}, t_1 ,t_2 )\int_{ \mathbb{R} } 2\chi( t_1+t_2-2u)du= f (\mathbf{z},\mathbf{t}),
\\
      \|f^{\sharp}\|^p_{L^p(\tilde{\mathscr N})}& =\int_{\mathbb{C}^{N}\times\mathbb{ R}^2} d\mathbf{z}d t_1 dt_2 \int_{
      \mathbb{R} }  |f (\mathbf{z}, t_1 ,t_2 )|^p 2^p \chi^p ( t_1+t_2-2u)  du= C_p\|f \|_{L^p( {\mathscr N})}^p,
 \end{split}\end{equation*}
 by definition and changing variables, where $C_p=\int_{
      \mathbb{R} } 2^{p-1} \chi^p(  u) du$.

Suppose that $f$   is compactly supported. If $ f^{\sharp}(\mathbf{z},\mathbf{t},u)\neq 0$, then by definition \eqref{f-sharp},  there exists $M>0$ such that
$|\mathbf{z}|<M$, $|t_1+u|<M$, $|t_2+u|<M$ and
 $|t_1+t_2|<1$. Thus, 
 $$|t_1-t_2|\leq |t_1+u| +|t_2+u|< 2M.$$
   Consequently, $|t_1 | ,  |t_2 |,
 |u|< M+1/ 2$. Namely, $ f^{\sharp}$ is compactly supported.
 \end{proof}

 \subsection{A general transference theorem} The
transference method  (cf. e.g. \cite{AC,CW,FWZ,GH,NS} and the references therein) works for our lifting.
 Let
$\tilde{ T}$ be a measurable function on
$\mathbb{C}^N\times\mathbb{C}^N\times\mathbb{R}^3$ with compact support.
Define an operator $\tilde{\mathcal{T}}$
 acting on functions on $ \mathbb{C}^N\times\mathbb{R}^3$ by
  \begin{equation*}
  \tilde{\mathcal{T}}(F)(\mathbf{z},\tilde{\mathbf{t}}) :=\int_{\mathbb{C}^N\times\mathbb{R}^3}\tilde {
    T}(\mathbf{z},\mathbf{w},\tilde{\mathbf{t}}-\tilde{\mathbf{s}} )F(\mathbf{w},\tilde{\mathbf{s} })d\mathbf{w}d\tilde{\mathbf{s}},
 \end{equation*} for $(\mathbf{z},\tilde{\mathbf{t}})\in \mathbb{C}^N \times\mathbb{R}^3$.
 Next define the projection of the kernel $\widetilde{T}$ by $\pi$
 as
  \begin{equation}\label{eq:transfer-T}
     {T} (\mathbf{z},\mathbf{w},\mathbf{t} ) :=\int_{\mathbb{R}} \tilde{{ T}}(\mathbf{z},\mathbf{w},t_1-u,t_2-u,u)du,
 \end{equation}  for $(\mathbf{z},\mathbf{w},\mathbf{t} )\in \mathbb{C}^N\times\mathbb{C}^N\times\mathbb{R}^2$.
Given  a $L^1$-function $f$ on
$ \mathbb{C}^N\times\mathbb{R}^2$, for $(\mathbf{z},\mathbf{t} )\in \mathbb{C}^N\times\mathbb{R}^2$,  define
  \begin{equation*}\begin{split}
   \mathcal{ T }(f)(\mathbf{z},\mathbf{t }) :&=\int_{\mathbb{C}^N\times\mathbb{R}^2} { T}(\mathbf{z}, \mathbf{w},\mathbf{t} -\mathbf{s}   )f(\mathbf{w},\mathbf{s
   }   )d\mathbf{w}d\mathbf{s}.
\end{split} \end{equation*}

 \begin{lem} For $f\in L^1(\mathbb{C}^N\times\mathbb{R}^2)$, we have
  $  \mathcal{T} (f)\circ\pi= \tilde{\mathcal{ T}} (f\circ\pi)$.
 \end{lem}
  \begin{proof} Note that
  \begin{equation*}\begin{split}\tilde{\mathcal{ T}} (f\circ\pi)(\mathbf{z} ,\tilde{\mathbf{t}} )&=\int_{\mathbb{C}^N\times\mathbb{R}^3} \tilde{{ T}}(\mathbf{z},
  \mathbf{w},\tilde{\mathbf{t}} -\tilde{\mathbf{s}}    )f( \mathbf{w},\tilde s_1 +\tilde s_3    ,\tilde s_2 +\tilde s_3  )d\mathbf{w}d\tilde{\mathbf{s}} \\
  &=\int_{\mathbb{C}^N\times\mathbb{R}^2}d\mathbf{w}d\mathbf{s} \int_{ \mathbb{R} }\tilde{{ T}}(\mathbf{z}, \mathbf{w},\tilde t_1+\tilde t_3 -s_1-u, \tilde t_2+\tilde t_3 -s_2-u, u  )f( \mathbf{w},s_1
  ,s_2
  )du,  \end{split} \end{equation*}
    by taking coordinates transformation $s_\alpha= \tilde s_\alpha +\tilde s_3   ,$ $\alpha=1,2 $, $u=\tilde t_3-\tilde s_3  $, which preserve  the Lebesgue
    measure. On the other hand,  we have
  \begin{equation*}\begin{split}\mathcal{T} (f)\circ\pi(\mathbf{z} ,\tilde {\mathbf{t}} )&=\int_{\mathbb{C}^N\times\mathbb{R}^2} { T}(\mathbf{z}, \mathbf{w},
  \tilde t_1+\tilde t_3 -   {s}_1
    ,\tilde t_2+\tilde t_3- {s}_2   )f(\mathbf{w},\mathbf{s }  )d\mathbf{w}d\mathbf{s}.
\end{split} \end{equation*}The identity follows.
   \end{proof}

 Let $ R_{\tilde{\mathbf{s}}}$ be the action of the   group $\mathbb{R}^3$  on $  L^p(\mathbb{C}^N\times\mathbb{R}^2)$ given by
 $R_{\tilde{\mathbf{s}}}(f)(\mathbf{w},\mathbf{t}  )=f(\pi(\mathbf{0}, \tilde{\mathbf{s}})^{-1}(\mathbf{w},\mathbf{t}  ))$, i.e.
\begin{equation}\label{eq:representation}   R_{\tilde{\mathbf{s}}}(f)(\mathbf{w},\mathbf{t}  ): =f (\mathbf{w},t_1- \tilde s_1- \tilde s_3,t_2- \tilde s_2- \tilde
s_3),
\end{equation}
for fixed $\tilde{\mathbf{s}} \in \mathbb{R}^3$. It is a {\it representation} of the   group $\mathbb{R}^3$  on $  L^p(\mathbb{C}^N\times\mathbb{R}^2)$, i.e.
\begin{equation}\label{eq:representation-def}
   R_{\tilde{\mathbf{s}}'}\circ R_{\tilde{\mathbf{s}}}(f)=R_{\tilde{\mathbf{s}}'+\tilde{\mathbf{s}}}(f),
\end{equation}
  for  any $\tilde{\mathbf{s}}, \tilde{\mathbf{s}}'\in \mathbb{R}^3$.  Obviously, we have
\begin{equation}\label{eq:Rf}
   \| R_{\tilde{\mathbf{s}}}(f)  \|_{L^p(\mathbb{C}^N\times\mathbb{R}^2)}=\|  f \|_{L^p(\mathbb{C}^N\times\mathbb{R}^2)}.
\end{equation}
We now have the following transference result generalizing Nagel-Stein \cite[Theorem 5.1.1]{NS},  which shows that
a priori $L^p$ bound for $\tilde{\mathcal{T}}$
 gives the same bound for $ \mathcal{T} $.
 \begin{thm}\label{thm:transference}
    Suppose $\tilde{ T}$ be a measurable function on
$\mathbb{C}^N\times\mathbb{C}^N\times\mathbb{R}^3$ with compact support, and there is a constant $A_p$ such that
    \begin{equation*}
       \| \tilde{\mathcal{T}}(F)\|_{L^p(\mathbb{C}^N\times\mathbb{R}^3)}\leq A_p\|  F \|_{L^p(\mathbb{C}^N\times\mathbb{R}^3)},
    \end{equation*}for  any $F\in L^p(\mathbb{C}^N\times\mathbb{R}^3)$.
Then the operator $\mathcal{T}$ satisfies for any $f\in L^p(\mathbb{C}^N\times\mathbb{R}^2)$,
  \begin{equation*}
       \|  {\mathcal{T}}(f)\|_{L^p(\mathbb{C}^N\times\mathbb{R}^2)}\leq A_p\|  f \|_{L^p(\mathbb{C}^N\times\mathbb{R}^2)}.
    \end{equation*}
 \end{thm}
 \begin{proof} Note that for $f \in L^p(\mathbb{C}^N\times\mathbb{R}^2)$,
     \begin{equation}\label{eq:T-Rs}\begin{split}
   \mathcal{ T }(f)(\mathbf{z}, {\mathbf{t}} )
    &=\int_{\mathbb{C}^N\times\mathbb{R}^2}\int_{\mathbb{R}} \tilde T (\mathbf{z},\mathbf{w},t_1- s_1 -u,t_2- s_2 -u,u)f(\mathbf{w},\mathbf{s}
    )d\mathbf{w}d\mathbf{s}  du
    \\&=
    \int_{\mathbb{C}^N\times\mathbb{R}^3}  \tilde{{ T}}(\mathbf{z},\mathbf{w}, \tilde s_1 , \tilde s_2
    , \tilde s_3)f(\mathbf{w},t_1- \tilde s_1- \tilde s_3,t_2- \tilde s_2- \tilde s_3)d\mathbf{w}d \tilde {\mathbf{s} }
    \\&=\int_{\mathbb{C}^N\times\mathbb{R}^3} \tilde{ T}(\mathbf{z},\mathbf{w}, \tilde{\mathbf{s}} )R_{\tilde{\mathbf{s}}}(f)(\mathbf{w},\mathbf{t}
    )d\mathbf{w}d\tilde{\mathbf{s}}
\end{split} \end{equation}by  using definition \eqref{eq:representation} and taking coordinates transformation $ \tilde s_\alpha=t_\alpha- s_\alpha -u$,  $\alpha=1,2 $, $
\tilde s_3 =u$, which preserve  the Lebesgue measure.
Moreover, for fixed $\tilde{\mathbf{u}}\in \mathbb{R}^3$,
 \begin{equation}\label{eq:Rf-2}\begin{split}
  R_{\tilde{\mathbf{u}} }( \mathcal{ T }(f)) (\mathbf{z},{\mathbf{t}})&= \mathcal{ T }(f)(\mathbf{z},t_1-\tilde u_1-\tilde u_3,t_2-\tilde u_2-\tilde u_3)
\\&=\int_{\mathbb{C}^N\times\mathbb{R}^3} \tilde{ T}(\mathbf{z},\mathbf{w}, \tilde{\mathbf{s}} )R_{\tilde{\mathbf{s}}}(f)(\mathbf{w},t_1-\tilde u_1-\tilde
u_3,t_2-\tilde u_2-\tilde u_3
)d\mathbf{w}d\tilde{\mathbf{s}} \\&=\int_{\mathbb{C}^N\times\mathbb{R}^3} \tilde{ T}(\mathbf{z},\mathbf{w}, \tilde{\mathbf{s}} )R_{\tilde{\mathbf{u}}
+\tilde{\mathbf{s}}}(f)(\mathbf{w},\mathbf{t} )d\mathbf{w}d\tilde{\mathbf{s}}  \\
    &=\int_{\mathbb{C}^N\times\mathbb{R}^3} \tilde{ T}(\mathbf{z},\mathbf{w}, \tilde{\mathbf{s}}- \tilde{\mathbf{u}
    })R_{\tilde{\mathbf{s}}}(f)(\mathbf{w},\mathbf{t} )d\mathbf{w}d\tilde{\mathbf{s}}
\end{split} \end{equation}
by using \eqref{eq:representation}, \eqref{eq:representation-def}, \eqref{eq:T-Rs} and taking transformation $\tilde{\mathbf{s}} \rightarrow
\tilde{\mathbf{s}}-\tilde{\mathbf{u}}$.

Let $E\subset \mathbb{R}^3$ be the
(compact) projection onto $\mathbb{R}^3$ of the compact support of the function
$ \tilde{T}$.  Given $\epsilon > 0$,   choose a (large) bounded
open set $V\subset \mathbb{R}^3$ so that
$
 \frac {|V+E|}{|V|}<1+\epsilon.
$
Then  we get
\begin{equation*}\begin{split}
     \|  {\mathcal{T}}(f)\|^p_{L^p(\mathbb{C}^N\times\mathbb{R}^2)}
    &=\frac 1{|V|}\int_V \|  R_{\tilde{\mathbf{u}} }( \mathcal{ T }(f))\|^p_{L^p(\mathbb{C}^N\times\mathbb{R}^2)}d\tilde{\mathbf{u}}\\
     &=\frac 1{|V|}\int_V d\tilde{\mathbf{u}}\int_{\mathbb{C}^N\times\mathbb{R}^2}  |R_{\tilde{\mathbf{u} }}( \mathcal{ T }(f))(\mathbf{z}, \mathbf{t} )|^p
    d\mathbf{z}d\mathbf{t} \\
    &=\frac 1{|V|}\int_V d\tilde{\mathbf{u}} \int_{\mathbb{C}^N\times\mathbb{R}^2} \left | \int_{\mathbb{C}^N\times\mathbb{R}^3} \tilde{
    T}(\mathbf{z},\mathbf{w}, \tilde{\mathbf{s}}-\tilde{ \mathbf{u}} )R_{\tilde{\mathbf{s}}}(f)(\mathbf{w},\mathbf{t}
    )d\mathbf{w}d\tilde{\mathbf{s}}\right|^p  d\mathbf{z}d\mathbf{t}\\ &\leq \frac 1{|V|}\int_{ \mathbb{R}^2}  d\mathbf{t}\int_{\mathbb{C}^N\times\mathbb{R}^3}
    \left | \int_{\mathbb{C}^N\times\mathbb{R}^3} \tilde{
    T}(\mathbf{z},\mathbf{w}, \tilde{\mathbf{s}}- \tilde{\mathbf{u} })R_{\tilde{\mathbf{s}}}(f)(\mathbf{w},\mathbf{t}
    )\chi_{V+E}(\tilde{\mathbf{s}})d\mathbf{w}d\tilde{\mathbf{s}}\right|^p  d\mathbf{z}d\tilde{\mathbf{u}},\end{split} \end{equation*}by using \eqref{eq:Rf} \eqref{eq:Rf-2} and the kernel $\tilde{
    T}(\mathbf{z},\mathbf{w}, \tilde{\mathbf{s}}-\tilde{ \mathbf{u}} )$
vanishing unless $\tilde{\mathbf{s}} \in V+E$.
 If denote
$$
   \tilde{F}_{\mathbf{t}}(\mathbf{w}, \tilde{\mathbf{s}}):=R_{-\tilde{\mathbf{s}}}(f)(\mathbf{w},\mathbf{t} )\chi_{V+E}(-\tilde{\mathbf{s}})
$$
 and take the transformation $\tilde{\mathbf{s}} \rightarrow -\tilde{\mathbf{s}} $ in the last integral, we get \begin{equation*}\begin{split}
   \|  {\mathcal{T}}(f)\|^p_{L^p(\mathbb{C}^N\times\mathbb{R}^2)}  &\leq\frac 1{|V|}\int_{ \mathbb{R}^2}  d\mathbf{t}\int_{\mathbb{C}^N\times\mathbb{R}^3} \left |
   \tilde{\mathcal{T}}(\tilde{F}_{\mathbf{t}})(\mathbf{z}, -
    \tilde{\mathbf{u} })\right|^p  d\mathbf{z}d\tilde{\mathbf{u}}\\
    &=\frac 1{|V|}\int_{ \mathbb{R}^2}\left \|
    \tilde{\mathcal{T}}(\tilde{F}_{\mathbf{t}})\right\|^p_{L^p(\mathbb{C}^N\times\mathbb{R}^3)}d\mathbf{t}\\
    & \leq
    \frac  {A_p^p}{|V|}\int_{ \mathbb{R}^2}\left \|  \tilde{F}_{\mathbf{t}} \right\|^p_{L^p(\mathbb{C}^N\times\mathbb{R}^3)}d\mathbf{t},
\end{split} \end{equation*}
by the  assumption of the theorem.
But
\begin{equation*}\begin{split}  \int_{ \mathbb{R}^2} \left\|  \tilde{F}_{\mathbf{t}} \right\|^p_{L^p(\mathbb{C}^N\times\mathbb{R}^3)}d\mathbf{t}&=
\int_{ \mathbb{R}^3}d\tilde{\mathbf{s}}\int_{\mathbb{C}^N\times\mathbb{R}^2}  \left|R_{-\tilde{\mathbf{s}}}(f)(\mathbf{w},\mathbf{t}
)\chi_{V+E}(-\tilde{\mathbf{s}})\right|^pd\mathbf{w}d\mathbf{t}
 =|V+E|\| f \|^p_{L^p(\mathbb{C}^N\times\mathbb{R}^2)},
\end{split}\end{equation*}by using \eqref{eq:Rf} again. It follows that
 \begin{equation*}\begin{split}
     \|  {\mathcal{T}}(f)\|^p_{L^p(\mathbb{C}^N\times\mathbb{R}^2)}
    & \leq{A_p^p}(1+\epsilon)\| f \|^p_{L^p(\mathbb{C}^N\times\mathbb{R}^2)},
\end{split} \end{equation*}
which completes the proof by taking $\epsilon\rightarrow 0$.
 \end{proof}
  \subsection{Transference of   Cauchy-Szeg\H o kernels} \label{Transference-CS} As in Corollary \ref{cor:flat-Szego1} in the appendix, under the identification $\tilde\iota: \mathscr U \rightarrow \tilde{{\mathcal U} }$ in \eqref{eq:tilde-ota}, the  Cauchy-Szeg\H o
  kernel $\tilde{S}_{\boldsymbol \varepsilon}$ on the domain $   \tilde{ \mathscr N }\times \mathbb{R}_+^3$   can   written as
\begin{equation}\label{eq:flat-Szego-tilde}
   \tilde{S}_{\boldsymbol \varepsilon}(  \mathbf{g}) =\prod_{\mu=1}^3
   \widetilde{{S}}_\mu (\varepsilon_\mu,  \mathbf{g}_\mu ),
\end{equation} for $\mathbf{g}=(\mathbf{g}_1,\mathbf{g}_2, \mathbf{g}_3)  \in \mathscr
H_1\times\mathscr H_2 \times\times\mathscr H_3$ and $\boldsymbol \varepsilon=(  \varepsilon_1,  \varepsilon_2,   \varepsilon_3)\in
\mathbb{R}^3_+ $,
where
\begin{equation}\label{eq:flat-Szego2}
   \widetilde{{S}}_\mu (\varepsilon_\mu, \mathbf{g}_\mu)=
 \frac {c_\mu}{ (|{\mathbf z}_\mu |^2 +\varepsilon_\mu- \mathbf{i} t_\mu ) ^{n_\mu+1}    },\qquad c_\mu:=\frac { n_\mu! }{ 4(\frac \pi
 2)^{n_\mu+1}},
\end{equation}is the
Cauchy-Szeg\H o kernel   on the domain  $ \mathscr H_\mu \times \mathbb{R}_+= \mathbb{C}^{n_\mu}\times \mathbb{R } \times \mathbb{R}_+$, if we write
$\mathbf{g}_\mu =( \mathbf{z}_\mu , t_\mu )\in \mathscr H_\mu$.  Its  push-forward  is
 \begin{equation*}\begin{split}
 \pi_*(  \tilde{ S}_{\boldsymbol \varepsilon})(\mathbf{z},\mathbf{t} ) &=\int_{\mathbb{R}} \tilde{S}_{\boldsymbol \varepsilon}(\mathbf{z},  t_1-u,t_2-u,u)du\\
    &=\int_{\mathbb{R}}  \frac {c_1}{ (|{\mathbf z}_1 |^2+\varepsilon_1 - \mathbf{i} (t_1 -u)) ^{n_1+1}    }  \frac {c_2}{ (|{\mathbf z}_2 |^2+\varepsilon_2 -
    \mathbf{i} (t_2-u) ) ^{n_2+1}    } \frac {c_3}{ (|{\mathbf z}_3 |^2+\varepsilon_3 - \mathbf{i} u ) ^{n_3+1}    }du\\ &=\frac
    1{2\pi\mathbf{i}}\int_{\mathbf{i}\mathbb{R}}  \frac {f(\zeta) }{ (\zeta -|{\mathbf z}_3 |^2-\varepsilon_3 ) ^{n_3+1}    }d\zeta
    \end{split} \end{equation*}
    by changing the integral to a contour integral ($\mathbf{i} u\rightarrow \zeta$),
    where
    \begin{equation*}
      f(\zeta)=2\pi c_1 c_2 c_3(-1)^{n_3+1}\frac { 1}{ (|{\mathbf z}_1 |^2 +\varepsilon_1- \mathbf{i}  t_1 +\zeta) ^{n_1+1}    }  \frac {1}{ (|{\mathbf z}_2
      |^2+\varepsilon_2 - \mathbf{i}  t_2+ \zeta ) ^{n_2+1}    }
    \end{equation*}is holomorphic on the right half plane $\{\zeta\in \mathbb{C};{\rm Re}\, \zeta>-\min(\varepsilon_1,\varepsilon_2)\}$. Now consider the
    semicircle
    $\Gamma_R:= \{R e^{i\theta}; \theta\in(-\pi/2,\pi/2)\}$. We see that
    \begin{equation*}
       \frac 1{2\pi\mathbf{i}}\int_{\mathbf{i}\mathbb{R}}  \frac {f(\zeta) }{ (\zeta -|{\mathbf z}_3 |^2-\varepsilon_3 ) ^{n_3+1}    }d\zeta =-\frac
       1{2\pi\mathbf{i}}
       \lim_{R\rightarrow +\infty }\int_{[-\mathbf{i}R,\mathbf{i}R]\cup \Gamma_R }  \frac {f(\zeta) }{ (\zeta -|{\mathbf z}_3 |^2-\varepsilon_3 ) ^{n_3+1}
       }d\zeta
    \end{equation*}
   by the integral over the semicircle $   \Gamma_R  $ tending to zero by simple estimate, where the contour integral in the right hand side is counter clockwise. Now apply  the residue theorem to this integral to get
     \begin{equation*}\begin{split}  \pi_*( & \tilde{ S}_{\boldsymbol \varepsilon})(\mathbf{z},\mathbf{t} )=-{\rm Res}\, \left(\frac {f(\zeta) }{ (\zeta -|{\mathbf
     z}_3 |^2-\varepsilon_3 ) ^{n_3+1}    },
     |{\mathbf z}_3
     |^2+\varepsilon_3\right)\\&
    =2\pi c_1 c_2 c_3\frac {(-1)^{n_3 }}{n_3!}\left.\left[  \frac {1}{ (|{\mathbf z}_1 |^2 +\varepsilon_1- \mathbf{i}  t_1 +\zeta) ^{n_1+1}    }  \frac {1}{
    (|{\mathbf z}_2 |^2+\varepsilon_2 - \mathbf{i}  t_2+\zeta ) ^{n_2+1}    }\right]^{(n_3)}\right|_{\zeta=|{\mathbf z}_3 |^2+\varepsilon_3}\\  &=2\pi c_1 c_2 c_3
    \frac
    {(-1)^{n_3 }}{n_3!} \left.\sum_{k=0}^{n_3}\binom {n_3}{k}
    \frac {(-n_1-1)\cdots(-n_1-k)}{ (|{\mathbf z}_1 |^2 +\varepsilon_1- \mathbf{i}  t_1 +\zeta) ^{n_1+k+1}    }  \cdot  \frac {(-n_2-1)\cdots(-n_2-n_3+k)}{
    (|{\mathbf
    z}_2 |^2+\varepsilon_2 - \mathbf{i}  t_2+\zeta ) ^{n_2+ n_3-k+1}    } \right|_{\zeta=|{\mathbf z}_3 |^2+\varepsilon_3}\\
    &= \frac {1}{ 4^2(\frac \pi
 2)^{N+2}} \sum_{k=0}^{n_3}\binom {n_3}{k}
      \frac {( n_1+k)!}{ (|{\mathbf z}_1 |^2+|{\mathbf z}_3 |^2+\varepsilon_1+\varepsilon_3 - \mathbf{i}  t_1  ) ^{n_1+k+1}    }  \cdot \frac {( n_2+n_3-k)!}{
      (|{\mathbf z}_2 |^2+|{\mathbf z}_3 |^2+\varepsilon_2+\varepsilon_3 - \mathbf{i}  t_2  ) ^{n_2+ n_3-k+1}    }\\
      &=  S_{\varepsilon_1+\varepsilon_3,\varepsilon_2+\varepsilon_3 }(\mathbf{z},\mathbf{t} )
\end{split} \end{equation*}by the expression of $S_{\varepsilon_1 ,\varepsilon_2 }(\mathbf{z},\mathbf{t} )$ in  Corollary \ref{cor:flat-Szego1}.
\section{Tubes and   tube   maximal function  }
 \subsection{Tubes }  Since
  $\tilde{\mathscr  N}$ has three commutative dilations, a natural ball
  is $\tilde{\mathbf{g}} \tilde{{B}} (\mathbf{0} ,\mathbf{r})$, where
  \begin{equation}\label{eq:product-balls}
  \tilde{{B}} (\mathbf{0} ,\mathbf{r}):={B}_1(\mathbf{0}_1,r_1)\times  {B}_2(\mathbf{0}_2,r_2)\times  {B}_3(\mathbf{0}_3, {r}_3),
  \end{equation}  for $\mathbf{r} :=\left({r_1} ,
          {r_2} ,r_3 \right)\in \mathbb{  R}^3_+ $,  is the product of three balls on  the Heisenberg groups $ {\mathscr H}_\mu$'s, respectively. A natural ball in
          the group $\mathscr  N$ is the image ${\mathbf{g}} \pi(\tilde{{B}} (\mathbf{0} ,\mathbf{r}))$
    under the projection   $\pi$ by $\pi$   homomorphic, where $\mathbf{g}=\pi(\tilde{\mathbf{g}})$.
  To see the shape of this image,   we   write
 $
     \pi=\left({\rm id}_{\mathbb{C}^{ N}}, \widehat{\pi}\right),
$
 where ${\rm id}_{\mathbb{C}^{ N}}$  is the  identity transformation of $ \mathbb{C}^{ N}$,    and  $  \widehat{\pi}$ is  the projection $ \widehat{\pi}:
 \mathbb{R}^3 \longrightarrow\mathbb{R}^2$  given by
  \begin{equation}\label{eq:hat-pi}\begin{split}
  (u_1 ,u_2,u_3)&\mapsto (u_1+u_3,u_2+u_3).
  \end{split}  \end{equation}
Note that \begin{equation*}
    {B}_1(\mathbf{0}_1,r_1)\times   {B}_2(\mathbf{0}_2,r_2)  \times  {B}_3(\mathbf{0}_3,r_3)                  =\Box_{\mathbf{r}}\times I_{\mathbf{r}^2}
 \end{equation*}is a cuboid, if we write
      $$
          \Box_{\mathbf{r}}:=\Box_{r_1}^{(1)}\times \Box_{r_2}^{(2)}\times \Box_{r_3}^{(3)} , \quad I_{\mathbf{r}^2}:=I _{r_1^2}\times I
    _{r_2^2}\times I _{r_3^2}.
      $$

Now suppose that     $r_1> r_2$.
 It is direct  to see that the image of $I_{\mathbf{r}^2}$ under $ \widehat{\pi}$ is a
 hexagon in the following  Figure  1(a),
 \begin{figure}[ht]
  \centering
  \subfigure[$\widehat{\pi}(I_{\mathbf{r}^2})$]{
    \includegraphics[width=0.48\textwidth]{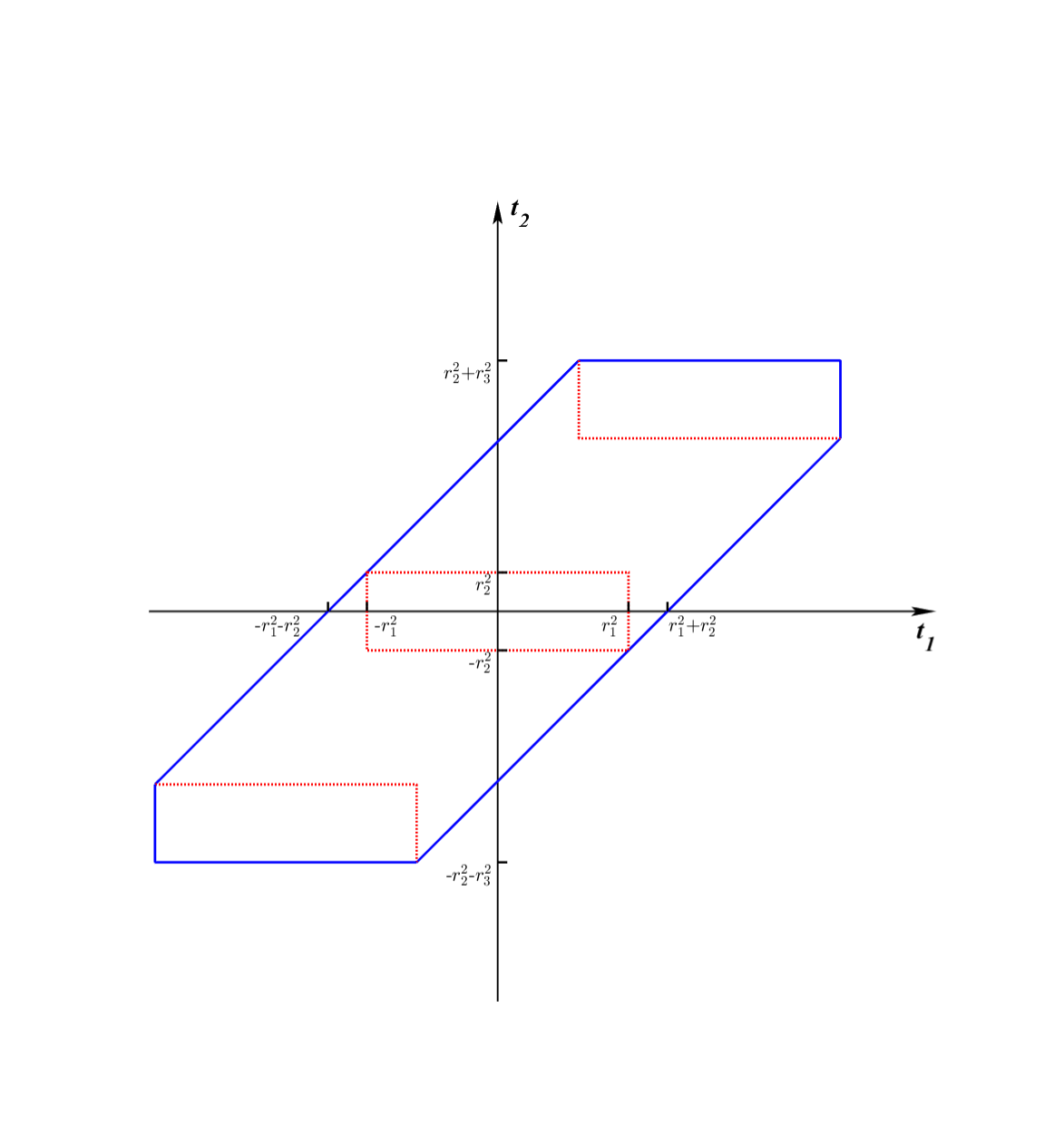}
    \label{fig:subfig1}
  }
  \subfigure[$P_{a,b }$]{
    \includegraphics[width=0.48\textwidth]{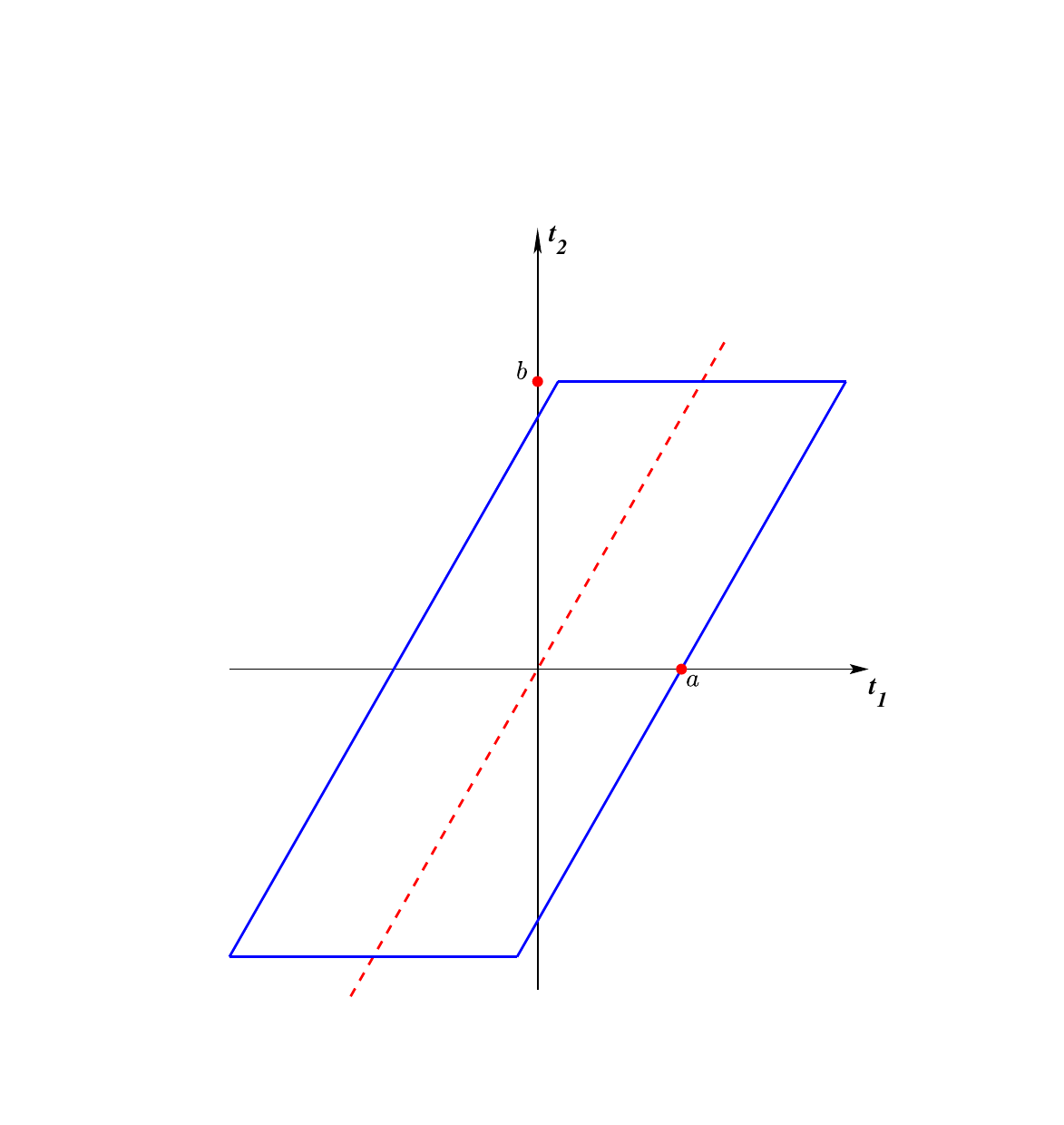}
    \label{fig:subfig2}
  }
  \caption{}
  \label{fig:both}
\end{figure} which is basically the {\it parallelogram}:
       \begin{equation*}
          P_{a,b }:=\{(t_1,t_2)\in \mathbb{R}^2; -a< t_1-t_2<a, |t_2|<b\},
       \end{equation*} with $a=r_1 ^{2}+r_2 ^{2}, b =r_2 ^{2}+r_3^{2}$.
 So  (1) if $r_3> r_2$, we define
$
     T(\mathbf{0},\mathbf{r}): = \Box_{\mathbf{r}} \times  {P}_{r_1 ^{2} ,  r_3^{2} }
    $;
            (2) if  $r_3\leq r_2$, we simply define
     $
     T(\mathbf{0},\mathbf{r}): = \Box_{\mathbf{r}}\times {P}_{r_1 ^{2} ,  r_2^{2} }.
  $ Namely, in both cases, we define
        \begin{equation*}\label{eq:tube-cube}
     T(\mathbf{0},\mathbf{r}):= \Box_{\mathbf{r}}\times {P}_{r_1 ^{2} ,  r_2^{2}\vee  r_3^{2}}.
       \end{equation*} While for $r_1\leq r_2$, define
$
     T(\mathbf{0},\mathbf{r})= \Box_{\mathbf{r}}\times {P}_{r_1 ^{2}\vee  r_3^{2} ,  r_2^{2}}.
 $ For $\mathbf{g}\in \mathscr  N$, set $ T(\mathbf{g},\mathbf{r}):= \mathbf{g}T(\mathbf{0},\mathbf{r})$.
    Following \cite{CCLLO}, we call $ T(\mathbf{g},\mathbf{r})$ a {\it tube}. For $r_1> r_2$, its volume is
  \begin{equation}\label{eq:tube-volume}
    | T(\mathbf{g},\mathbf{r})|= 2 ^{Q-2}  r_1 ^{2n_1}  r_2 ^{2n_2}  r_3 ^{2n_3} r_1^2  \max\{ r_2^2,  r_3^2\},
  \end{equation}where $Q =2N+4$ is the {\it homogeneous dimension} of $ { \mathscr
  N}$. It is similar for $r_1 \leq r_2$.
  \begin{prop}\label{prop:tube-pi}
$ T(\mathbf{g},\mathbf{r}/2)\subset    \pi  (\tilde{{B}} (\mathbf{g} ,\mathbf{r}) )\subset T(\mathbf{g},2\mathbf{r})$  for $ \mathbf{r}  \in \mathbb{
R}^3_+$.
  \end{prop}
   \begin{proof} It is sufficient to prove the case for $\mathbf{g}=\mathbf{0}$.  Without loss of generality, assume $r_1> r_2$. It is obvious that
   \begin{equation*}
      \pi (\tilde{{B}} (\mathbf{0} ,\mathbf{r}))\subset \Box_{\mathbf{r}} \times  {P}_{r_1 ^{2}+r_2 ^{2} ,r_2 ^{2} +r_3^{2} }\subset \Box_{\mathbf{r}} \times
      {P}_{2r_1 ^{2}  ,2(r_2 ^{2}\vee r_3^{2}) } \subset T(\mathbf{0},2\mathbf{r}),
   \end{equation*}
by definition.
  Then the result follows from the following inclusions
   \begin{equation*}
      \pi (\tilde{{B}} (\mathbf{0} ,\mathbf{r}))\supset \Box_{\mathbf{r}} \times  {P}_{r_1 ^{2}/4  ,(r_2 ^{2}  +r_3^{2})/4 } \supset T(\mathbf{0},\mathbf{r}/2).
   \end{equation*}The second inclusion is obvious by definition again. For
  the first one, note that for $(t_1,t_2)\in  {P}_{r_1 ^{2}/4  ,(r_2 ^{2}  +r_3^{2})/4 } $, we have $|t_1-t_2|<r_1 ^{2}/4 $, $| t_2|< (r_2 ^{2}  +r_3^{2})/4$.
  \\
   {\it
  Case i:    $| t_2|>  r_3^{2}  /4$}. We take $s_3={\rm sgn}\, t_2 \cdot r_3^{2} /4$, $s_1=t_1-s_3$ and $s_2=t_2-s_3$. Then $\widehat{\pi}(s_1,s_2,s_3)=(t_1,t_2)$ satisfying
   $|s_1-s_2|=|t_1-t_2|<r_1 ^{2}/4 $, $|s_2|=|t_2-s_3|\leq r_2 ^{2}/4  $, and   $|s_1 |\leq|s_1-s_2|+|s_2 | <r_1 ^{2}/2 $. Thus for $\mathbf{z}\in \Box_{\mathbf{r}}$, we have $(\mathbf{z},s_1,s_2,s_3)\in    \tilde{{B}}
   (\mathbf{0} ,\mathbf{r})$, and so $(\mathbf{z}, t_1,t_2)\in  \pi(\tilde{{B}} (\mathbf{g} ,\mathbf{r}) )$.
   \\{\it Case ii:    $| t_2|<  r_3^{2}  /4 $}. We take $s_3=t_2$,
   $s_2=0$ and $ s_1=t_1-t_2 $. Then, $|s_1|<r_1 ^{2}/4 $, and so $(\mathbf{z},s_1,s_2,s_3)\in    \tilde{{B}}
   (\mathbf{0} ,\mathbf{r})$ for $\mathbf{z}\in \Box_{\mathbf{r}}$. Hence, $(\mathbf{z},t_1,t_2)=\pi(\mathbf{z},s_1,s_2,s_3)\in  \pi (\tilde{{B}} (\mathbf{g} ,\mathbf{r}) )$.
      \end{proof}
 \subsection{Tube maximal function}
Recall that the {\it normalised dilate  of a function } $\varphi^{
(\mu)}$
on $ {\mathscr
  {H}}_\mu$ is
  \begin{equation}\label{eq:normalised}
     \varphi_{r_\mu}^{
(\mu)}(\mathbf{g}_\mu ):=r_\mu^{-Q_\mu}\varphi_{}^{
(\mu)}(r_\mu^{-1}\mathbf{g}_\mu ),
  \end{equation}  where $Q_\mu=2n_\mu+2$ is the {\it homogeneous dimension} of ${\mathscr
  {H}}_\mu$,
and the  {\it normalised characteristic function} is
 \begin{equation*}
    \chi_{r_\mu}^{
(\mu)}(\mathbf{g}_\mu ):=\frac 1{\left| {B}_\mu(\mathbf{0}_\mu,r_\mu)\right|}\mathbbm 1_{
 {B}_\mu(\mathbf{0}_\mu,r_\mu)}(\mathbf{g}_\mu ),
  \end{equation*}
where $\mathbbm{1}_E$ denotes the characteristic function of a set $E$.
For $\mathbf{r}\in \mathbb{R}^3_+$, let
    $ \varphi_{\mathbf{r} }$ and $\chi_{\mathbf{r} }$ be given by \eqref{eq:varphi-r0}.
  Define  the {\it iterated maximal operator}  by
 \begin{equation*}
    M_{it}  (f)(\mathbf{g})=\sup_{\mathbf{r}\in \mathbb{R}^3_+} \left |f* \chi_{\mathbf{r} }(\mathbf{g} )\right|.
 \end{equation*}

  The following proposition implies that  $\chi_{\mathbf{r} } $ is essentially the normalised characteristic function of the tube
 $  T(\mathbf{0},\mathbf{r})$.

  \begin{prop} \label{prop:tube} For $ f \in L^1(\mathscr N )$ and $\mathbf{r}  \in \mathbb{  R}^3_+ $, we have
\begin{equation*}
   \frac 19  |f|* \chi_{\frac 13\mathbf{r} }(\mathbf{g} ) \leq \frac 1{| T(\mathbf{g},\mathbf{r})|}\int_{  T(\mathbf{g},\mathbf{r})}|f (\mathbf{h})|d\mathbf{h}\leq
     |f|* \chi_{3 \mathbf{r} }(\mathbf{g} ).
\end{equation*}
Consequently, $M   (f) \approx M_{it}  (f)$.
    \end{prop}
 \begin{proof} Without loss of generality, we can  assume  $r_1> r_2$. Note that
\begin{equation*}\begin{split}\frac 1{| T(\mathbf{g},\mathbf{r})|}\int_{  T(\mathbf{g},\mathbf{r})}|f (\mathbf{h})|d\mathbf{h}&=\int_{\mathbb C^{N}\times
\mathbb{R}^2}|f ((\mathbf{z},\mathbf{t})   (\mathbf{z}',\mathbf{t}')) |v_{\mathbf{r} } (\mathbf{z}',\mathbf{t}')d  \mathbf{z}'d \mathbf{t}',
    \end{split}  \end{equation*}
    where $\mathbf{g}=(\mathbf{z},\mathbf{t})\in\mathscr N  $, and
    \begin{equation*}
       v_{\mathbf{r} } (\mathbf{z}',\mathbf{t}')=\frac 1{|\Box_{\mathbf{r}}|  }\mathbbm 1_{\Box_{\mathbf{r}}}(\mathbf{z}' ) \frac 1{  |P_{r_1^2 , r_2^2\vee
       r_3^2} |}   \mathbbm
       1_{P_{r_1^2 , r_2^2\vee r_3^2}}( \mathbf{t}').
    \end{equation*}
On the other hand, we have
  \begin{equation*}\label{eq:mean-tube}\begin{split}|f|* \chi_{\mathbf{r} }(\mathbf{z},\mathbf{t}) &=   \int_{\mathbb C^{N}\times \mathbb{R}^2}\left|f
  \left((\mathbf{z},\mathbf{t})   (\mathbf{z}',\mathbf{t}') ^{-1}\right)\right|  \chi_{\mathbf{r} } (\mathbf{z}',\mathbf{t}')d  \mathbf{z}'d \mathbf{t}' =
  \int_{\mathbb C^{N}\times \mathbb{R}^2}|f ((\mathbf{z},\mathbf{t})   (\mathbf{z}',\mathbf{t}')) |w_{\mathbf{r} } (\mathbf{z}',\mathbf{t}')d
  \mathbf{z}'d \mathbf{t}'
    \end{split}  \end{equation*}by \eqref{eq:convolution-def} and taking transformation $(-\mathbf{z}',-\mathbf{t}')\rightarrow (\mathbf{z}',\mathbf{t}')$,
where
    \begin{equation*}
      w_{\mathbf{r} } (\mathbf{z}' ,\mathbf{t}' )=\frac 1{|\Box_{\mathbf{r}}|  } \mathbbm 1_{\Box_{\mathbf{r}}} (\mathbf{z} '  )
\cdot\frac {W_{\mathbf{r} }(\mathbf{t}')}{ 2^3r_1^2r_2^2 r_3^2}
    \end{equation*}
with
    \begin{equation}\label{eq:w-r}
    W_{\mathbf{r} }(\mathbf{t}'):=     \int_{\mathbb{R}}\mathbbm  1_{(-r_1^2,r_1^2)}( t_1'-u)\mathbbm  1_{(-r_2^2,r_2^2)}(t_2'-u) \mathbbm
    1_{(-r_3^2,r_3^2)}(  u)du.
    \end{equation}It is direct to see that $W_{\mathbf{r} }( -\mathbf t')=W_{\mathbf{r} }( \mathbf{t}')$.
    We claim that  $W_{\mathbf{r} }$ satisfies the estimate
    \begin{equation}\label{eq:w-claim}
        \frac 29 \min\{r_2^2,r_3^2\}\mathbbm   1_{P_{ (r_1^2+r_2^2 )/9,   (r_2^2+r_3^2 ) /9}} (\mathbf{t}')   \leq     W_{\mathbf{r} }(\mathbf{t}')\leq
         2\min\{r_2^2,r_3^2\}\mathbbm  1_{P_{r_1^2+r_2^2, r_2^2+r_3^2}}(\mathbf{t}'),
    \end{equation}for any $\mathbf{t}'\in \mathbb{R}^2$.
    Then it follows from the claim and
  \begin{equation*}
 \frac {    \min\{r_2^2,r_3^2\}}{4r_1^2 r_2^2 r_3^2}= \frac 1{4 r_1^2  \max\{r_2^2,r_3^2\} }=\frac 1{  |P_{r_1^2 , r_2^2\vee r_3^2} |}
  \end{equation*} that
$
   \frac 19    v_{\mathbf{r} /3} (\mathbf{z}' ,\mathbf{t} ') \leq  w_{\mathbf{r} } (\mathbf{z}' ,\mathbf{t}' )\leq    v_{3\mathbf{r} } (\mathbf{z} ',\mathbf{t}' )
$.
  The  result follows.

To show the claim  \eqref{eq:w-claim}, note that if $W_{\mathbf{r} }(\mathbf{t}')\neq0$,  we must have $|t_1'-t_2'|<r_1^2+r_2^2 $. Otherwise,   for any $u$
satisfying
$|t_1'-u|< r_1^2$, we  must have  $|t_2'-u|\geq r_2^2$, which implies $W_{\mathbf{r} }(\mathbf{t}') =0$ by definition \eqref{eq:w-r}. It is a contradiction. By the same
argument, we get $|
t_2'|<r_2^2+r_3^2 $ if $W_{\mathbf{r} }(\mathbf{t}')\neq0$. Thus
supp $ W_{\mathbf{r} } \subset P_{r_1^2+r_2^2, r_2^2+r_3^2} $. Moreover, if  $r_3< r_2$,
then
   \begin{equation*}
      W_{\mathbf{r} }(\mathbf{t}')\leq  \int_{\mathbb{R}}\mathbbm 1_{(-r_3^2,r_3^2)}(  u)du =2r_3^2;
   \end{equation*}
 while if  $r_3\geq r_2$, then
     \begin{equation*}
       W_{\mathbf{r} }(\mathbf{t}')\leq  \int_{\mathbb{R}} 1_{(-r_2^2,r_2^2)}(t_2'-u) du =2r_2^2.
     \end{equation*}
     So the second inequality  in \eqref{eq:w-claim} holds.

  For the first inequality in  the claim  \eqref{eq:w-claim},  let $\mathbf{ t }'\in P_{(r_1^2+r_2^2)/9, (r_2^2+r_3^2)/9}$, i.e. $|t_1'-t_2'|<(r_1^2+r_2^2)/9,|t_2'|<(r_2^2+r_3^2)/9$.
   \\{\it Case i. $r_3 \leq r_2  $.}
For $u$ satisfying $| u|<
    r_3^2 /9 $,  we have \begin{equation*}  \begin{split}
     |t_1'-u|&\leq |t_1'-t_2'|+|t_2'|+|u|<(r_1^2+r_2^2)/3<r_1^2, \\  |t_2'-u|&\leq |t_2'|+|u|<(r_2^2+r_3^2)/3< r_2^2.
   \end{split}   \end{equation*} Thus,    the integrand in \eqref{eq:w-r}
equals to $1$ over the interval $   u \in( -r_3^2 /9,
    r_3^2 /9 )$ and so $ W_{\mathbf{r} }(\mathbf{t}')\geq 2 r_3^2 /9$.
  \\
   {\it Case ii. $r_3  >  r_2 $}.
For  $u$ satisfying  $|t_2'-u|<r_2^2 /9$,  we have
   \begin{equation*} \begin{split}
     |t_1'-u|&\leq |t_1'-t_2'|+|t_2'-u|<(r_1^2+r_2^2)/3<r_1^2, \\ |u|&< |t_2'-u|+| t_2'| <(r_2^2+r_3^2)/3< r_3^2.
   \end{split}   \end{equation*}
     Thus the    integrand in \eqref{eq:w-r}
equals to $1$ over the interval  $u\in (t_2'- r_2^2 /9, t_2'+r_2^2 /9)$ and so $ W_{\mathbf{r} }(\mathbf{t}')\geq  2r_2^2 /9$.
     So the first inequality  in \eqref{eq:w-claim} holds. The claim is proved.
\end{proof}

\section{The Calder\'on reproducing formula} Denote by $  { {*}}_{\mu}$   the {\it convolution on $\mathscr N$ along the subgroup  $  {\mathscr H}_\mu $}, i.e. for
$f\in L^1(\mathscr N)$ and $ H\in L^1({\mathscr H}_\mu)$,
 \begin{equation}\label{eq:convolution-subgroup} \begin{split}
f*  _{\mu } H (\mathbf{g}    ):
=&\int_{  \mathscr H _\mu }  f\left(\mathbf{g}\tau_\mu(\mathbf{h}_\mu )^{-1}  \right)H ( \mathbf{h}_\mu ) d \mathbf{h}_\mu
 ,
 \end{split}\end{equation}for $\mathbf{g}\in \mathscr N $.
 Define {\it maximal function     along the subgroup $ {\mathscr H}_\mu$} as
 \begin{equation*}\begin{split}
   M_{ \mu}  (f)(\mathbf{g}):=\sup_{ {r}_\mu\in \mathbb{R} _+} \frac 1{\left| {B}_\mu(\mathbf{0}_\mu,r_\mu)\right |}\int_{
  {B}_\mu(\mathbf{0}_\mu,r_\mu)  }|f (\mathbf{g} \tau_\mu(\mathbf{h}_\mu) )|d\mathbf{h_\mu},
 \end{split}\end{equation*}for $\mathbf{g}\in \mathscr N $ and $f\in L^1(\mathscr N ) $.
Denote    by $  \tilde{{
{*}}}_{\mu}$    the {\it convolution on $\tilde{\mathscr N} $ along the subgroup  $  {\mathscr H}_\mu $}, i.e. for
$f\in L^1(\tilde{\mathscr N})$ and $ H \in L^1({\mathscr H}_\mu)$, define
 \begin{equation*}\begin{split}
f\tilde{* } _{\mu } H (\tilde{\mathbf{g} }  )
=&\int_{  \mathscr H _\mu }  f\left( \tilde{\mathbf{g} } \tilde \tau_\mu(\mathbf{h}_\mu)  ^{-1}  \right)H ( \mathbf{h}_\mu ) d \mathbf{h}_\mu ,
 \end{split}\end{equation*}for $\tilde{\mathbf{g}}\in \tilde{\mathscr N }$,
 where $\tilde \tau_\mu:\mathscr H _\mu \rightarrow \mathscr H _1 \times \mathscr H _2 \times \mathscr H _3$ is the natural embedding homomorphism, $\mu=1,2,3$.
 Here $\tilde{\mathbf{g} } \tilde  \tau_\mu(\mathbf{h}_\mu)  ^{-1} =(\mathbf{g}_{\mu'},\mathbf{g}_\mu  \mathbf{h}_\mu ^{-1} ,\mathbf{g}_{\mu''})$ if $\tilde{\mathbf{g}}
 =(\mathbf{g}_{\mu'},\mathbf{g}_\mu    ,\mathbf{g}_{\mu''})$ in $\tilde{\mathscr N}$.

  \begin{proof}[Proof of Lemma \ref{lem:convolution}] Note that
\begin{equation*}\begin{split}
  \pi_*(F \tilde{*} G)(\mathbf{z},\mathbf{t}  ) = &\int_{\mathbb{R}}(F \tilde{*} G)(\mathbf{z},t_1-u,t_2-u,u)du
\\=&\int_{\mathbb{R}}du\int_{\mathbb{C}^{N}\times\mathbb{ R}^3} d\mathbf{w}\, d\tilde{\mathbf{s }}  \, F ( \mathbf{{w}}, \tilde{\mathbf{s}}) \\&\quad\cdot
G\Big (\mathbf{z}-\mathbf{w},t_1-u-\tilde s_1-\Phi_1(\mathbf{w}_1,\mathbf{z}_1),t_2-u-\tilde s_2-\Phi_2(\mathbf{w}_2,\mathbf{z}_2),u-\tilde
s_3-\Phi_3(\mathbf{w}_3,\mathbf{z}_3)\Big) .
\end{split}\end{equation*}
Now take coordinates transformation
$u-\tilde s_3-\Phi_3(\mathbf{w}_3,\mathbf{z}_3)= u'$,   $\tilde s_\alpha+\tilde s_3=  s_\alpha'$, $ \alpha=1,2$ and  $\tilde s_  3=  s_3'$ to get
\begin{equation*}\begin{split} \pi_*(F \tilde{*} G) &(\mathbf{z},\mathbf{t } )=\int_{\mathbb{R}}du ' \int_{\mathbb{C}^{N}\times\mathbb{ R}^3} d\mathbf{w}\, d\tilde
{\mathbf{s}} '   \, F
( \mathbf{{w}},  s_1'- s_3', s_2'-s_3',  s_3' )
\\&  \cdot G\Big(\mathbf{z}-\mathbf{w},t_1- s_1' -u'-\Phi_1(\mathbf{w}_1,\mathbf{z}_1)-\Phi_3(\mathbf{w}_3,\mathbf{z}_3),t_2 - s_2'
-u'-\Phi_2(\mathbf{w}_2,\mathbf{z}_2)-
\Phi_3(\mathbf{w}_3,\mathbf{z}_3),u'\Big)\\=& \int_{\mathbb{C}^{N}\times\mathbb{ R}^2} d\mathbf{w}ds_1'd s_2'\, \pi_*(F) ( \mathbf{{w}},  s_1
', s_2'  )
\\&\quad\cdot \pi_*(G)\Big(\mathbf{z}-\mathbf{w},t_1- s_1 '-\Phi_1(\mathbf{w}_1,\mathbf{z}_1)-\Phi_3(\mathbf{w}_3,\mathbf{z}_3),t_2- s_2'
-\Phi_2(\mathbf{w}_2,\mathbf{z}_2)-
\Phi_3(\mathbf{w}_3,\mathbf{z}_3) \Big) \\
=&\pi_*(F)*\pi_*(G)(\mathbf{z},\mathbf{t } ).
 \end{split}\end{equation*}
  The lemma is proved.
    \end{proof}

 \begin{lem}\label{lem:hat-convolution} For $ \varphi^{
(\mu)}\in L^1({\mathscr H}_\mu)$,  $F\in L^1(\tilde {\mathscr N} ) $ and $f\in L^1(\mathscr N ) $, we have
    \begin{equation}\label{eq:hat-convolution}\begin{split}
 F\tilde *\Phi_{\mathbf{r} } = &    \left(\left(F\tilde*_1\varphi_{r_1}^{
(1)}\right)\tilde* _{ 2} \varphi_{r_2}^{
(2)}\right ) \tilde{ *} _3  \varphi_{r_3}^{
(3)} , \qquad f*\varphi_{\mathbf{r} } =    \left(\left(f*_1\varphi_{r_1}^{
(1)}\right)* _{ 2} \varphi_{r_2}^{
(2)}\right ) { *} _3  \varphi_{r_3}^{
(3)}
,
 \end{split}\end{equation}where $\Phi_{\mathbf{r} }  :=  \varphi_{r_1}^{
(1)} \varphi_{r_2}^{
(2)}  \varphi_{r_3}^{
(3)}  $ and $\varphi_{\mathbf{r} } =\pi_*(\Phi_{\mathbf{r} }) $.
 \end{lem}
\begin{proof} We prove the second identity, which is a little bit more complicated. By definition,
 \begin{equation}\label{eq:f*varphi}\begin{split}
f*\varphi_{\mathbf{r} }(\mathbf{z} ,\mathbf{t})&= \int_{\mathscr N} d\mathbf{w}d\mathbf{s}
\int_{\mathbb{R}}f \left((\mathbf{z} ,\mathbf{t})( \mathbf{{w}}, \mathbf{s}  \right)^{-1}) \varphi_{r_1}^{
(1)}  ( \mathbf{w}_1,s_1 -u)\varphi_{r_2}^{
(2)}  ( \mathbf{w}_2, s_2-u  ) \varphi_{r_3}^{
(3)}( \mathbf{w}_3,u)   d u.
 \end{split}\end{equation}
It follows from  the multiplication law \eqref{eq:multiplication1} of $\mathscr N$ and the embedding \eqref{eq:embedding}  that
 \begin{equation*}\begin{split}
(\mathbf{z} ,\mathbf{t}) &( \mathbf{{w}}, \mathbf{s}  )^{-1} =
 \left (\mathbf{z}-\mathbf{w},t_1-s_1 -\Phi_1(\mathbf{z}_1,\mathbf{w}_1)-\Phi_3(\mathbf{z}_3,\mathbf{w}_3),t_2-s_2 -\Phi_2(\mathbf{z}_2,\mathbf{w}_2)-
\Phi_3(\mathbf{w}_3,\mathbf{z}_3)\right ) \\=&
(\mathbf{z}_1,\mathbf{z}_2- \mathbf{w}_2,\mathbf{z}_3- \mathbf{w}_3,t_1-u  -\Phi_3(\mathbf{z}_3,\mathbf{w}_3),t_2-s_2 -\Phi_2(\mathbf{z}_2,\mathbf{w}_2)-
\Phi_3(\mathbf{z}_3,\mathbf{w}_3) )\cdot\tau_1 ( \mathbf{w}_1, s_1 -u ) ^{-1}\\=&
(\mathbf{z}_1,\mathbf{z}_2 ,\mathbf{z}_3- \mathbf{w}_3,t_1-u  -\Phi_3(\mathbf{z}_3,\mathbf{w}_3),t_2 -u-
\Phi_3(\mathbf{z}_3,\mathbf{w}_3) )\cdot\tau_2( \mathbf{w}_2, s_2-u  ) ^{-1}\cdot\tau_1 ( \mathbf{w}_1, s_1 -u ) ^{-1}\\=&
(\mathbf{z} ,\mathbf{t})\cdot\tau_3( \mathbf{w}_3, u) ^{-1}\cdot\tau_2( \mathbf{w}_2, s_2-u  ) ^{-1}\cdot\tau_1( \mathbf{w}_1, s_1 -u ) ^{-1}.\end{split}\end{equation*}
Substituting it into \eqref{eq:f*varphi}, we
find that
\begin{equation*}\begin{split}
f*\varphi_{\mathbf{r} }(\mathbf{z} ,\mathbf{t})&=   \int_{\mathbb{R}} d u\int_{\mathbb{C}^{n_2+n_3}\times\mathbb{R} } \left(f*_1\varphi_{r_1}^{
(1)}\right) \Big((\mathbf{z} ,\mathbf{t})\cdot\tau_3( \mathbf{w}_3, u) ^{-1}\cdot\tau_2( \mathbf{w}_2, s_2-u  ) ^{-1}\Big) \\
&\qquad \qquad\qquad\qquad \qquad\qquad\qquad\cdot\varphi_{r_2}^{
(2)}( \mathbf{w}_2, s_2-u  )
\varphi_{r_3}^{
(3)}( \mathbf{w}_3,u)  d\mathbf{w}_2  d\mathbf{w}_3d {s }_2\\
&=   \int_{\mathbb{C}^{ n_3} \times\mathbb{R} }\left(\left(f*_1\varphi_{r_1}^{
(1)}\right)* _{ 2} \varphi_{r_2}^{
(2)}\right ) \Big((\mathbf{z} ,\mathbf{t})\cdot\tau_3( \mathbf{w}_3, u) ^{-1} \Big)
\varphi_{r_3}^{
(3)}( \mathbf{w}_3,u)   d\mathbf{w}_3 d u\\
&= \left(\left(\left(f*_1\varphi_{r_1}^{
(1)}\right)* _{ 2} \varphi_{r_2}^{
(2)}\right ) { *} _3  \varphi_{r_3}^{
(3)} \right)(\mathbf{z} ,\mathbf{t})
 \end{split}\end{equation*}by the definition  \eqref{eq:convolution-subgroup}  of   convolution  along  a subgroup.
  \end{proof}

   \begin{proof}[Proof of Theorem \ref{thm:maximal}]
Note that $\tau_\mu(\mathbf{h}_\mu) ^{-1}=\tau_\mu(\mathbf{h}_\mu^{-1})$ ($\tau_\mu$ is a homomorphism), and so for $\mathbf{g}\in\mathscr N$,
 \begin{equation*}\begin{split}
   \frac 1{|  {B}_\mu(\mathbf{0}_\mu,r_\mu) |}\int_{ {B}_\mu(\mathbf{0}_\mu,r_\mu)  }|f (\mathbf{g} \tau_\mu(\mathbf{h}_\mu) )|d\mathbf{h_\mu}&=\int_{ {\mathscr
   H}_\mu }|f
   (\mathbf{g} \tau_\mu(\mathbf{h}_\mu) )|\chi_{r_\mu}^{
(1)}(  \mathbf{h}_\mu   ) d\mathbf{h_\mu}
   \\&=\int_{ {\mathscr H}_\mu }\left|f \left(\mathbf{g} \tau_\mu(\mathbf{h}_\mu) ^{-1}\right)\right|\chi_{r_\mu}^{
(1)}(  \mathbf{h}_\mu   ) d\mathbf{h_\mu} =|f |{*}_\mu\chi_{r_\mu}^{
(\mu)}( \mathbf{g}),
 \end{split}\end{equation*} by $\chi_{r_\mu}^{
(\mu)}(\mathbf{h}_\mu^{-1} )=\chi_{r_\mu}^{
(\mu)}(\mathbf{h}_\mu  ) $.
Hence, $|f |{*}_\mu\chi_{r_\mu}^{
(\mu)}( \mathbf{g})\leq M_{ \mu}  (f)(\mathbf{g})$, and consequently,
   \begin{equation}\label{eq:it-maximal-function}\begin{split}
 | f|* \chi_{\mathbf{r} } (\mathbf{g})= &   \left(\left(|f|  {{*}}_{1 }\chi_{r_1}^{
(1)}\right) {{*}}_{ 2}\chi_{r_2}^{
(2)}\right) {*  }_3   \chi_{r_3}^{
(3)}   (\mathbf{g})  \leq   M_{ 3}\circ  M_{ 2} \circ M_{ 1}  (f)(\mathbf{g}),
 \end{split}\end{equation}
  by Lemma \ref{lem:hat-convolution}. Thus, the   boundedness of $  M_{it}  $  on $L^p(  \mathscr N)$ follows from the boundedness of $M_{ \mu}$ on
  $L^p( \mathscr N) $, which is the consequence of  the
  boundedness of $M_{ \mu}$ on $L^p(  \mathscr H _\mu)$.
 The result follows from the equivalence of $  M   $ and $  M_{it}  $ in  Proposition \ref{prop:tube}.
    \end{proof}

The convolutions  along subgroups are commutative/associative by he following proposition.
 \begin{prop}\label{prop:associativity}  For $f\in L^1(\tilde{\mathscr N})$,  $ H ,H'\in L^1({\mathscr H}_\mu)$ and  $ G \in L^1({\mathscr H}_\nu)$, we have
  \begin{equation}\label{eq:associativity}\begin{split}
  \left ( f\tilde{*}_\mu   H \right)  \tilde{*}_\nu G&=\left(f\ \tilde{*}_\nu G\right)  \tilde{*}_\mu   H,\\
  \left ( f\tilde{*}_\mu   H \right)  \tilde{*}_\mu   H '&=   f\tilde{*}_\mu \left (  H  \tilde{*}_\mu   H ' \right),
 \end{split}  \end{equation}
 where $ H  \tilde{*}_\mu   H '$ is a convolution $ H  $ and $   H '$ on ${\mathscr H}_\mu$.
 \end{prop}
\begin{proof}
 Recall that   elements $\tilde\tau_1( {\mathbf{h}}_1), \tilde\tau_2( {\mathbf{h}}_2)$ and $ \tilde\tau_3( {\mathbf{h}}_3 )$ for any  $\mathbf{h}_\mu\in  {\mathscr
 H}
  _ \mu$ are mutually commutative by definition.
Then,
  \begin{equation*}\begin{split}
   ( f\tilde{*}_\mu    H) \ \tilde{*}_\nu  G( {\mathbf{g}}) &=  \int_{  \mathscr H _\nu} (f\tilde{*}_\mu   H)\left( \mathbf{g}\tilde \tau_\nu(\mathbf{h}_\nu) ^{-1}
   \right) G( \mathbf{h}_\nu ) d\mathbf{h}_\nu \\
&=\int_{  \mathscr H _\nu } \left(\int_{  \mathscr H _\mu }  f\left( \mathbf{g} \tilde \tau_\nu(\mathbf{h}_\nu)  ^{-1} \tilde \tau_\mu(\mathbf{h}_\mu)  ^{-1}
\right) H( \mathbf{h}_\mu   ) d
\mathbf{h}_\mu\right)  G(   \mathbf{h}_\nu   ) d \mathbf{h}_\nu\\
&=\int_{  \mathscr H _\mu } \left( \int_{  \mathscr H _\nu }f\left( \mathbf{g}  \tilde\tau_\mu(\mathbf{h}_\mu)  ^{-1} \tilde\tau_\nu(\mathbf{h}_\nu)  ^{-1}
\right) G(  \mathbf{h}_\nu  ) d
 \mathbf{h}_\nu \right) H( \mathbf{h}_\mu ) d \mathbf{h}_\mu
 =(f\ \tilde{*}_\nu  G)\tilde{*}_\mu    H.
  \end{split}\end{equation*}The similar identity gives us the second one in \eqref{eq:associativity}.
  The proposition is proved.\end{proof}

A  function $\varphi^{(\mu)}$  on ${\mathscr H}_\mu$,
 $\mu=1,2,3$,
is called {\it Poisson-bounded} if
\begin{equation*}
   \|\varphi^{(\mu)}\|_M:=  \sup\left\{(1+\|\mathbf{g}_\mu\|)^{|I_\mu|+Q_\mu+1}
\left|\mathbf{Y}_\mu^{I_\mu}\varphi^{(\mu)}(\mathbf{g}_\mu) \right| ;  \mathbf{g}_\mu \in {\mathscr H}_\mu
, 0 \leq |I_\mu| \leq M\right\}
\end{equation*}
is
bounded for  sufficiently large $M$, where $Q_\mu=2n_\mu
 +2$ is the homogeneous dimension of ${\mathscr H}_\mu$ and
\begin{equation}\label{eq:Y-I}
   \mathbf{Y}_\mu^{I_\mu}:=Y_{\mu1}^{I_{\mu 1}}\cdots Y_{\mu
 (2n_\mu)}^{I_{\mu (2n_\mu)}},
\end{equation}for the multi-index $ I_\mu =  (I_{\mu 1},\cdots , I_{\mu (2n_\mu)})$, and
$Y_{\mu j}$'s are   left invariant vector fields \eqref{eq:left-invariant} on   $ {\mathscr H}_\mu$.  A Poisson-bounded  function $\varphi^{
(\mu)}$ on  $ {\mathscr
  {H}}_\mu$    is said to be   {\it w-invertible} if there exists a Poisson-bounded function $\psi^{
(\mu)}$, called a   {\it
w-inverse of $\varphi^{
(\mu)}$}, satisfying the {\it Calder\'on
reproducing formula} \begin{equation}\label{eq:reproducing-Heisenberg}
\int_{\mathbb{R}_+} f*\varphi_{r_\mu}^{
(\mu)}*\psi^{
(\mu)}_{r_\mu}( {\mathbf{g}}_\mu )\frac { d {r_\mu}} {r_\mu}=f( {\mathbf{g}}_\mu ),  \end{equation}
  for $   f \in  L
^2
( {\mathscr H}_\mu)$ .
 The Poisson kernel is an important special case of Poisson-bounded  function. Let $p  ^{
(\mu)}(t_\mu)$
be the convolution kernel  of the
operator  $e^{-t_\mu\sqrt{  \widetilde{ \triangle}_\mu}}$
  on $ {\mathscr
  {H}}_\mu$. We
obtain the {\it flag-like  Poisson kernel}  given by
 $p_{\mathbf{r} }  =\pi_* (p_{r_1}^{
(1)} p_{r_2}^{
(2)}p_{r_3}^{
(3)} ) $. The {\it flag-like heat kernel} $h_{\mathbf{r} }$ arises similarly.

 \begin{proof}[Proof of Theorem \ref{prop:reproducing}] By Lemma \ref{lem:lift-function},   we have $f=\pi_*( f^{\sharp})$ with $ f^{\sharp} \in L
^1\cap L
^2
 (\tilde{\mathscr N} )$. By applying the Calder\'on
reproducing formula \eqref{eq:reproducing-Heisenberg} to $L^1$-function $f^{\sharp}(
   {\mathbf{g}}_1  , {\mathbf{g}}_2 , \cdot )$ on  $ {\mathscr
  {H}}_3$ for almost all fixed $ {\mathbf{g}}_1 , {\mathbf{g}}_2  $,  we get
  \begin{equation*}
     f^{\sharp}( {\mathbf{g}}_1 , {\mathbf{g}}_2 , {\mathbf{g}}_3 )=\int_{\mathbb{R} _+} \left[f^{\sharp}
     \tilde{*} _3 \varphi_{r_3}^{
(3)} \ \tilde{*}_3 \psi_{r_3}^{
(3)}\right ] ( {\mathbf{g}}_1 , {\mathbf{g}}_2 , {\mathbf{g}}_3 )\frac { d {r_3}}{ {r_3}},
 \end{equation*} for $  {\mathbf{g} }_\mu\in  {\mathscr
  {H}}_\mu$.
  Then
  apply  the Calder\'on
reproducing formula \eqref{eq:reproducing-Heisenberg} on   $ {\mathscr
  {H}}_2$ and  $ {\mathscr
  {H}}_1$ again to obtain
 \begin{equation*}\begin{split}
     f^{\sharp}( {\mathbf{g}})&=\int_{\mathbb{R}^3_+} \left[\left( f^{\sharp}\tilde{*}_1  \varphi_{r_1}^{
(1)} \ \tilde{*}_1 \psi_{r_1}^{(1)}  \right)\tilde{*}_2   \varphi_{r_2}^{
(2)} \ \tilde{*}_2 \psi_{r_2}^{(2)} \right]\tilde{*}_3   \varphi_{r_3}^{
(3)} \ \tilde{*}_3 \psi_{r_3}^{(3)} ( {\mathbf{g}})\frac { d\mathbf{r}}{\mathbf{r}}
     \\&=\int_{\mathbb{R}^3_+} \left(\cdots\left.\left.\left.\left( f^{\sharp}\tilde{*}_1  \varphi_{r_1}^{
(1)} \right) \tilde{*}_2 \varphi_{r_2}^{(2)} \right)\tilde{*}_3   \varphi_{r_3}^{
(3)}\right) \tilde{*}_1 \psi_{r_1}^{(1)}  \right)  \ \tilde{*}_2 \psi_{r_2}^{(2)} \right)\ \tilde{*}_3 \psi_{r_3}^{(3)} ( {\mathbf{g}})\frac { d\mathbf{r}}{\mathbf{r}}
     \\&=\int_{\mathbb{R}^3_+}  f^{\sharp}\tilde{*}\left(\varphi_{r_1}^{
(1)} \varphi_{r_2}^{(2)} \varphi_{r_3}^{(3)}\right)\tilde{*}\left(\psi_{r_1}^{
(1)} \psi_{r_2}^{(2)} \psi_{r_3}^{(3)}\right)( {\mathbf{g}})\frac { d\mathbf{r}}{\mathbf{r}},
 \end{split} \end{equation*}
 by the commutativity/associativity of convolutions  along subgroups in Proposition \ref{prop:associativity} and using Lemma \ref{lem:hat-convolution}.
 Then, for ${\mathbf{g}}=(\mathbf{z},\mathbf{ t } )\in \mathscr N $,
 \begin{equation*}\begin{split}
    f( {\mathbf{g}}) &=\pi_*( f^{\sharp})( {\mathbf{g}})=\int_{\mathbb{R}}\int_{\mathbb{R}^3_+}
\left(   f^{\sharp}\tilde{*}\left(\varphi_{r_1}^{
(1)} \varphi_{r_2}^{(2)} \varphi_{r_3}^{(3)}\right)\tilde{*}\left(\psi_{r_1}^{
(1)} \psi_{r_2}^{(2)} \psi_{r_3}^{(3)}\right)\right)
    (\mathbf{z}, t_1-u,t_2-u,u)du \frac { d\mathbf{r}}{\mathbf{r}}
   \\
      & =\int_{\mathbb{R}^3_+}\pi_*\left(   f^{\sharp}\tilde{*}\left(\varphi_{r_1}^{
(1)} \varphi_{r_2}^{(2)} \varphi_{r_3}^{(3)}\right)\tilde{*}\left(\psi_{r_1}^{
(1)} \psi_{r_2}^{(2)} \psi_{r_3}^{(3)}\right)\right)( {\mathbf{g}}) \frac { d\mathbf{r}}{\mathbf{r}}
\\&= \int_{\mathbb{R}^3_+} \pi_*(   f^{\sharp})*\pi_*\left( \varphi_{r_1}^{
(1)} \varphi_{r_2}^{(2)} \varphi_{r_3}^{(3)} \right)*\pi_*\left( \psi_{r_1}^{
(1)} \psi_{r_2}^{(2)} \psi_{r_3}^{(3)}\right)( {\mathbf{g}})  \frac { d\mathbf{r}}{\mathbf{r}}
 =\int_{\mathbb{R}^3_+} f*\varphi_ {\mathbf{r}}*\psi_ {\mathbf{r}}( {\mathbf{g}}) \frac { d\mathbf{r}}{\mathbf{r }}
 \end{split}\end{equation*}
 by using Fubini's  Theorem and  Lemma \ref{lem:convolution}.
    \end{proof}
\begin{rem}\label{rem:compact-support} In the Calder\'on
reproducing formula \eqref{eq:reproducing-Heisenberg}  on the Heisenberg group, it is possible to
  choose  the function $\varphi^{(\mu)}$ or $\psi^{(\mu)}$
  compactly supported and w-invertible Poisson bounded, and they both have mean value zero  \cite[Section 4]{MRS}.
\end{rem}

 \begin{proof}[Proof of Theorem \ref{thm:g-function}]   Define a function $F:\mathscr N\longrightarrow \mathcal{ L}:=L
^2
(\mathbb{R}^2_+,\frac {dr_1}{r_1}\frac {dr_2}{r_2} )$ by
$
  F(\mathbf{g} )=     \left(f   {{*}}_{1 } \varphi_{r_1}^{
(1)}\right) {{*}}_{ 2}\varphi_{r_2}^{
(2)} (\mathbf{g} )
$
 with $
  | F(\mathbf{g} )|_{\mathcal{ L}}^2 :  =  \int_{\mathbb{R}^2_+} | F(\mathbf{g} )  |^2  \frac {dr_1}{r_1}\frac {dr_2}{r_2}  .
$ Here $\mathcal{ L}$ is a Hilbert space.
For a function $G:\mathscr H_3\longrightarrow \mathcal{ L} $,  define
 \begin{equation*}
   g_{\varphi^{(3)}} (G  )(\mathbf{z}_3 ,v ):=\left( \int_{\mathbb{R} _+}\left|  G  \tilde{ *}  _3\varphi_{r_3}^{(3)} (\mathbf{z}_3 ,v )\right|_{\mathcal{ L}} ^2 \frac {dr_3}{r_3}\right)^{\frac 12},
\end{equation*}where $\tilde{ *}  _3$ is the convolution on $\mathscr H_3$. For fixed $
\mathbf{z}_1 , \mathbf{z}_2, u   $ and $F$ given above, let $F_{\mathbf{z}_1,\mathbf{z}_2,u}:\mathscr H_3\longrightarrow \mathcal{ L}  $ be the
function on $\mathscr H_3$ given by  $F_{\mathbf{z}_1,\mathbf{z}_2,u}(\mathbf{z}_3 ,v ):=F(\mathbf{z} ,v+u,v- u)$. Then,
\begin{equation} \label{eq:g-varphi-3}\begin{split}
   g_{\varphi^{(3)}}^2(F_{\mathbf{z}_1,\mathbf{z}_2,u}  )(\mathbf{z}_3 ,v )& =\int_{\mathbb{R} _+}\left| \int_{\mathscr H_3}
   F_{\mathbf{z}_1,\mathbf{z}_2,u}((\mathbf{z}_3 ,v )  (\mathbf{z}_3' ,v ')^{-1})  \varphi_{r_3}^{
(3)} (\mathbf{z}_3' ,v ') d \mathbf{z}_3'dv ' \right|_{\mathcal{ L}} ^2 \frac {dr_3}{r_3}
   \\&=\int_{\mathbb{R}^3 _+} \left| \int_{\mathscr H_3}F\Big((\mathbf{z} ,v+u,v- u) \tau_3 (\mathbf{z}_3' ,v ')^{-1}\Big)  \varphi_{r_3}^{
(3)} (\mathbf{z}_3' ,v ') d \mathbf{z}_3'dv '\right|^2 \frac {d\mathbf{r} }{\mathbf{r} }
   \\&=\int_{\mathbb{R}^3_+}\left|\left(\left(f   {{*}}_{1 } \varphi_{r_1}^{
(1)}\right) {{*}}_{ 2}\varphi_{r_2}^{
(2)}\right) {   {*} }_3\varphi_{r_3}^{
(3)} \right|^2(\mathbf{z} ,  v+u,v- u)  \frac {d\mathbf{r} }{\mathbf{r} }\\& =\int_{\mathbb{R}^3_+}\left| f *  \varphi_{\mathbf{r }} \right|^2(\mathbf{z} ,
v+u,v- u  )  \frac {d\mathbf{r} }{\mathbf{r} }
 \end{split}\end{equation}by Lemma \ref{lem:hat-convolution}.
By the vector-valued Littlewood-Paley inequality on $\mathscr H_3$, we have\begin{equation*} \begin{split}
 \int_{   {{\mathscr H}}_3} g_{\varphi^{(3)}}^p(F_{\mathbf{z}_1,\mathbf{z}_2,u}  )(\mathbf{z}_3 ,v )d\mathbf{z}_3d v\leq C  \int_{   {{\mathscr H}}_3} \left|
 F_{\mathbf{z}_1,\mathbf{z}_2,u} \right|_{\mathcal{ L}} ^p(\mathbf{z}_3 ,v )   d\mathbf{z}_3d v .
 \end{split}\end{equation*}
Integrate it over $\mathbf{z}_1,\mathbf{z}_2,u$ and use \eqref{eq:g-varphi-3} to get
 \begin{equation*}  \begin{split}
  \|g_{\boldsymbol\varphi }(f)\|_p ^p=& \frac 12\int_{ \mathbb{C}^{n_1+n_2} \times \mathbb{R}^1} d\mathbf{z}_1 d\mathbf{z}_2 d u \int_{  \mathbb{C}^{n_3} \times
  \mathbb{R}^1} g_{\varphi^{(3)}}^p(F_{\mathbf{z}_1,\mathbf{z}_2,u}  )(\mathbf{z}_3 ,v )d\mathbf{z}_3d v\\\leq & \frac   C 2 \int_{  \mathscr N }
\left(\int_{\mathbb{R}^2_+}\left| \left(f   {{*}}_{1 } \varphi_{r_1}^{
(1)}\right) {{*}}_{ 2}\varphi_{r_2}^{
(2)} \right|^2(\mathbf{z} , v+u,v- u  )\frac {dr_1}{r_1}\frac {dr_2}{r_2}\right)^{\frac p2}  d\mathbf{z} d udv\\ =& C  \int_{  \mathscr N }
\left(\int_{\mathbb{R}^2_+}\left| \left(f   {{*}}_{1 } \varphi_{r_1}^{
(1)}\right) {{*}}_{ 2}\varphi_{r_2}^{
(2)} \right|^2(\mathbf{z} ,\mathbf{ t}  )\frac {dr_1}{r_1}\frac {dr_2}{r_2}\right)^{\frac p2}  d\mathbf{z} d \mathbf{ t} .
 \end{split}\end{equation*}
Repeating this procedure, we get
\begin{equation*}  \begin{split}
  \|g_{\boldsymbol\varphi }(f)\|_p ^p \lesssim &  \int_{  \mathscr N } \left(\int_{\mathbb{R} _+}\left|  f   {{*}}_{1 } \varphi_{r_1}^{
(1)}  \right|^2(\mathbf{z} , \mathbf{t}  )\frac {dr_1}{r_1} \right)^{\frac p2}  d\mathbf{z} d \mathbf{t} \lesssim  \int_{  \mathscr N }  |  f(\mathbf{z} ,
\mathbf{t}  )   | ^{  p }  d\mathbf{z} d \mathbf{t}.
 \end{split}\end{equation*}

 The converse can be proved  by duality. Let $\langle \cdot , \cdot\rangle$ be the pair between $  L
^p
(\mathscr N)$ and $  L
^q
(\mathscr N)$ with $\frac 1p+\frac 1q=1$. For any $h \in L
^q
(\mathscr N)$, we have
 \begin{equation}\label{eq:g-converse}\begin{split}
|\langle f, h\rangle |=&  \left| \int_{\mathbb{R}^3_+}\left\langle f*   {\varphi}_{\mathbf{r} }*  {\psi}_{\mathbf{r} } ,  h \right \rangle \frac
{d\mathbf{r}}{\mathbf{r}}\right|
 =   \left| \int_{ \mathscr N} \int_{\mathbb{R}^3_+}  f*   {\varphi}_{\mathbf{r} } (\mathbf{g}) \overline{ { h} *  \check{{\psi}}_{\mathbf{r} }(\mathbf{g})}
\frac {d\mathbf{r}}{\mathbf{r}} d
 \mathbf{g}\right|\\
\leq &\int_{ \mathscr N} \left(\int_{\mathbb{R}^3_+}\left| f*  {\varphi}_{\mathbf{r} }  (\mathbf{g})\right|^2\frac {d\mathbf{r}}{\mathbf{r}}\right)^{\frac
12}\left(\int_{\mathbb{R}^3_+}\left|  {h}* \check{ \psi} _{\mathbf{r} }  (\mathbf{g})\right|^2\frac {d\mathbf{r}}{\mathbf{r}}\right)^{\frac 12}d \mathbf{g}
\\
\leq &\left(\int_{ \mathscr N} \left(\int_{\mathbb{R}^3_+}\left| f*  {\varphi}_{\mathbf{r} }  (\mathbf{g})\right|^2\frac {d\mathbf{r}}{\mathbf{r}}\right)^{\frac
p2}d \mathbf{g}   \right)^{\frac 1p}\left(\int_{ \mathscr N} \left(\int_{\mathbb{R}^3_+}\left|  {h}* \check{ \psi }_{\mathbf{r} }
(\mathbf{g})\right|^2\frac
{d\mathbf{r}}{\mathbf{r}}\right)^{\frac q2}d \mathbf{g}   \right)^{\frac 1q}\\
\leq & \|g_{\boldsymbol\varphi }(f)\|_p  \|g_{\breve{\boldsymbol\psi }}( {h})\|_q\lesssim  \|g_{\boldsymbol \varphi}(f)\|_p   \|h\|_q,
 \end{split}\end{equation}where $\psi$ is real,
$
    \breve{{\psi}}_{\mathbf{r} } (\mathbf{g} ):={\psi}_{\mathbf{r} } (\mathbf{g}^{-1}).
$
   It follows that $\|  f \|_p\lesssim \|g_{\boldsymbol\varphi }(f)\|_p $.
  \end{proof}

\begin{rem}   We have $\| S_{area,\boldsymbol \varphi}(f)\|_{ L
^2
(\mathscr N)}= \| g_{\boldsymbol \varphi}(f)\|_{ L
^2
(\mathscr N)}\approx\| f\|_{ L
^2
(\mathscr N)}$ by the definition of the  Lusin-Littlewood-Paley area function \eqref{eq:area-function},  Fubini's theorem and Theorem \ref{thm:g-function}.
\end{rem}

 \section{Flag-like singular integrals}

\subsection{Convolution kernels of tri-parameters on $ \tilde{{\mathscr N}} $} \label{section:convolution}
The theory of singular integrals of bi-parameters on the product of two stratified nilpotent groups in \cite{MRS} can be easily generalized to the case of  tri-parameters.
It is convenient to formulate the  size
estimates and the  cancellation conditions
for a kernel in terms of normalized bump functions (cf. e.g. \cite{MRS}).
 By definition, a {\it normalized bump function} (briefly denoted by n.b.f.) is a smooth function  supported on the
unit ball bounded by a fixed constant together with its gradient. Observe that
the $L^1$-norm of a normalized bump function is also bounded by a fixed constant.
  A {\it  standard  convolution kernel of tri-parameters} on $ \tilde{{\mathscr N}} $ is a distribution $ K$
on $ \tilde{{\mathscr N}} $, which coincides with a smooth function away from the subgroups $\mathscr H_ \mu\times\mathscr H_\nu$  for all $\mu\neq\nu$,    and
satisfies
 \\
(1) ({\it The size
estimates}) For any multi-indices  ${I_1},{I_2},{I_3}$,
\begin{equation}\label{eq:size-estimates}
   |\mathbf{Y}_1^{I_1} \mathbf{Y}_2^{I_2} \mathbf{Y}_3^{I_3}K(\mathbf{g} )|\leq C_{{I_1},{I_2},{I_3}}  \| \mathbf{g}_1 \|^{-Q_1-|{I_1}| }\| \mathbf{g}_2
   \|^{-Q_2-|{I_2}| | }\| \mathbf{g}_3 \|^{-Q_3- |{I_3}| }
\end{equation}
 for all $\mathbf{g}=(\mathbf{g}_1 ,\mathbf{g}_2 ,\mathbf{g}_3 )\in  \tilde{{\mathscr N}}$, where $\mathbf{Y}_\mu^{I_\mu} $ for a multi-index $I_\mu  $  is given by
 \eqref{eq:Y-I},  $\mu=1,2,3$.
  \\
 (2)  ({\it The cancellation conditions})
 \begin{equation}\label{eq:cancellation1}
   \left|\int_{\mathscr H_{\mu' }\times\mathscr  H_{\mu''}}\mathbf{Y}_\mu^{I_\mu }  K(\mathbf{g} )\phi  (\delta_1\mathbf{g}_{\mu'},
   \delta_2\mathbf{g}_{\mu''})d\mathbf{g}_{\mu'}d
   \mathbf{g}_{\mu''} \right|\leq C_{ I_\mu }\| \mathbf{g}_\mu\|^{-Q_\mu-|{I_\mu }| },
\end{equation}
for every multi-index $I_\mu $ and every  n.b.f. $\phi $ on $\mathscr H_{\mu'} \times \mathscr H_{\mu''}$ and
every $\delta_1, \delta_2> 0$;
\begin{equation}\label{eq:cancellation2}
   \left|\int_{\mathscr  H_{\mu''}}\mathbf{Y}_\mu^{I_\mu } \mathbf{Y}_{\mu'}^{I_{\mu'} } K(\mathbf{g} )\phi  (  \delta\mathbf{g}_{\mu''}) d \mathbf{g}_{\mu''}
   \right|\leq C_{ I_\mu,I_{\mu'} }\| \mathbf{g}_\mu\|^{-Q_\mu-|{I_\mu }| }\| \mathbf{g}_{\mu'} \|^{-Q_{\mu'}-|{I_{\mu'}}|   },
\end{equation}
 for every multi-indices $I_\mu,I_{\mu'}  $ and every  n.b.f. $\phi $ on $\mathscr H_{\mu''}$ and
every $\delta> 0$; and
\begin{equation}\label{eq:cancellation3}
   \left|\int_{\mathscr N}  K(\mathbf{g} )\phi (\delta_1\mathbf{g}_1, \delta_2\mathbf{g}_2,\delta_3\mathbf{g}_3 )d\mathbf{g}  \right|\leq C,
\end{equation}
for   every  n.b.f. $\phi $ on $\tilde {\mathscr N}$ and
every $\delta_1, \delta_2 , \delta_2> 0$.

It is convenient to use the above estimates for
 n.b.f. on a bounded ball  instead of the  unit ball in definition, because it is a dilation of a  n.b.f.

 Recall that on a nilpotent group $G$, the {\it space $\mathscr D(G)$ of test functions} consists of all compactly supported
 smooth  functions with the topology  of uniform convergence of functions and any partial  derivatives  on  compact  subsets. The {\it space $\mathscr D'(G)$ of distributions} is the dual of $\mathscr D(G)$. Let $\langle \cdot, \cdot\rangle$ be the pair between
 $\mathscr D'(G)$ and $\mathscr D(G)$.
Given a test function $ \phi \in\mathscr D({\mathscr H}_\mu)$, define the distribution $K_{\phi }\in \mathscr D'(\mathscr  H_{\mu' }\times \mathscr  H_{\mu''})$ by
\begin{equation*}
   \langle K_{\phi }, \varphi\rangle=\langle K, {\phi }\otimes \varphi \rangle
\end{equation*}
for any test function  $\varphi\in \mathscr D(\mathscr H_{\mu' }\times\mathscr   H_{\mu''}) $; Given a test function $\varphi$ on $\mathscr  H_{\mu' }\times \mathscr
H_{\mu''}$, define
the distribution $K_{\varphi}\in \mathscr D'( {\mathscr H}_\mu)$ by
\begin{equation*}
   \langle K_{\varphi}, \phi_\mu\rangle=\langle K, {\phi_\mu}\otimes\varphi \rangle
\end{equation*}
for any test function $ \phi_\mu\in \mathscr D( {\mathscr H}_\mu)$. The integrals over subgroups in the above cancellation conditions should be interpreted as
such pairs.

Fix  a test function $ \varphi(\mathbf{g} ) :=  \varphi ^{
(1)}  (\mathbf{g}_1 ) \varphi ^{
(2)}(\mathbf{g}_2 )\varphi ^{
(3)} (\mathbf{g}_3 ) $  with $ \varphi ^{
(\mu)}  $ supported in $B_\mu (\mathbf{0}_\mu,1)$,  $\varphi ^{
(\mu)} (\mathbf{g}_\mu^{-1})=\varphi ^{
(\mu)}(\mathbf{g}_\mu )$ and $\int \varphi ^{
(\mu)}=1$. Consider the following {\it regularization} of $K$:
\begin{equation}\label{eq:K-varepsilon}
   K_\varepsilon(\mathbf{g} )=\varphi( \varepsilon \mathbf{g})(\varphi_{\varepsilon}* K)(\mathbf{g} ),
\end{equation}where $\varphi_{\varepsilon }:=  \varphi_{\varepsilon}^{
(1)}   \varphi_{\varepsilon}^{
(2)}\varphi_{\varepsilon}^{
(3)}$.
Here the convolution of $ \phi\in \mathscr D(\tilde{\mathscr  N})$  with a distribution $K$  is defined as
\begin{equation}\label{eq:convolution-distribution}
 \phi*  K(\mathbf{g})=\langle K, \phi_{\mathbf{g}}\rangle,\qquad \phi_{\mathbf{g}} (\mathbf{h}):= \phi( \mathbf{g}\mathbf{h}^{-1}),
\end{equation}
which coincides with the definition \eqref{eq:convolution-def} of convolution for $L^1$ functions. Obviously,  $\phi*   K $ is   smooth.

  \begin{prop} \label{prop:regularization} For $0<\varepsilon<1$,  $K_\varepsilon$ is a convolution kernel  on $ \tilde{{\mathscr N}} $  with constants independent
  of $\varepsilon$.
     \end{prop}
 \begin{proof} This fact for convolution   kernels on product spaces were used in \cite{MRS} \cite{NS} without proof. Here we sketch a proof for the
 convenience of readers. We only consider $K_\varepsilon =  \varphi_{\varepsilon}* K $.
Note that
 \begin{equation*}\begin{split}
    Y_{\mu j}  K_\varepsilon(\mathbf{g} )&=\left. \frac d{da}\right|_{a=0}
     \Big\langle K, ({\varphi}_\varepsilon)_{\mathbf{g}e(a)}\Big\rangle=\left\langle K,\left.\frac
    d{da}\right|_{a=0} ({\varphi}_\varepsilon)_{\mathbf{g} }(\cdot e(-a))\right\rangle \\&
    = \Big\langle K,  - Y_{\mu j}({\varphi}_\varepsilon)_{\mathbf{g} }
  \Big\rangle =  \Big\langle Y_{\mu j}K,   ({\varphi}_\varepsilon)_{\mathbf{g} }\Big\rangle,
\end{split} \end{equation*}by using \eqref{eq:convolution-distribution}, the continuity of a distribution and    definition of partial derivatives of a distribution,
where $e(a)$ is the  one-parameter subgroup generated by left invariant vector $Y_{\mu j}$.
 Consequently, we have
  \begin{equation}\label{eq:Y-K-varepsilon}\begin{split}
  \mathbf{Y}_1^{I_1} \mathbf{Y}_2^{I_2} \mathbf{Y}_3^{I_3} K_\varepsilon(\mathbf{g} )=    \left    \langle K,   (-1)^{|I|}\mathbf{Y}_1^{I_1} \mathbf{Y}_2^{I_2}
  \mathbf{Y}_3^{I_3}({\varphi}_\varepsilon)_{\mathbf{g} } \right\rangle =    \left\langle  \mathbf{Y}_1^{I_1} \mathbf{Y}_2^{I_2} \mathbf{Y}_3^{I_3}K,
  ({\varphi}_\varepsilon)_{\mathbf{g} } \right\rangle   .
\end{split}\end{equation}

For the cancellation conditions for $K_\varepsilon$,  note that for a  n.b.f. $\psi$ on $ \mathscr H_3$,
\begin{equation} \label{eq:Y-K-varepsilon-int}\begin{split} \int_{  \mathscr H_3} \mathbf{Y}_1^{I_1}\mathbf{Y}_2^{I_2} K_\varepsilon(\mathbf{g} )\psi (
\delta_3\mathbf{g}_3 ) d\mathbf{g}_3 & =  \int_{  \mathscr H_3}  \left\langle  \mathbf{Y}_1^{I_1}\mathbf{Y}_2^{I_2} K,    ({\varphi}_\varepsilon)_{\mathbf{g} }
\right\rangle\psi (  \delta_3\mathbf{g}_3 )  d\mathbf{g}_3\\& = \left\langle  \mathbf{Y}_1^{I_1} \mathbf{Y}_2^{I_2} K,  \int_{\mathscr H_3}
({\varphi}_\varepsilon)_{\mathbf{g} } \psi (  \delta_3\mathbf{g}_3 )  d\mathbf{g}_3\right\rangle    = \left\langle  \mathbf{Y}_1^{I_1} \mathbf{Y}_2^{I_2} K,
\Psi\right\rangle ,
 \end{split} \end{equation}by using \eqref{eq:Y-K-varepsilon}. The second identity follows from the  continuity of a distribution, and
 \begin{equation*}
    \Psi( \mathbf{ h} ):=
 (\varphi^{(1)}_{  \varepsilon})_{\mathbf{g}_1}(\mathbf{ h}_1 )  (\varphi^{(2)}_{  \varepsilon})_{\mathbf{g}_2}(\mathbf{ h}_2 ) \cdot \widehat{\psi} (
 \delta_3 \mathbf{ h}_3 )
 \end{equation*}
   with
 \begin{equation}\label{eq:widehat-psi} \begin{split}
\widehat{ \psi} (  \mathbf{ h}_3 )&= \int_{\mathscr H_3}    (\varphi^{(3)}_{  \varepsilon})_{\mathbf{g}_3}(\delta_3^{-1}\mathbf{ h}_3  )  \psi (
\delta_3\mathbf{g}_3 )  d\mathbf{g}_3
= \int_{\mathscr H_3}     \varphi^{(3)} _{\varepsilon \delta_3 }( \mathbf{g}_3)  \psi ( \mathbf{g}_3 \mathbf{ h}_3 )  d\mathbf{g}_3.
 \end{split} \end{equation}

{\it Case i:   $\|\mathbf{g}_1\|,\|\mathbf{g}_2\| >C_0\varepsilon$ for a large $C_0>0$}. If $ \varepsilon \delta_3 <1$, $\widehat{ \psi}$ is   a  n.b.f.   on a bounded ball by \eqref{eq:widehat-psi}. Thus
\begin{equation}\label{eq:L.H.S.}
  \left |\left\langle  \mathbf{Y}_1^{I_1} \mathbf{Y}_2^{I_2} K,
\Psi\right\rangle\right|\lesssim \int_{\mathscr H_1\times \mathscr H_2} \frac {  \varphi^{(1)}_{  \varepsilon} (\mathbf{g}_1\mathbf{
   h}_1^{-1} ) \varphi^{(2)}_{  \varepsilon}  (\mathbf{g}_2\mathbf{ h}_2^{-1} )  }{\| \mathbf{h}_1\|^{ Q_1+| I_1| }\| \mathbf{h}_{2} \|^{ Q_2+| I_{2}   |  }}
   d\mathbf{h}_1  d\mathbf{h}_2\lesssim \| \mathbf{g}_1\|^{-Q_1-| I_1| }\| \mathbf{g}_{2} \|^{-Q_2-| I_{2}   | },
\end{equation}by the cancellation condition
\eqref{eq:cancellation2} for $K$. If $\varepsilon \delta_3 \geq1$, note that $ \Psi( \mathbf{ h} ) = (\varphi^{(1)}_{  \varepsilon})_{\mathbf{g}_1}(\mathbf{ h}_1 )
(\varphi^{(2)}_{  \varepsilon})_{\mathbf{g}_2}(\mathbf{ h}_2 )  \check{\psi } ( \varepsilon^{-1}\mathbf{ h}_3 )$ with
\begin{equation*}
   \check{  \psi} (  \mathbf{ h}_3 )
= \int_{\mathscr H_3}     \varphi^{(3)} ( \mathbf{g}_3\mathbf{ h}_3^{-1})  \psi ( \varepsilon \delta_3 \mathbf{g}_3 )  d\mathbf{g}_3
\end{equation*}  also to be a n.b.f. on a bounded ball. The estimate \eqref{eq:L.H.S.} holds similarly.

{\it Case ii: $\|\mathbf{g}_1\|,\|\mathbf{g}_2\|\leq C_0\varepsilon$}. Note that $ \langle  \mathbf{Y}_1^{I_1} \mathbf{Y}_2^{I_2} K,  \Psi \rangle=(- \varepsilon)^{
-Q_1-|{I_1}|-Q_2-|{I_2}|} \langle  K,  \widehat{ \Psi} ({\varepsilon}^{-1} \cdot, {\varepsilon}^{-1}\cdot,\delta_3 \cdot ) \rangle$ with
\begin{equation*}
   \widehat{ \Psi }(  \mathbf{ h}  ):=  ({\mathbf{Y}}_1^{I_1}  \psi_1) \left( \mathbf{ h}_1 \right)  ({\mathbf{Y}}_2^{I_2}  \psi_2 )
   \left(\mathbf{ h}_2 \right)\widehat{ \psi} (  \mathbf{ h}_3 ),
\end{equation*} by \eqref{eq:Y-K-varepsilon}, where $\psi_1\left( \mathbf{ h}_1 \right):=\varphi^{(1)}_{{\varepsilon}^{-1}  \mathbf{g}_1}  \left( \mathbf{ h}_1 \right)$, $\psi_2   \left(\mathbf{ h}_2 \right):=\varphi^{(2)}_{{\varepsilon}^{-1}
\mathbf{g}_2}   \left(\mathbf{ h}_2 \right)  $.
  $ \widehat{ \Psi }$ is  a  n.b.f. on a bounded ball, and so
  \begin{equation}\label{eq:varepsilon-g}
   \left |\left\langle  \mathbf{Y}_1^{I_1} \mathbf{Y}_2^{I_2} K,
\Psi\right\rangle\right| \ \lesssim  \varepsilon ^{
-Q_1-|{I_1}|-Q_2-|{I_2}|}     \lesssim \prod_{\mu=1}^2\| \mathbf{g}_\mu \|^{-Q_\mu-|{I_\mu}| }
   \end{equation}
  by  the cancellation condition
\eqref{eq:cancellation3} for $K$.

{\it Case iii:   $\|\mathbf{g}_2\|>C_0\varepsilon,\|\mathbf{g}_1\|\leq C_0\varepsilon$ or $\|\mathbf{g}_2\|\leq C_0\varepsilon,\|\mathbf{g}_1\|>C_0\varepsilon$}. The combined method of
the above $2$ cases yields the estimate.

So the cancellation condition \eqref{eq:cancellation2} for $K_\varepsilon$ holds with constants independent of $\varepsilon$.
It is similar to get the cancellation conditions \eqref{eq:cancellation1} and \eqref{eq:cancellation3} for $K_\varepsilon$. We omit the details.

For the size
estimate   \eqref{eq:size-estimates} for  $K_\varepsilon$, if $\|\mathbf{g}_\mu \|\geq C_0\varepsilon$ for each $\mu$, we have \begin{equation*}
  \left|\mathbf{Y}_1^{I_1} \mathbf{Y}_2^{I_2} \mathbf{Y}_3^{I_3} K_\varepsilon(\mathbf{g} )\right|
 \lesssim  \int \frac {{\varphi}_\varepsilon (\mathbf{g}\mathbf{h}^{-1})}{\| \mathbf{h}_1 \|^{ Q_1+|{I_1}| }\| \mathbf{h}_2
   \|^{ Q_2+|{I_2}|  }\| \mathbf{h}_3 \|^{ Q_3+|{I_3}| }} d\mathbf{h}  \lesssim \prod_{\mu=1}^2\| \mathbf{g}_\mu \|^{-Q_\mu-|{I_\mu}| },
\end{equation*}by applying the size
estimate  for $K$ to the right hand side of \eqref{eq:Y-K-varepsilon}. The estimate holds by $ \varphi_\varepsilon  $ supported in $B(\mathbf{0},\varepsilon)$.
If $\|\mathbf{g}_\mu \|< C_0\varepsilon$ for each $\mu$, note that  $ \hat{{\varphi}} ( \mathbf{h} ) =   \varphi ( \varepsilon ^{-1}\mathbf{g}  \cdot  \mathbf{h}^{-1} )$ is a   n.b.f. on a bounded ball,  and
\begin{equation}\label{eq:Y-K-varepsilon2}
   \mathbf{Y}_1^{I_1} \mathbf{Y}_2^{I_2} \mathbf{Y}_3^{I_3}({\varphi}_\varepsilon)_{\mathbf{g} }(
   \mathbf{h}) = \varepsilon^{ - Q -|{I_1}| -|{I_2}|  - |{I_3}| }\left (\mathbf{Y}_1^{I_1} \mathbf{Y}_2^{I_2} \mathbf{Y}_3^{I_3} \hat{{\varphi}} \right)(\varepsilon^{-1} \mathbf{h} ),
\end{equation}while  $\mathbf{Y}_1^{I_1} \mathbf{Y}_2^{I_2} \mathbf{Y}_3^{I_3} \hat{{\varphi}}$ is also a   n.b.f. on a bounded ball.
   Thus
  $
      |\mathbf{Y}_1^{I_1} \mathbf{Y}_2^{I_2} \mathbf{Y}_3^{I_3} K_\varepsilon(\mathbf{g} )|
 \lesssim \prod_{\mu=1}^3\| \mathbf{g}_\mu \|^{-Q_\mu-|{I_\mu}| }
  $ as \eqref{eq:varepsilon-g}
   by  the cancellation condition  \eqref{eq:cancellation3} for $K $.
   It is similar to get the size
estimate  when some of $\|\mathbf{g}_\mu\|$'s $\geq C_0\varepsilon$   and others are $< C_0\varepsilon$. We omit the details.
   \end{proof}

 \subsection{Flag-like singular integrals    and     boundedness on $L^p(\mathscr N)$  }Since for
 $F\in L_{loc}^1(\tilde {\mathscr N})$,
 \begin{equation}\label{eq:phi}\begin{split}
   \pi_*\left( \breve{\varphi}^\sharp \tilde{ * } F\right)(\mathbf{0})&=    \breve{\varphi }  * \pi_*(F)(\mathbf{0}) =   \int_{ {\mathscr N}}
 \varphi(\mathbf{h})  \pi_*(F)(\mathbf{h}) d\mathbf{h}=\left \langle \pi_*F,  \varphi \right\rangle_{  {{\mathscr N}}},
 \end{split}\end{equation}
 we define the {\it push-forward distribution} $K^\flat$ for  a convolution kernel $ K $ on $ \tilde{{\mathscr N}} $ as
\begin{equation}\label{eq:K-flat}
    \langle K^\flat,  \varphi  \rangle_{  {{\mathscr N}}} = \pi_*\left( \breve{\varphi}^\sharp \tilde{ * } K\right)(\mathbf{0}),
\end{equation}
for $\varphi\in \mathscr  D(\mathscr N)$.

\begin{prop}   For a convolution kernel $ K $ on $ \tilde{{\mathscr N}} $,
 \eqref{eq:K-flat} defines a distribution $ K^\flat $ on $  \mathscr N   $. \end{prop}
\begin{proof}
If we write
\begin{equation}\label{eq:K-flat1} \langle K^\flat,  \varphi  \rangle_{  {{\mathscr N}}}
  =\int_{\mathbb{R}}(\breve{\varphi}^\sharp\tilde*  K )(\mathbf{0}_{\mathbf{z}},-u,-u,u)du=\int_{|u|\leq M} +\int_{|u|>M},
\end{equation}we find that
\begin{equation}\label{eq:K-flat2} \begin{split}
 \left | \int_{|u|\leq M} \left\langle  { K },\breve{\varphi}^\sharp_{(\mathbf{0}_{\mathbf{z}},-u,-u,u)}   \right\rangle du \right|&\lesssim  \|\varphi\|_{C^1},
\\
 \left | \int_{|u|> M} \int_{\mathscr N}  { K } (\mathbf{h})\breve{\varphi}^\sharp ((\mathbf{0}_{\mathbf{z}},-u,-u,u)\mathbf{h}^{-1} )d \mathbf{h}  du \right|& \lesssim
 \int_{|u|> M} \frac 1{|u|^{Q/2}} du\lesssim  \|\varphi\|_{C^1},
\end{split}\end{equation}
for large $M$, where we have used the facts that $\breve{\varphi}^\sharp_{(\mathbf{0}_{\mathbf{z}},-u,-u,u)} /\|\varphi\|_{C^1}$ for $|u| \leq M$ is a   n.b.f. on a bounded ball,
and for $|u| > M$, $|{ K } (\mathbf{z},\mathbf{t})|\lesssim\frac 1{|u|^{Q/2}}$ for $ (\mathbf{z},\mathbf{t})$ in the support of  the function $\breve{\varphi}^\sharp ((\mathbf{0}_{\mathbf{z}},-u,-u,u)(-\mathbf{z},-\mathbf{t}) )$, where $\mathbf{t}\approx ( -u,-u, u)$, by the size
estimate   \eqref{eq:size-estimates} for  $K $. Thus, $|\langle K^\flat,  \varphi  \rangle_{  {{\mathscr N}}}|\lesssim  \|\varphi\|_{C^1}$. The result follows.
  \end{proof}

\begin{cor} \label{prop:pi-K-varepsilon} $\pi_*(K_\varepsilon)$ converges to the distribution $K^\flat$ as $ {\varepsilon\rightarrow0}$.
    \end{cor}
 \begin{proof}  By   \eqref{eq:phi},   $\langle  \pi_*(K_\varepsilon) ,  \varphi  \rangle_{  {{\mathscr N}}} = \pi_*\left( \breve{\varphi}^\sharp \tilde{ * } K_\varepsilon\right)(\mathbf{0})$.
   The result follows from the definition
 of $     K^\flat$ in \eqref{eq:K-flat} by taking limit $\varepsilon\rightarrow 0$, $K_\varepsilon\rightarrow K$ in $\mathscr D'(\tilde{ {\mathscr N}})$, and using \eqref{eq:K-flat2} for $K_\varepsilon$ instead of $K$.
 \end{proof}
As \cite[Lemma 4.2]{MRS},  we can prove the following pointwise estimate.
\begin{prop}
 \label{prop:decay} Suppose that $K$ is a     convolution kernel on $ \tilde{{\mathscr N} }$  and  $ {{\varphi}}^{(\mu)}$, $\mu=1,2,3$, are
  n.b.f. on $ {\mathscr H}_\mu$, respectively,   with  mean value zero. Then, there exists a constant $C>0$, only depending on $I^{(\mu)}$'s, ${{\varphi}}^{(\mu)}$'s and
  constants in the size
estimates   and   cancellation conditions for $K$, such that
   \begin{equation*}
      \left|\mathbf{Y}_1^{I_1} \mathbf{Y}_2^{I_2} \mathbf{Y}_3^{I_3}(\Phi_{\mathbf{r} } \tilde{ {*}}  K)\right|\left ( {\mathbf{g}}\right)\leq C\prod_{\mu=1}^3\frac {r_\mu}{\left(r_\mu+ \left\|
       {\mathbf{g}}_\mu\right\|\right)^{ Q_\mu+|I_\mu|+1}  },
   \end{equation*}where $ \Phi_{\mathbf{r} } ({\mathbf{g}} ):=\varphi_{r_1}^{
(1)} ({\mathbf{g}}_1)  \varphi_{r_2}^{
(2)}  ({\mathbf{g}}_2)  \varphi_{r_3}^{
(3)} ({\mathbf{g}}_3) $.
   \end{prop}
\begin{proof} We assume $r_\mu = 1$. The general case is similar. We   also assume  $K$  smooth, if replace $K$ by $K_\varepsilon$.

   If $\|\mathbf{g}_\mu\|$, $\mu=1,2,3$, are all small,
 \begin{equation*}\begin{split}
 \left|\Phi\tilde{ {*}}  K\left( {\mathbf{g}}\right) \right|=\left|\int_{ {\mathscr N}}  K \left( {\mathbf{h}}\right)
 \breve{ \Phi}\left( {\mathbf{h}}{\mathbf{g}}^{-1}\right)d {\mathbf{h}} \right|\leq C,
 \end{split}  \end{equation*} since $\breve{ \Phi}( \cdot {\mathbf{g}}^{-1} )$ is  a   n.b.f.   on a bounded ball.

  If one of $\|\mathbf{g}_\mu\|$'s is large   and other two are small, without loss of generality, we can assume that
  $\|\mathbf{g}_1\|$ and $\|\mathbf{g}_2\|$ are small and  $\|\mathbf{g}_3\|$ is large. Then  we have
 \begin{equation*}\begin{split}
  \Phi\tilde{ {*}} K \left( {\mathbf{g}}\right)&=\int_{ {\mathscr N}}  K\left ( {\mathbf{h}}_1,
  {\mathbf{h}}_2, {\mathbf{h}}_3\right)  \breve{ \Phi}\left({\mathbf{h}}_1 {\mathbf{g}}_1^{-1},
{\mathbf{h}}_2 {\mathbf{g}}_2^{-1},{\mathbf{h}}_3{\mathbf{g}}_3^{-1} \right)d {\mathbf{h}}\\&=\int_{ {\mathscr N}}  K\left ( {\mathbf{h}}_1,  {\mathbf{h}}_2,
{\mathbf{h}}_3{\mathbf{g}}_3 \right)
\breve{ \Phi}\left({\mathbf{h}}_1 {\mathbf{g}}_1^{-1},
{\mathbf{h}}_2 {\mathbf{g}}_2^{-1} , {\mathbf{h}}_3 \right)d {\mathbf{h}} \\
&=\int_{ {\mathscr N}} \Big [K \left( {\mathbf{h}}_1,  {\mathbf{h}}_2, {\mathbf{h}}_3{\mathbf{g}}_3\right) - K
\left( {\mathbf{h}}_1,  {\mathbf{h}}_2, {\mathbf{g}}_3 \right )\Big  ] \breve{ \varphi} ^{
(1)}
\left({\mathbf{h}}_1   {\mathbf{g}}_1^{-1}   \right)\breve{ {\varphi}} ^{
(2)} \left(
{\mathbf{h}}_2 {\mathbf{g}}_2^{-1}
\right)\breve{ {\varphi }} ^{(3)} \left(  {\mathbf{h}}_3\right)d {\mathbf{h}} .
 \end{split}  \end{equation*}since $ {{\varphi}} ^{
(3)} $ has mean value zero on $ {\mathscr H}_3$.  By applying the stratified mean value theorem of the integral form
 \cite{FS} and  the cancellation condition to   n.b.f.   $\breve{ {\varphi
 }} ^{
(1)} (\cdot{\mathbf{g}}_1^{-1})\breve{ {\varphi }} ^{
(2)} (\cdot {\mathbf{g}}_2^{-1} )$ on a bounded ball, we get
  \begin{equation*}\begin{split}
 \left|\Phi\tilde{ {*}} K\left( {\mathbf{g}}\right)\right|&\lesssim \int_{ {\mathscr H}_3}  \left| \breve{ {\varphi }} ^{
(3)} \left(
  {\mathbf{h}}_3\right)\right|\frac {\left\| {\mathbf{h}}_3\right\|}{\left\| {\mathbf{g}}_3\right\|^{Q_3+1}}d {\mathbf{h}}_3 \lesssim
 \frac {1}{\left\| {\mathbf{g}}_3\right\|^{Q_3+1}}.
 \end{split}  \end{equation*}

 If one of $\|\mathbf{g}_\mu\|$'s is   small and other two are large, without loss of generality, we can assume that $\|\mathbf{g}_1\|$ is small and
 $\|\mathbf{g}_2\|$, $\|\mathbf{g}_3\|$ are large. Then,  we have
 \begin{equation*}\begin{split}
 \left| \Phi\tilde{ {*}} K\left( {\mathbf{g}}\right)\right|=& \int_{ {\mathscr N}} \Big  [K \left( {\mathbf{h}}_1,
  {\mathbf{g}}_2  {\mathbf{h}}_2,  {\mathbf{h}}_3 {\mathbf{g}}_3\right) -K \left( {\mathbf{h}}_1,
  {\mathbf{g}}_2, {\mathbf{h}}_3{\mathbf{g}}_3 \right) -K \left( {\mathbf{h}}_1,
  {\mathbf{h}}_2{\mathbf{g}}_2,  {\mathbf{g}}_3\right)  + K \left( {\mathbf{h}}_1,  {\mathbf{g}}_2, {\mathbf{g}}_3 \right )\Big  ]
 \\& \qquad \cdot\breve{ {\varphi }} ^{
(1)} \left( {\mathbf{h}}_1 {\mathbf{g}}_1^{-1}\right)\breve{ {\varphi }} ^{
(2)} \left(
  {\mathbf{h}}_2\right)\breve{ {\varphi }} ^{
(3)} \left(  {\mathbf{h}}_3\right)d {\mathbf{h}}
 \end{split}  \end{equation*}as above, since $ {{\varphi}}^{(2)}$ and $ {{\varphi}}^{(3)}$ have mean value zero. By using the stratified mean
 value theorem  of the integral form twice and  the cancellation condition again, we get
  \begin{equation*}\begin{split}
 \left| K \tilde{ {*}} \Phi \left( {\mathbf{g}}\right)\right|&\lesssim \int_{ {\mathscr H}_3}\left | \breve{ {\varphi }} ^{
(2)} \left(
  {\mathbf{h}}_2\right)\breve{ {\varphi }} ^{
(3)} \left(  {\mathbf{h}}_3\right)\right|\frac
 {\left\| {\mathbf{h}}_2\right\|}{\left\| {\mathbf{g}}_2\right\|^{Q_2+1}} \frac
 {\left\| {\mathbf{h}}_3\right\|}{\left\| {\mathbf{g}}_3\right\|^{Q_3+1}}d {\mathbf{h}}_ 2d {\mathbf{h}}_3 \lesssim \frac
 {1}{\left\| {\mathbf{g}}_2\right\|^{Q_2+1}\left\| {\mathbf{g}}_3\right\|^{Q_3+1}}.
 \end{split}  \end{equation*}

 If $\|\mathbf{g}_\mu\|$'s      are all large,   we  use difference of $K$ of three order to get the estimate similarly.

The general case follows from $ \mathbf{Y}_1^{I_1} \mathbf{Y}_2^{I_2} \mathbf{Y}_3^{I_3}(\Phi_{\mathbf{r} } \tilde{ {*}}  K)=\Phi_{\mathbf{r} } \tilde{ {*}} \mathbf{Y}_1^{I_1} \mathbf{Y}_2^{I_2} \mathbf{Y}_3^{I_3} K$.
\end{proof}

\begin{lem}\label{lem:T-r-r'} Let  $\Phi_{\mathbf{r} } ({\mathbf{g}} ):=\varphi_{r_1}^{
(1)} ({\mathbf{g}}_1)  \varphi_{r_2}^{
(2)}  ({\mathbf{g}}_2)  \varphi_{r_3}^{
(3)} ({\mathbf{g}}_3) $ and $\Psi_{\mathbf{r} } ({\mathbf{g}} ):=\psi_{r_1}^{
(1)} ({\mathbf{g}}_1) \psi_{r_2}^{
(2)}  ({\mathbf{g}}_2)  \psi_{r_3}^{
(3)} ({\mathbf{g}}_3) $ with    $   {{\varphi}}^{(\mu)}$  and  $ {{\psi}}^{(\mu)}$ satisfying conditions in Proposition \ref{prop:decay}, $\mu=1,2,3$. Then
for $\mathbf{r } ,\mathbf{ r'}\in\mathbb{R}^3_+$,
\begin{equation*}
  \left |{\Phi}_{\mathbf{r} }\tilde *K  \tilde *  {\Psi}_{\mathbf{ r'} }\left( {\mathbf{g}}\right)\right|\lesssim \prod_{\mu=1}^3\left\{\left(\frac {r_\mu}{r_\mu'} \wedge \frac {r_\mu'}{r_\mu}\right)^{\frac 13}\frac {r_\mu}{\left(r_\mu+ \left\|
       {\mathbf{g}}_\mu\right\|\right)^{ Q_\mu+1}  }\right\} .
\end{equation*}
\end{lem}
 \begin{proof} At first, we   assume $ {r}_\mu'\leq {r}_\mu$. Without loss of generality, let $r_1'/r_1 =\min\{r_1' /r_1,r_2'/r_2 ,r_3' /r_3\} $. Then
 \begin{equation*}\begin{split}
    {\Phi}_{\mathbf{r} }\tilde *K \tilde  *  {\Psi}_{\mathbf{ r'} } (\mathbf{g} )&= \int_{\mathscr N} {\Phi}_{\mathbf{r} }\tilde *K   (\mathbf{g}  \mathbf{h}^{-1}  ) {\Psi}_{\mathbf{ r'} } \left( {\mathbf{h}}\right)d\mathbf{h} \\
    &=\int_{\mathscr N}\Big[ {\Phi}_{\mathbf{r} }\tilde *K \left (\mathbf{g}_1\mathbf{h}_1^{-1}, \mathbf{g}_2  \mathbf{h}_2^{-1} ,\mathbf{g}_3  \mathbf{h}_3^{-1} \right )- {\Phi}_{\mathbf{r} } \tilde*K  \left (\mathbf{g}_1, \mathbf{g}_2  \mathbf{h}_2^{-1} ,\mathbf{g}_3  \mathbf{h}_3^{-1} \right ) \Big]{\Psi}_{\mathbf{ r'} } \left( {\mathbf{h}}\right)d\mathbf{h} .
\end{split}\end{equation*}Note that on supp ${\Psi}_{\mathbf{ r'} }$, we have $\|\mathbf{h}_1\| \leq  r_1' \leq r_1  $  and $r_\mu +\|\mathbf{g}_\mu  \mathbf{h}_\mu^{-1}\|  \sim r_\mu +\|\mathbf{g}_\mu   \|$, $\mu=2,3$.
  Therefore,
  \begin{equation*}\begin{split}
  \Big|{\Phi}_{\mathbf{r} }\tilde *K  \left (\mathbf{g}_1\mathbf{h}_1^{-1}, \mathbf{g}_2  \mathbf{h}_2^{-1} ,\mathbf{g}_3  \mathbf{h}_3^{-1} \right )- {\Phi}_{\mathbf{r} }\tilde *K  \left (\mathbf{g}_1, \mathbf{g}_2  \mathbf{h}_2^{-1} ,\mathbf{g}_3  \mathbf{h}_3^{-1} \right )\Big| &\lesssim\frac {  r_1'}{ r_1+ \left\|
       {\mathbf{g}}_1\right\|   } \prod_{\mu=1}^3\frac {r_\mu}{\left(r_\mu+ \left\|
       {\mathbf{g}}_\mu\right\|\right)^{ Q_\mu+1}  }\\& \leq\prod_{\mu=1}^3 \left(  \frac {r_\mu'}{r_\mu}\right)^{\frac 13}\frac {r_\mu}{\left(r_\mu+ \left\|
       {\mathbf{g}}_\mu\right\|\right)^{ Q_\mu+1}  },
\end{split}\end{equation*}
by  using the stratified mean
 value theorem  and  Proposition \ref{prop:decay}.
  The general  case follows from
  \begin{equation*}
     \varphi_{r_\mu}^{(\mu)}\tilde *_\mu K\tilde *_\mu \psi_{r_\mu '}^{(\mu)}=\breve \psi_{r_\mu '}^{(\mu)}\tilde *_\mu
\breve K_\mu\tilde *_\mu   \breve\varphi_{r_\mu}^{(\mu)},
  \end{equation*}
   and applying the above argument. Here $\breve K_\mu (\mathbf{g}_{\mu},\mathbf{g}_{\mu'},\mathbf{g}_{\mu''}  )=K (\mathbf{g}_{\mu}^{-1},\mathbf{g}_{\mu'},\mathbf{g}_{\mu''}  )$ is still a convolution kernel, since the    size
estimates and the  cancellation conditions can be easily checked. Because in the    size
estimates \eqref{eq:size-estimates}  and the  cancellation conditions \eqref{eq:cancellation1}-\eqref{eq:cancellation3} for $K$, the left invariant vector fields  $Y_{\mu j}$'s can be replaced by the right ones.
\end{proof}

\begin{cor}  \label{cor:decay} Assume as in Lemma \ref{lem:T-r-r'} . The we have
   \begin{equation*}
      \left|{{\varphi}}_{\mathbf{r} }{*} K^\flat * \psi_{\mathbf{ r'} }\right|\left( {\mathbf{g}}\right)\lesssim \prod_{\mu=1}^3 \left(\frac {r_\mu}{r_\mu'} \wedge \frac {r_\mu'}{r_\mu}\right)^{\frac 13}\sum_{\mathbf{k}\in \mathbb{Z}_+^3} 2^{-|\mathbf{k}|}
      \chi_{\mathbf{2^{\mathbf{k}} r} }\left( {\mathbf{g}}\right),
   \end{equation*}
   where ${{\varphi}}_{\mathbf{r} }=\pi_*(\Phi_{\mathbf{r} })$, $\psi_{\mathbf{r}' }=\pi_*(\Psi_{\mathbf{r}' })$, $2^{\mathbf{k}} \mathbf{r}=(2^{k_1} {r}_1   , 2^{k_2} {r}_2   , 2^{k_3} {r}_3 )$ and
   $|\mathbf{k}|:=k_1+k_2+k_3$.
\end{cor}
\begin{proof} The estimate in Proposition \ref{prop:decay} implies
   \begin{equation}\label{eq:T-r-r'}
      \left|\Phi_{\mathbf{r} }  \tilde   {*} K_\varepsilon \tilde   * {\Psi}_{\mathbf{r'} }( {\mathbf{g}})\right|
     \lesssim\prod_{\mu=1}^3 \left(\frac {r_\mu}{r_\mu'} \wedge \frac {r_\mu'}{r_\mu}\right)^{\frac 13} \sum_{\mathbf{k}\in \mathbb{Z}_+^3} 2^{-|\mathbf{k}|}  {\chi}_{2^{k_1} {r}_1 }^{
(1)} ( {\mathbf{g}}_1)  { \chi}_{2^{k_2} {r}_2 }^{
(2)}   ( {\mathbf{g}}_2) {\chi}_{2^{k_3} {r}_3 } ^{
(3)} ( {\mathbf{g}}_3),
   \end{equation}where the implicit constant is independent of $\varepsilon$.
The result follows from $ {{\varphi}}_{\mathbf{r} } {*} K^\flat * \psi_{\mathbf{r'} } = \lim_{\varepsilon\rightarrow 0} {{\varphi}}_{\mathbf{r} } {*} K_\varepsilon^\flat * \psi_{\mathbf{r'} } =
\lim_{\varepsilon\rightarrow 0} \pi_*(\Phi_{\mathbf{r} } \tilde  {*} K_\varepsilon \tilde  *{\Phi}_{\mathbf{r'} })   $ and applying $\pi_*$ to
  both sides of  \eqref{eq:T-r-r'}.
\end{proof}

 \begin{proof}[Proof of Theorem \ref{thm:Kb}] We may assume $f\in\mathscr D(\mathscr N)$.  Denote $ \mathbf{rr'}:=(r_1r_1', r_2r_2',r_3r_3')$. Note that
 \begin{equation*}
    f*K_\varepsilon ^\flat (\mathbf{g})=\int_{\mathbb{R}^3_+} \int_{\mathbb{R}^3_+}  f *  {\varphi}_{\mathbf{r} }  *  {\psi}_{\mathbf{r} } *K_\varepsilon ^\flat  * \hat {\varphi}_{\mathbf{r r'} }  *
  \hat  {\psi}_{\mathbf{r r'} }(\mathbf{g})\frac
{d\mathbf{r}}{\mathbf{r}} \frac
{d\mathbf{r'}}{\mathbf{r'}}= \int_{\mathbb{R}^3_+} T_{ \mathbf{r}'  }f(\mathbf{g})
 \frac
{d\mathbf{r'}}{\mathbf{r'}}
 \end{equation*}by using the Calder\'on
reproducing formula \eqref{eq:reproducing} twice,
 where
 \begin{equation*}
    T_{ \mathbf{r}'  }f(\mathbf{g}) := \int_{\mathbb{R}^3_+}  f *  {\varphi}_{\mathbf{r} }  *  {\psi}_{\mathbf{r} } *K^\flat  *  \hat {\varphi}_{\mathbf{rr'} }  *
   \hat  {\psi}_{\mathbf{rr'} }(\mathbf{g})\frac
{d\mathbf{r}}{\mathbf{r}}.
 \end{equation*}
 Here we can choose ${\psi}$ and $\hat {\varphi} $ to have compact supports by Remark \ref{rem:compact-support}.  Arguing as \eqref{eq:g-converse}, we can find
 \begin{equation}\label{eq:T-g}\begin{split}
\|T_{  \mathbf{r}'  }f\|_p&  \lesssim   \left\|\left(\int_{\mathbb{R}^3_+}\left|  f *  {\varphi}_{\mathbf{r} }  *  {\psi}_{\mathbf{r} } *K^\flat  *  \hat {\varphi}_{\mathbf{ rr'} }   \right|^2\frac
{d\mathbf{r}}{\mathbf{r}}\right)^{\frac 12}\right\|_p.
 \end{split}\end{equation}
 But for $F=f*  {\varphi}_{\mathbf{r} }$,
\begin{equation*}\begin{split}
  | F*  \psi_{\mathbf{r} } *K_\varepsilon ^\flat *  \hat {\varphi}_{\mathbf{r r'} } (\mathbf{g})|&\lesssim C_{ \mathbf{r}' }\sum_{\mathbf{k}\in \mathbb{Z}_+^3} 2^{-|\mathbf{k}|} |F|*\chi_{2^k\mathbf{r} }\left( {\mathbf{g}}\right)  \lesssim  C_{  \mathbf{r}' } M_{ 3}\circ  M_{ 2} \circ M_{ 1}  (F)(\mathbf{g}),
 \end{split}\end{equation*}by Corollary \ref{cor:decay} and the estimate \eqref{eq:it-maximal-function}, where $C_{  \mathbf{r}' }=\prod_{\mu=1}^3 \left(\frac {1}{r_\mu'} \wedge  r_\mu' \right)^{\frac 13}$.
  Thus
\begin{equation*}\begin{split}
\text{R.H.S. of  } \eqref{eq:T-g}&  \lesssim C_{ \mathbf{r}' } \left\| \left(\int_{\mathbb{R}^3_+} \left|M_{ 3}\circ  M_{ 2} \circ M_{ 1} ( f*
{\varphi}_{\mathbf{r} })  \right|^2\frac {d\mathbf{r}}{\mathbf{r}}\right)^{\frac 12} \right\|_p \\&\lesssim C_{  \mathbf{r}' } \left\| \left(\int_{\mathbb{R}^3_+}  \left|  f*
{\varphi}_{\mathbf{r} }   \right|^2\frac {d\mathbf{r}}{\mathbf{r}}\right)^{\frac 12} \right\|_p
\lesssim C_{ \mathbf{r}' }\left\|  f \right\|_p
 \end{split}\end{equation*}
 by using   Fefferman-Stein  vector-valued maximal
inequality \cite{FeffS} and Theorem \ref{thm:g-function}. Then
\begin{equation*}\begin{split}
   \| f*K_\varepsilon ^\flat \|_p \leq  \int_{\mathbb{R}^3_+} \| T_{ \mathbf{r}'  }f\|_p\frac
{d\mathbf{r'}}{\mathbf{r'}}&
\lesssim \int_{\mathbb{R}^3_+}\prod_{\mu=1}^3 \left(\frac {1}{r_\mu'} \wedge  r_\mu' \right)^{\frac 13}\left\|  f \right\|_p
 \frac
{d\mathbf{r'}}{\mathbf{r'}}\lesssim  \left\|  f \right\|_p.
  \end{split}\end{equation*}by Minkowski's inequality.
  The theorem follows by taking $\varepsilon \rightarrow 0$.
  \end{proof}

  Theorem \ref{thm:Kb} can also be proved by using the general transference Theorem \ref{thm:transference}. We omit the details.

 \subsection{Cauchy-Szeg\H o  kernels as   flag-like convolution   kernels} Recall that
 the
Cauchy-Szeg\H o projection operator $\mathcal{P}$ on the Heisenberg group $\mathscr H_\mu$ can be written as $\mathcal{P}(\varphi)=\varphi*\widetilde{{S}}_\mu$
with
$\widetilde{{S}}_\mu=\lim_{\epsilon\rightarrow 0}\widetilde{S}_\mu ( \epsilon,\cdot)$ as a distribution \cite[Section 2.4 in chapter XII]{St93}, where $\widetilde{S}_\mu ( \epsilon,\cdot)$ is given by
\eqref{eq:flat-Szego2}.
Noting that
\begin{equation}\label{eq:S-mu-partial}
  \widetilde{S}_\mu ( \epsilon,\cdot)=\frac {c_\mu}{\mathbf{i} n_\mu} \frac {\partial}{\partial t_\mu }\frac {1}{ (|{\mathbf z}_\mu |^2 +\varepsilon_\mu- \mathbf{i} t_\mu )
  ^{n_\mu }    },
\end{equation}
and   $(|{\mathbf z}_\mu |^2 +\varepsilon_\mu- \mathbf{i} t_\mu ) ^{-n_\mu }\rightarrow
(|{\mathbf z}_\mu |^2  - \mathbf{i} t_\mu ) ^{-n_\mu }    $ in $L_{loc}^1(\mathscr H_\mu )$ as $\varepsilon_\mu\rightarrow0$,  we see that
\begin{equation*}
   \widetilde{{S}}_\mu=\frac {c_\mu}{\mathbf{i} n_\mu}\frac {\partial}{\partial t_\mu }\frac {1}{ (|{\mathbf z}_\mu |^2  - \mathbf{i} t_\mu ) ^{n_\mu }    }
\end{equation*} as a distribution, since taking partial derivatives is continuous on $\mathscr  D'(\mathscr H_\mu)$.
So the
Cauchy-Szeg\H o  kernel $\widetilde{{S}}_\mu$ is a distribution smooth outside the origin and   homogeneous of  degree $-Q_\mu$ on the Heisenberg group $\mathscr H_\mu$. Thus   \cite[Theorem 6.13]{FS} implies that
\begin{equation*}
   \widetilde{{S}}_\mu=p.v. \widetilde{{S}}_\mu  +c\delta_{0}
\end{equation*}for some constant $c$,  and $\widetilde{{S}}_\mu$ also satisfies the {\it strong cancellation condition}:
\begin{equation}\label{eq:cancellation-int}
   \int_{\|\mathbf{h}_\mu\|=1}\tilde{S}_\mu(  \mathbf{h}_\mu)d\sigma=0,
\end{equation}where $d\sigma$ is a measure on the sphere $\|\mathbf{h}_\mu\|=1$ satisfying $\int_{\mathscr H_\mu}f(\mathbf{h}_\mu)d\mathbf{h}_\mu= \int_0^\infty r^{Q-1}dr\int_{\|\mathbf{h}_\mu\|=1}f( r \mathbf{h}_\mu)d\sigma$ for any $f\in L^1(\mathscr H_\mu)$.

 It is direct to check   $\widetilde{{S}}_\mu$  satisfying the size
estimates \eqref{eq:size-estimates} by differentiating the expressions \eqref{eq:flat-Szego2} for $ \varepsilon={0}$. For the  cancellation conditions
\eqref{eq:cancellation3} for $ \widetilde{{S}}_\mu$,  note that for a n.b.f.  $\phi$,
\begin{equation}\label{eq:int-cancellation}
  p.v. \int_{\mathscr H_\mu}\widetilde{{S}}_\mu ( \mathbf{g}_\mu )\phi(\delta \mathbf{g}_\mu  )d  \mathbf{g}_\mu =   \int_{\mathscr H_\mu}\widetilde{{S}}_\mu (
  \mathbf{g} _\mu ) (\phi( \delta\mathbf{g}_\mu  )- \phi( \mathbf{0}_\mu  ))d  \mathbf{g}_\mu ,
\end{equation}by the strong cancellation condition \eqref{eq:cancellation-int}.
Since $| \phi( \delta\mathbf{g}_\mu  )- \phi( \mathbf{0}_\mu  )|\lesssim \|\delta\mathbf{g}_\mu\| $ by the stratified mean value theorem again, the integral
in the right hand side of \eqref{eq:int-cancellation} converges and is  uniformly bounded by homogeneity of $\widetilde{S}_\mu $.

 Similarly, the
Cauchy-Szeg\H o projection operator $\mathcal{P}$  on $\tilde{\mathscr  N}$ can be written as $\mathcal{P}(\varphi)=\varphi*\widetilde{{S}} $ with
$\widetilde{{S}} =\lim_{\boldsymbol \varepsilon\rightarrow 0}\tilde{S}_{\boldsymbol \varepsilon} $  as a distribution and
 \begin{equation}\label{eq:S-mu-partial-2}
\widetilde{{S}}_{\boldsymbol \varepsilon} =\frac {\partial^3\widetilde{\mathscr {S}}_{\boldsymbol \varepsilon}}{\partial t_1\partial t_2 \partial t_3 } \qquad \text{with}\qquad  \widetilde{\mathscr {S}}_{\boldsymbol \varepsilon}=\prod_{\mu=1}^3\frac {   c_\mu}{\mathbf{i} n_\mu }\frac {   1}{ (|{\mathbf z}_\mu |^2+\varepsilon_\mu  -
   \mathbf{i} t_\mu ) ^{n_\mu }    }
\end{equation}
by the expression of $\tilde{S}_{\boldsymbol \varepsilon}$ in \eqref{eq:flat-Szego-tilde}-\eqref{eq:flat-Szego2}.
  Thus the Cauchy-Szeg\H o  kernel $\widetilde{{S}}$  is a distribution  given by \eqref{eq:S-mu-partial-2} with $\boldsymbol \varepsilon=\mathbf{0}$. It is also  the tensor product $\widetilde{{S}}_1\otimes \widetilde{{S}}_2\otimes \widetilde{{S}}_3$ of distributions.
  It is obviously smooth  away from the subgroups $\mathscr H_ \mu\times\mathscr H_\nu$  for all $\mu\neq\nu$, where     $\widetilde{{S}}( \mathbf{h}
  )=\prod_{\mu=1}^3\widetilde{{S}}_\mu (  \mathbf{h}_\mu)$.  Thus,  the size
estimates \eqref{eq:size-estimates} and the cancellation conditions   \eqref{eq:cancellation1}-\eqref{eq:cancellation3} for $\widetilde{{S}}$ follows from those of  $\widetilde{S}_\mu $ on the Heisenberg group $\mathscr H_\mu$, and so
it is a standard convolution kernel  on $\tilde{\mathscr  N}$.

\begin{prop}\label{prop:CS-flag-like}
   The  Cauchy-Szeg\H o  kernel   $S$  on $ {\mathscr  N}$ is exactly $\tilde  { S}^\flat$, i.e. it is a   flag-like   kernel.
\end{prop}
\begin{proof}As above, the
Cauchy-Szeg\H o projection operator $\mathcal{P}$  on $ {\mathscr  N}$ can be written as $\mathcal{P}(\varphi)=\varphi* S  $ with
$ S =\lim_{\boldsymbol \varepsilon\rightarrow 0}  S_{\boldsymbol\varepsilon}( \mathbf{h} ) $ to be the  distribution
 \begin{equation*}
 \frac { n_1 ! n_2 ! }{ 4^2(\frac \pi
 2)^{N+2}\mathbf{i}^{n_3  } }
    \left (\frac {\partial  }{\partial t_1  } + \frac {\partial  }  {   \partial t_2  }\right)^{n_3 }    \frac {1}{ (|{\mathbf z}_1 |^2+|{\mathbf z}_3 |^2  - \mathbf{i}  t_1  ) ^{n_1 +1}    }  \cdot \frac {1}{
      (|{\mathbf z}_2 |^2+|{\mathbf z}_3 |^2 - \mathbf{i}  t_2  ) ^{n_2+1 }    }
\end{equation*} as the partial derivative of a local integrable function on ${\mathscr  N}$,
where   $ S_{\boldsymbol\varepsilon}( \mathbf{h} )$ is given by  \eqref{eq:flat-Szego-tilde}. Thus for $\varphi\in \mathscr  D( {\mathscr N})$
 \begin{equation*}\begin{split} \langle S,\varphi\rangle&=\lim_{\epsilon\rightarrow 0}\langle S_{\epsilon,\epsilon },\varphi \rangle=\lim_{\epsilon\rightarrow 0}\left\langle \pi_*(   \tilde{ S}_   {\epsilon,\epsilon,\epsilon }),\varphi\right \rangle\\
 &=\lim_{\epsilon\rightarrow 0}\pi_*\left( \breve{\varphi}^\sharp \tilde{ * }\tilde{ S}_   {\epsilon,\epsilon,\epsilon }  \right)(\mathbf{0})= \pi_* \left( \breve{\varphi}^\sharp \tilde{ * }\tilde{ S}\right )  (\mathbf{0})=\left\langle  \tilde  { S}^\flat   ,\varphi\right \rangle,
 \end{split} \end{equation*}by using \eqref{eq:phi} and $
      \pi_*(   \tilde{ S}_{\boldsymbol \varepsilon}) = S_{\varepsilon_1+\varepsilon_3,\varepsilon_2+\varepsilon_3 }
 $ in  Subsection \ref{Transference-CS}. In the fourth identity, we use $\breve{\varphi}^\sharp \tilde{ * }\tilde{ S}_   {\epsilon,\epsilon,\epsilon }=\frac {\partial^3 \breve{\varphi}^\sharp }{\partial t_1\partial t_2 \partial t_3 }\tilde{ * } \widetilde{\mathscr {S}}_   {\epsilon,\epsilon,\epsilon }$ and take limit.
Thus,  $S=\tilde  { S}^\flat$.
\end{proof}

\begin{rem}
   It is interesting to  characterize flag-like convolution kernels directly  by suitable  size
estimates and the cancellation conditions on the group $ {\mathscr  N}$, as the characterization of flag kernels given by M\"uller-Ricci-Stein \cite{MRS}.
\end{rem}

 \section{Tiles, shards  and partitions of   $\mathscr  N$ }

 \subsection{Tiling   of   the Heisenberg  group } The Heisenberg subgroup
$\mathscr H ^\nu$  is $\mathbb{ R}^{2\nu} \times\mathbb{ R}  $ with the  multiplication    given by
\begin{equation}\label{eq:multiplicationH}
   (y  ,s ) (  y ' ,s  ')= \left( y  +  y ', s +s '+   B(y,y')
       \right),\qquad B(y,y')=\sum_{l=1}^{ \nu} (y_{2l-1} 'y_{2l }-y_{2l-1}y_{2l }'),
\end{equation}where $y ,y ' \in\mathbb{ R}^{2\nu} $, $ s  ,s  '\in \mathbb{R}$.
Let
$\mathscr H_{\mathbb{Z}}^\nu$
be the discrete subgroup
$
   \left\{(y, t)\in \mathscr H^\nu ;  y \in \mathbb{Z}^{2\nu},  t \in  \mathbb{Z}\right \}.
$

 Recall  \cite{St,Ty} that the {\it basic tile}  is
\begin{align}\label{basic tile}
 {T}_o:=\left\{(  y , t)\in\mathscr H^{\nu}\mid   y\in [0,1)^{2\nu}, ~ f _o(  y)\leq t<f_o (  y)+1   \right\},
\end{align}
where
\begin{equation}\label{eq:basic tile-f}
   f_o (  y)=\sum_{m=1}^\infty\frac 1{ 4^m}B \big( \left[ 2^m  y\right]~{\rm mod} ~2, \langle2^m  y\rangle\big)
\end{equation}is a continuous function.
Here [ , ] and $\langle\, ,\rangle$ denote the integer and fractional part functions, respectively,  interpreted componentwisely  ($[y]_j=[y_j]$, etc.),
and $B (  y,  y') $ is defined in \eqref{eq:multiplication3}.

   $\mathscr H^{\nu}=\cup_{g\in \mathscr H^{\nu}_{\mathbb Z}}~ g({T}_o)$ is a (disjoint) tiling of $\mathscr H^{\nu}$ and it is self-similar, i.e.
   \begin{equation*}
  \delta_{2}({T}_o)=\bigcup_{g\in\Gamma_o}  g {T}_o
\end{equation*}
where $\Gamma_o=\left\{(y, t): y_j=0,1, t=0, 1, 2,   3\right\}.$
Set
\begin{equation*}
  {  \mathfrak{T}}_0=\{g  {T}_o; g \in\mathscr H_{\mathbb{Z}}^\nu\},\qquad   {\mathfrak{T}}_j
 :=\delta_{2^j} {\mathfrak{T}}_0,\qquad   {\mathfrak{T} } =\bigcup_{j\in \mathbb{Z}}  {\mathfrak{T}}_j.
\end{equation*}
   An element $ T \in  {\mathfrak{T}}$ is called a {\it tile}, and   an element of $    {\mathfrak{T}}_j$ is called a {\it tile of scale $j$}.

\begin{thm}\label{thm:Tiles} \cite{St,Ty}
Let $\mathfrak{T}_j$ and $\mathfrak{T}$ be defined as above.
Then the following hold:
\begin{enumerate}
  \item for each $j \in \mathbb Z$, $\mathfrak{T}_j$ is a partition of $\mathscr H^{\nu}$, that is, $\mathscr H^{\nu} = \bigcup_{T \in \mathfrak{T}_j} T$;
  \item $\mathfrak{T}$ is nested, that is, if $T, T' \in \mathfrak{T}$, then either $T$ and $T'$ are disjoint or one is a subset of the other;
  \item for each $j \in \mathbb Z$ and $T\in\mathfrak{T}_j$, $T$ is a union of $2^{ 2\nu+2}$ disjoint congruent subtiles in $\mathfrak{T}_{j-1}$;
  \item for each $T \in \mathfrak{T}_j$, there exists $\xi\in T$  such that $B(\xi, C_1 2^j) \subseteq T \subseteq B(\xi, C_2  2^j)$, where   the
      constants $C_1$ and $C_2$ depend only on $\nu$;
  \item if $T \in \mathfrak{T}_j$, then $ g(T) \in \mathfrak{T}_j$ for all $g \in \delta_{2^j} (\mathscr H^{\nu}_{\mathbb Z})$, and $\delta_{2^k} (T) \in
      \mathfrak{T}_{j+k}$ for all $k \in \mathbb Z$.
\end{enumerate}
\end{thm}

  Since
every tile is a dilate and translate of the basic tile $ {T}_o$, they all have similar geometry. Each tile in $\mathfrak{T}_j$ has fractal boundary
and approximates a Heisenberg ball of radius $2^{j}$.

 \subsection{Shards  in   $\mathscr N $ }
At first, note that
\begin{equation}\label{basic tile2}
   T  _o =  \Big \{(y,t): y\in [0,1)^{2\nu},  {t}\in
    [ f_o (y), f_o (y)+1) \Big \},
\end{equation}is   a twisted  product in the sense that for a fixed $y$, points $t$ belong
 to the interval $f_o (y)+[0,  1)$.
But we only have the following estimate
\begin{equation}\label{eq:f}
 -\nu< f_o (y)< \nu,
\end{equation}for $y\in [0,1)^{2\nu}$, by definition \eqref{eq:basic tile-f}. Thus, the position of the interval $f_o (y)+[0,  1)$ may change dramatically as $y$.
So we consider the following union of  translates of $ T _o$ along $\mathbb{R}^1$:
\begin{equation}\label{eq:shard0}
    {S}  _o:=\bigcup_{g\in\Gamma}    g T_o , \qquad{\rm with }\qquad \Gamma=   \left\{( \mathbf{ {0} } _y,
      m );m= 0  ,\ldots,2^{\kappa}-1  \right\},
\end{equation}where the positive integer $ {\kappa}   $ will be chosen latter. By    \eqref{basic tile2},
 we see that
\begin{equation}\label{eq:I-z-m}
   {S}   _o = \left \{(y,t): y\in [0,1)^{2\nu},  {t}\in
     {I}_{y  } \right\}
 ,
 \end{equation} with
 \begin{equation}\label{eq:I-y}
    {I}_{y  }:=\bigcup_{m=0 }^{ 2^{\kappa} -1}\left[f_o (y)+m  , f_o (y)+  m+1  \right) = f_o (y)+ \left[ 0 ,    2^{\kappa}   \right).
 \end{equation}
If we take $\kappa$ so large, we can assume $\nu<2^{\kappa-10}$, and so $|f_o (y)|$ is small compared to the length $2^{\kappa}$. Thus,
$f_0(y)+
 \left[  0 ,   2^{\kappa}   \right)
$ is a small translate of $[0,2^{\kappa } )$.
Let
$\mathscr H_{\mathbb{Z};\kappa}^\nu$
be the discrete subgroup
$   \mathbb{Z}^{2\nu} \times 2^{\kappa }\mathbb{Z}
$, and let\begin{equation*}
  {  \mathfrak{S}}_0=\{g  {S}_o; g \in\mathscr H_{\mathbb{Z};\kappa}^\nu\},\qquad   {\mathfrak{S}}_j
 :=\delta_{2^j} {\mathfrak{S}}_0,\qquad   {\mathfrak{S} } =\bigcup_{j\in \mathbb{Z}}  {\mathfrak{S}}_j.
\end{equation*}

\begin{cor}\label{cor:shard}
For each $j \in \mathbb Z$, $\mathfrak{S}_j$ is a partition of $\mathscr H^{\nu}$, that is, $\mathscr H^{\nu} = \bigcup_{S \in \mathfrak{S}_j} S$.
   \end{cor}
\begin{proof}
Since $\mathscr H^{\nu}=\bigcup_{g\in  \mathscr H_{\mathbb{Z} }^\nu}~ g T_o $ is a (disjoint) tiling of $\mathscr H^{\nu}$ by Theorem \ref{thm:Tiles}, we have
 \begin{equation*}\begin{split}
 \bigcup_{g\in \mathscr H_{\mathbb{Z};\kappa}^\nu} gS_o&= \bigcup_{g \in \mathscr H_{\mathbb{Z};\kappa}^\nu}   \bigcup_{g'\in \Gamma}    g g'  T_o
 =\bigcup_{h\in \mathscr H_{\mathbb{Z}}^\nu}   h  T_o,
\end{split} \end{equation*}
 by $\mathscr H_{\mathbb{Z}}^\nu=\mathscr H_{\mathbb{Z};\kappa}^\nu \cdot\Gamma$. The result for general $j$ follows by dilations.
\end{proof}

   Denote by $ {\mathfrak{T}}_{j_\mu}( {\mathscr H}_\mu)$ the set of
 tiles   of scale $j_\mu$ in $ {\mathscr H}_\mu$,   $\mu=1,2,3$.
    Note that natural tiles in the group $ \tilde{{\mathscr N }}$ are products
$
    {T}_1\times  {T}_2\times  {T}_3
$
with $ T_\mu\in  {\mathfrak{T}}_{ j_\mu}({\mathscr H}_\mu)$. For $\mathbf{j}=(j_1,j_2,j_3)\in \mathbb{Z}^3$, we assume $j_2\leq j_1$ without loss of generality.
Their
image under the projection of $\pi$ are
 \begin{equation*}
  \pi\left[  {T}_1\times  {T}_2\times  {T}_3\right]\approx
\Box_{ \mathbf{j }}^+\times P_{2  ^{ 2 j_1} , 2_2 ^{ 2 j_2} \vee 2  ^{ 2j_3}},
\end{equation*}  by Proposition \ref{prop:tube-pi} and Theorem \ref{thm:Tiles} (4), where
\begin{equation*}
   \Box_{ \mathbf{j }}^+:= \Box_{  j_1 }^{(1)+}\times \Box_{ j_2 }^{(2)+}\times \Box_{  j_3 }^{(3)+},\qquad \Box_{  j_\mu }^{(\mu)+}:=[0,2^{j_\mu})^{2n_\mu}.
\end{equation*}
So to obtain a partition of the group $ {\mathscr N }$ of scale $\mathbf{j}$, we consider
   a  product  of the form
\begin{equation*}
   {\mathcal{S}}  _o:=      {S}^{(1)}_{ o}\times   {S}^{(2)}_{ o}\times \Box_{0
      }^{(3)}=\left\{(\mathbf{z},\mathbf{t}): \mathbf{z}\in \Box_{ \mathbf{0 }}^+, \mathbf{t}\in  {I}_{\mathbf{z} } \right \},
 \end{equation*}
where  $ {S}^{(\mu)}_{ o}$ is the  union of tiles in $ {\mathscr H}_\mu$ defined by   \eqref{eq:shard0}, $\mu=1,2$, and
 \begin{equation*}
    {I}_{\mathbf{z} }:=(f _o(\mathbf{z}_1), f _o(\mathbf{z}_2))+ \left[0 ,
   2^{\kappa}  \right)\times  \left[0 ,
   2^{\kappa}  \right)
 \end{equation*}by \eqref{eq:I-y}.
  It is   a twisted  product in the sense that for a fixed $\mathbf{z}$, points $\mathbf{t}$ belong
 to a square ${I}_{\mathbf{z} }$,  whose   position depends on $\mathbf{z}$.
 Its translates
 under the lattice
$
2^{\kappa } \mathbb{Z} \times 2^{\kappa } \mathbb{Z}
$
   in $\mathbb{R}^2$ constitute a partition of
 $
    \Box_{ \mathbf{j }}^+\times \mathbb{R}^2.
 $

The corresponding object of scale $\mathbf{j}$  is
 \begin{equation}\label{eq:twisted-product}
   {\mathcal{S}}^{\mathbf{j}}_o:=    \delta_{2^{j_1 }} {S}^{(1)}_{ o}\times \delta_{2^{j_2 }} {S}^{(2)}_{ o}\times \Box_{
   j_3 }^{(3)+}=\left\{(\mathbf{z},\mathbf{t}): \mathbf{z}\in \Box_{ \mathbf{j }}^+, \mathbf{t}\in \hat  {I}_{\mathbf{z},   \mathbf{j}} \right \},
 \end{equation}
 with
 \begin{equation}\label{eq:twisted-product2}
  \hat  {I}_{\mathbf{z},  \mathbf{j}}:=(2  ^{ 2j_1}  f_o(2  ^{ - j_1}\mathbf{z}_1),  2  ^{ 2j_2}  f_o(2  ^{ - j_2}\mathbf{z}_2))+ 2^{2j_1+\kappa} \left[0 ,
   1\right)\times  2^{2j_2+\kappa} \left[0 ,
  1 \right).
 \end{equation}
  It is also a twisted  product in the above sense, and its translates
 under the lattice
 \begin{equation}\label{eq:lattice-R-2}
 2^{2j_1+\kappa }  \mathbb{Z} \times 2^{2j_2+\kappa }  \mathbb{Z}
 \end{equation}
   in $\mathbb{R}^2$ constitute a partition of
 $
    \Box_{ \mathbf{j }}^+\times \mathbb{R}^2
$.  Then, translates of
  $  \Box_{ \mathbf{j }}^+\times \mathbb{R}^2$ under  the   lattice
 \begin{equation*}
    2^{\mathbf{j}  } \mathbb{Z}^{2\mathbf{n} }:=2^{j_1 } \mathbb{Z}^{2n_1}\times 2^{j_2 } \mathbb{Z}^{2n_2 } \times 2^{j_3 } \mathbb{Z}^{ 2n_3}
 \end{equation*}
   constitute a partition of the group $ {\mathscr N} $. Namely,    translates of ${\mathcal{S}}^{\mathbf{j}}_o$ of scale $\mathbf{j}$ under
\begin{equation}\label{eq:G-j}
   2^{\mathbf{j}  } \mathbb{Z}^{2\mathbf{n} }\times  2^{2j_1 +\kappa } \mathbb{Z} \times
    2^{2j_2 +\kappa }   \mathbb{Z}
\end{equation} constitute a partition of the group $ {\mathscr N} $.

Now let us construct the basic shard as the union of translates of ${\mathcal{S}}^{\mathbf{j}}_o$ so that its shape is comparable to the tube $T(\mathbf{0},2^{\mathbf{j}
} )$,  and their translates   constitute a partition $ {\mathscr N} $. Without
loss of generality, we assume  $j_2  < j_1$.

{\it Case  i:  $j_3<j_2   $}.  We define the {\it basic shard $ \mathcal{R}^{\mathbf{j} }_o $ of scale $\mathbf{j}$}  to be the
union of $4$ translates of $ {\mathcal{S}}^{\mathbf{j}}_o$:
\begin{equation}\label{eq:R0-j0}
     \mathcal{R}^{\mathbf{j} }_o:=\bigcup_{ m_1,m_2=0,-1 } \left(\mathbf{0}_{\mathbf{z}},m_12^{2j_1+\kappa  }, m_2 2^{2j_2+\kappa  }\right)
     {\mathcal{S}}^{\mathbf{j}}_o=\left\{(\mathbf{z},\mathbf{t}): \mathbf{z}\in \Box_{ \mathbf{j }}^+, \mathbf{t}\in  {I}_{\mathbf{z},   \mathbf{j}}  \right \},
 \end{equation}so that ${I}_{\mathbf{z},  \mathbf{j}}$ contain  the origin as an inner point,
where
 \begin{equation}\label{eq:I-zj'}
    {I}_{\mathbf{z},  \mathbf{j}} :=(2  ^{ 2j_1}  f_o(2  ^{ - j_1}\mathbf{z}_1),  2  ^{ 2j_2}  f_o(2  ^{ - j_2}\mathbf{z}_2))+ 2^{2j_1+\kappa}\left[ - 1 ,
1  \right)\times  2^{2j_2+\kappa}\left[ -1 ,1
  \right).
 \end{equation}
It is essential a cuboid
 $
    \mathcal{R}^{\mathbf{j} }_o\approx \Box_{ \mathbf{j }}^+\times2^{2j_1 +\kappa} (- 1, 1)\times 2^{2j_2 +\kappa }(- 1, 1).
$
A {\it shard of scale $\mathbf{j}$} is
  the  translate of $ \mathcal{R}^{\mathbf{j} }_o $  under an element of the group
\begin{equation}\label{eq:G-j-1}G_{\mathbf{j}}:=
   2^{\mathbf{j}  } \mathbb{Z}^{2\mathbf{n} }  \times 2^{2j_1 +\kappa } 2\mathbb{Z} \times
    2^{2j_2 +\kappa }   2\mathbb{Z}.
\end{equation}

{\it Case  ii: $j_2 <j_3< j_1 $}.  The {\it basic shard  of scale $\mathbf{j}$}  is the union of some translates of $ {\mathcal{S}}^{\mathbf{j}}_o$:
\begin{equation}\label{eq:R0-j}
     \mathcal{R}^{\mathbf{j} }_o:=\bigcup_{\substack {m_1=0,-1,\\ m_2\in    \Gamma_1} } \left(\mathbf{0}_{\mathbf{z}}, m_12^{2j_1+\kappa  }, m_22^{2j_2+\kappa
     }\right)  {\mathcal{S}}^{\mathbf{j}}_o.
\end{equation}
where
\begin{equation*}
     \Gamma_1:=   \Big\{  - 2^{2(j_3 -j_2) } , \cdots, -1,0,1 \cdots,2^{2(j_3 -j_2)} -1\Big\}
\end{equation*}with $\# \Gamma_1=2\cdot 2^{2(j_3 -j_2)}$.
  Then
 $
   \mathcal{R}^{\mathbf{j} }_o = \{(\mathbf{z},\mathbf{t}): \mathbf{z}\in \Box_{ \mathbf{j }}^+, \mathbf{t}\in     {{I}}_{\mathbf{z},   \mathbf{j} }'  \}
$
 with
 \begin{equation*}\label{eq:I-z-j'}\begin{split}
    {{I}}_{\mathbf{z},   \mathbf{j} }':= &\left\{2  ^{ 2j_1}  f_o(2  ^{ - j_1}\mathbf{z}_1)+ \left[- 2^{2j_1+\kappa} ,
     2^{2j_1+\kappa}  \right)\right\} \times\bigcup_{ m_2\in    \Gamma_1} \left\{ 2  ^{ 2j_2}f_o(2  ^{ - j_2}\mathbf{z}_2)+\left[ m_2   2^{2j_2+ \kappa }  ,
   ( m_2+1)2^{2j_2+ \kappa }  \right) \right\} \\
= &( 2  ^{ 2j_1}  f_o(2  ^{ - j_1}\mathbf{z}_1),  2  ^{ 2j_2}f_o(2  ^{ - j_2}\mathbf{z}_2))+ 2^{2j_1+\kappa}\left[ - 1 ,
1  \right)\times  2^{2j_3+\kappa}\left[ -1 ,1
  \right).\end{split} \end{equation*}$\mathcal{R}^{\mathbf{j} }_o$ is essential a cuboid
$
    \mathcal{R}^{\mathbf{j} }_o\approx \Box_{ \mathbf{j }}^+\times 2^{2j_1+\kappa}(-1, 1)\times 2^{2j_3+\kappa }(-1, 1).
$
Since
$
      2^{2j_2 +\kappa}   \mathbb{Z} =2^{2j_3 +\kappa} 2 \mathbb{Z} \cdot  2^{2j_2 +\kappa}\Gamma_1, $
we see that
\begin{equation*}
    \left(\mathbf{0},  2 {m }_1 2^{2j_1+\kappa  },2 {m }_2 2^{2j_3+\kappa }\right) \mathcal{R}^{\mathbf{j} }_o,
\end{equation*} $ m_1,m_2 \in \mathbb{Z } ,$  are disjoint and their union is  $\Box_{ \mathbf{j }}\times \mathbb{R}^2$. The translate of $ \mathcal{R}^{\mathbf{j} }_o $ under an element of the group
\begin{equation}\label{eq:G-j-2}G_{\mathbf{j}}:=
   2^{\mathbf{j}  } \mathbb{Z}^{2\mathbf{n} } \times  2^{2j_1 +\kappa } 2\mathbb{Z} \times
    2^{2j_3 +\kappa }   2\mathbb{Z}
\end{equation}is called
a  {\it shard  of scale $\mathbf{j}$}. We will see that they constitute  a partition $ {\mathscr N} $.

{\it Case iii:    $j_1<j_3 $}. In this case,  $j_3 $ is very large. As in \eqref{eq:R0-j},   we first define
\begin{equation*}
   \widehat{  {\mathcal{R}}}^{\mathbf{j} }_o:=\bigcup_{\substack {m_1=0,-1,\\ m_2\in    \Gamma_2} }  \left(\mathbf{0}_{\mathbf{z}},m_12^{2j_1+\kappa  },
   m_22^{2j_2+\kappa  }\right)  {\mathcal{S}}^{\mathbf{j}}_o,
\end{equation*}
where
\begin{equation*}
     \Gamma_2:=   \left\{   - 2^{2(j_1 -j_2)} , \cdots, -1,0,1 \cdots,  2^{2(j_1 -j_2)} -1\right\}.
\end{equation*}
Then as in the Case  $ii$, we have
 $
  \widehat{ \mathcal{R}}^{\mathbf{j} }_o =\{(\mathbf{z},\mathbf{t}): \mathbf{z}\in \Box_{ \mathbf{j }}^+, \mathbf{t}\in     {{I}}_{\mathbf{z},   \mathbf{j} }''  \}
$
 with
 \begin{equation}\label{eq:twisted-product3-I}
    {{I}}_{\mathbf{z},   \mathbf{j} }''= ( 2  ^{ 2j_1}  f_o(2  ^{ - j_1}\mathbf{z}_1),  2  ^{ 2j_2}f_o(2  ^{ - j_2}\mathbf{z}_2))+ 2^{2j_1+\kappa}\left[-1 ,
     1  \right)\times 2^{2j_1+\kappa}\left[-1 ,
     1  \right)  .
 \end{equation}
 $ \widehat{\mathcal{R}}^{\mathbf{j} }_o$ is a small translate of  the  cuboid $    \Box_{ \mathbf{j }}^+\times 2^{2j_1+\kappa}\left(-1 ,
     1  \right)\times 2^{2j_1+\kappa}\left(-1 ,
     1  \right) .
$

Now consider the union of $\widehat{ \mathcal{R}}^{\mathbf{j} }_o$ with its two neighbors  in $t_1$ direction:
\begin{equation}\label{eq:hat-R-o}
   \check{ { \mathcal{R}}}^{\mathbf{j} }_o: =\left(\mathbf{0}_{\mathbf{z}},  -2\cdot 2^{2j_1+\kappa  }, 0\right) \widehat{ \mathcal{R}}^{\mathbf{j} }_o \bigcup \widehat{
   \mathcal{R}}^{\mathbf{j} }_o \bigcup \left(\mathbf{0}_{\mathbf{z}},  2\cdot 2^{2j_1+\kappa   }, 0\right)\widehat{\mathcal{R}}^{\mathbf{j} }_o.
\end{equation}
Then we further consider the union of its translate  along the diagonal: \begin{equation}\label{eq:R-o}
     \mathcal{R}^{\mathbf{j} }_o:=\bigcup_{ m \in  \Gamma_3}  \left(\mathbf{0}, m 2^{2j_1 +\kappa},  m2^{2j_1+\kappa }\right)
    \check{{\mathcal{R}}}^{\mathbf{j} }_o,
\end{equation}which is called the {\it basic shard  of scale $\mathbf{j}$},  where
\begin{equation}\label{eq:Gamma-3}
   \Gamma_3:=   \left\{   - 2^{2(j_3 -j_1)}  , \cdots,-2,0,2 , \cdots,  2^{2(j_3 -j_1)}-2 \right\}.
\end{equation}
\begin{figure}[h]
  \centering
  \includegraphics[width=1\textwidth]{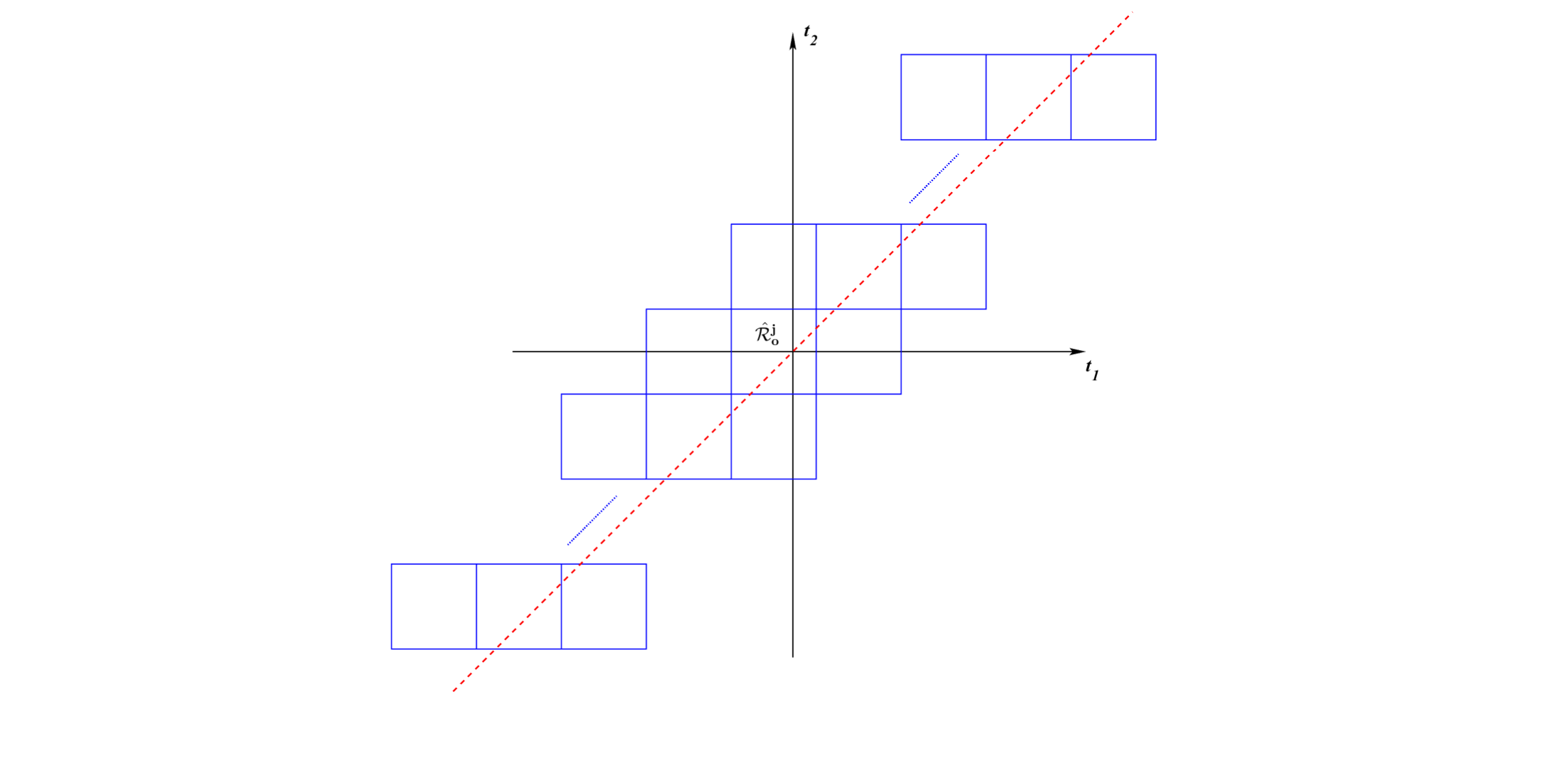}
  \caption{$\mathbf{t}$ slices of $\mathcal{R}^{\mathbf{j}}_{o}$}
  \label{fig:fig3}
\end{figure}

\noindent
The  translates of $\mathcal{R}^{\mathbf{j} }_o $ under an element of
 \begin{equation}\label{eq:G-j-3}G_{\mathbf{j}}:=
   2^{\mathbf{j}  } \mathbb{Z}^{2\mathbf{n} } \times  2^{2j_1 +\kappa } 6\mathbb{Z} \times
2^{2j_3 +\kappa }    2\mathbb{Z}
\end{equation}
is called a {\it shard  of scale $\mathbf{j}$} (cf. Figure 2). We will see they constitute  a partition $ {\mathscr N} $. Now denote the set of  shards of scale $\mathbf{j} $ by
$
   \mathfrak{R }_{\mathbf{j}}:=\left\{ \mathbf{g}\mathcal{R}^{\mathbf{j} }_o;\mathbf{g}\in
G_{\mathbf{j}}\right\} ,
$
 and the set of all shards by
$
   \mathfrak{R }:=\bigcup_{\mathbf{j}\in \mathbb{Z} ^3}\mathfrak{R }_{\mathbf{j}}.
$

 \begin{lem}\label{prop:shard-P} We have
     \begin{equation}\label{eq:shard-P}\Box_{ \mathbf{j }}^+\times  P_{2^{2j_1+\kappa-3  } ,2^{2j_3+\kappa-3  } }
 \subset   \mathcal{R}^{\mathbf{j} }_o\subset \Box_{ \mathbf{j }}^+\times  P_{ 2^{2j_1+\kappa+3}, 2^{2j_3+\kappa+3 }}.
\end{equation}
  \end{lem}
\begin{proof} We only show the most complicated case $iii$. It is similar for   other cases.  Recall  the estimate
 \begin{equation}\label{eq:estimate-f}
   |f_o(2  ^{ - j_\mu}\mathbf{z}_\mu)| <2^{\kappa-10},
\end{equation}
by the choice of $\kappa$ and $2  ^{ - j_\mu}\mathbf{z}_\mu\in [0,1)^{2n_\mu}$ for $\mathbf{z}_\mu\in  \Box_{
   j_\mu }^{(\mu)+}$,
  $\mu=1,2$. We see that  $I _{\mathbf{z},   \mathbf{j} }''$ in \eqref{eq:twisted-product3-I} for $ \mathbf{z}\in \Box_{ \mathbf{j }}^+$
  is a small translate of a fixed square $ 2^{2j_1+\kappa} [-1,1 )\times 2^{2j_1+\kappa}[-1,1)$ by
 \begin{equation}\label{eq:estimate-f-2} | ( 2  ^{ 2j_1}  f_o(2  ^{ - j_1}\mathbf{z}_1)|,\quad |  2  ^{ 2j_2}f_o(2  ^{ - j_2}\mathbf{z}_2) |<2^{2j_1+\kappa-10}\end{equation}
Similarly, it follows from  the  definition \eqref{eq:hat-R-o} of $ \check{ \mathcal{R}} ^{\mathbf{j} }_o$ that $
 \check{ \mathcal{R}} ^{\mathbf{j} }_o = \{(\mathbf{z},\mathbf{t}): \mathbf{z}\in \Box_{ \mathbf{j }}, \mathbf{t}\in     \check I_{\mathbf{z},   \mathbf{j} }  \}
$
 with $  \check I_{\mathbf{z},   \mathbf{j} }  := ( \check I_{\mathbf{z},   \mathbf{j} }  )_1\times ( \check I_{\mathbf{z},   \mathbf{j} }  )_2$, where
 \begin{equation}\label{eq:twisted-product4-I}\begin{split}
    ({{ \check I}}_{\mathbf{z},   \mathbf{j} } )_1:=   2  ^{ 2j_1}  f_o(2  ^{ - j_1}\mathbf{z}_1) + 2^{2j_1+\kappa}\left[-3  ,
   3  \right) ,\\
    (  \check I _{\mathbf{z},   \mathbf{j} } )_2=   2  ^{ 2j_2}f_o(2  ^{ - j_2}\mathbf{z}_2) +  2^{2j_1+\kappa}\left[-1 ,
    1  \right).
 \end{split}  \end{equation}
  We see that  for fixed $t_2\in   (\check I  _{\mathbf{z},   \mathbf{j} } )_2$, we have the inclusion
  \begin{equation*}\label{eq;segment}
   (t_2-2^{2j_1 +\kappa} ,t_2+2^{2j_1  +\kappa})\subset 2^{2j_1+\kappa}\left[-5/2  ,5/
   2  \right)   \subset  ( \check I _{\mathbf{z},   \mathbf{j} } )_1
  \end{equation*}by the estimate \eqref{eq:estimate-f-2}. So by translation,
  we get
  \begin{equation*}
   (m2^{2j_1+\kappa}+t_2-2^{2j_1 +\kappa  } ,m2^{2j_1+\kappa}+t_2+2^{2j_1 +\kappa  })\subset m2^{2j_1+\kappa}+({{ \check I}}_{\mathbf{z},   \mathbf{j} } )_1,
  \end{equation*}for $m \in  \Gamma_3$ in \eqref{eq:Gamma-3}. Thus  $(t_2-2^{2j_1+\kappa} ,t_2+2^{2j_1+\kappa})\times \{t_2\}\subset \mathcal{R}^{\mathbf{j} }_o$ for $|t_2|\leq 2^{2j_3+\kappa-3} $. The first inclusion follows.

  The second inclusion follows from that for fixed $t_2\in ({{ \check I}}_{\mathbf{z},   \mathbf{j} } )_2$,
   \begin{equation*}
    ({{ \check I}}_{\mathbf{z},   \mathbf{j} } )_1 \subset  (t_2-8\cdot 2^{2j_1 +\kappa} ,t_2+8\cdot 2^{2j_1 +\kappa })   \end{equation*}
and for $m \in  \Gamma_3$ in \eqref{eq:Gamma-3}, we have $|m2^{2j_1+\kappa}+t_2|<    2^{2j_1 +\kappa  }(2^{2(j_3 -  j_1)}+3) <2^{2j_3+\kappa+3 }$.
\end{proof}

Since the origin belongs to the boundary of the cube $\Box_{ \mathbf{j }}^+$, we consider the point
\begin{equation*}
   \zeta_{\mathbf{j}}:=(\zeta_{\mathbf{j}1}, \zeta_{\mathbf{j}2}, \zeta_{\mathbf{j}3},0,0), \qquad {\rm where}\quad \zeta_{\mathbf{j}\mu}:= \Big
   (\underbrace{2^{j_\mu-1},\ldots,2^{j_\mu-1}}_{n_\mu}\Big),
\end{equation*}$\mu=1,2,3$,  as the {\it center of the shard $\mathcal{R}^{\mathbf{j} }_o$}. We will see it  an inner point of the basic shard
$\mathcal{R}^{\mathbf{j} }_o$.
 Unlike tiles in the Heisenberg group are nested by Theorem \ref{thm:Tiles}, shards are not nested. But they still constitute  a partition $ {\mathscr N} $ for each scale by the following proposition.

\begin{prop} \label{prop:shard}
 (1)
    $\mathfrak{R }_{\mathbf{j}}$
      is a partition $ {\mathscr N} $ for each $\mathbf{j}=(j_1,j_2,j_3)\in \mathbb{Z}^3$.
\\
       (2) There exists an absolute positive integer   $\sigma  $ only depending on    $ {\mathscr N} $ such that for $\mathbf{g}\in G_{\mathbf{j}}$, we have
\begin{equation*}   T(\mathbf{g}\zeta_{\mathbf{j}}, 2^{ \mathbf{j} -\sigma }   )  \subset \mathbf{g} \mathcal{R}^{\mathbf{j} }_o\subset T(\mathbf{g},  2^
{\mathbf{j}+\sigma  } )
.
\end{equation*}
\end{prop}
\begin{proof} We only prove the most complicated case $j_2< j_1<j_3$. It is similar for   other cases.

(1)  It is easy to see that
 \begin{equation*}
    ( 6 m  2^{2j_1 +\kappa },
0 )+ {{{ \check I}}}_{\mathbf{z},   \mathbf{j} }
  \end{equation*}
for ${{{ \check I}}}_{\mathbf{z},   \mathbf{j} }
 $ given by \eqref{eq:twisted-product4-I}, are mutually disjoint for different $
  m \in \mathbb{Z} $, and so are their translates  in  along the diagonal: $\left(\mathbf{0}_{\mathbf{z}}, 2m  2^{2j_1 +\kappa }, 2 {m } 2^{2j_1+\kappa }\right)
       \check{\mathcal{R}}^{\mathbf{j} }_o$. Namely,
     \begin{equation*}
  (\mathbf{0}_{\mathbf{z}},6 m  2^{2j_1 +\kappa },
0 )+ {\mathcal{R}}^{\mathbf{j} }_o,
 \end{equation*}for $ m \in \mathbb{Z}$, are mutually disjoint and their union is $ \{(\mathbf{z},\mathbf{t}): \mathbf{z}\in \Box_{ \mathbf{j }}^+, \mathbf{t}\in
  \mathfrak{ I }_{\mathbf{z},   \mathbf{j} }    \}$ with
 \begin{equation*}
     \mathfrak{ I }_{\mathbf{z},   \mathbf{j} }   =\mathbb{R}\times\Big\{ 2  ^{ 2 j_2 } f_0(2  ^{ - j_2 }\mathbf{z}_2)+\Big[- 2^{2j_3 +\kappa} - 2^{2j_1 +
     \kappa} ,
 2^{2j_3 + \kappa}- 2^{2j_1 + \kappa}  \Big)\Big\}.
 \end{equation*}
 Translates of $\mathfrak{ I }_{\mathbf{z},   \mathbf{j} }$  under $
2^{2j_3 +\kappa }  2  \mathbb{Z}$ in $t_2$ direction   are also mutually disjoint and their union is $\mathbb{R}^2$. Thus we have shown  translates of ${
\mathcal{R}}^{\mathbf{j} }_o$  under   $0_{\mathbf{z}}\times 2^{2j_1 +\kappa } 6\mathbb{Z} \times
2^{2j_3 +\kappa }    2\mathbb{Z}$ are  mutually disjoint and their union is
$
   \Box_{ \mathbf{j }}^+\times\mathbb{R}^2.
$

Note that the group $G_{\mathbf{j}}$ in \eqref{eq:G-j-3} can be written as the product of two mutually commuting subgroups:
\begin{equation*}
   G_{\mathbf{j}} =
  \Big (2^{\mathbf{j } } \mathbb{Z}^{2\mathbf{n} } \times (0,0) \Big)  \Big(0_{\mathbf{z}}\times  2^{2j_1 +\kappa } 6\mathbb{Z} \times
2^{2j_3 +\kappa }  2  \mathbb{Z}  \Big),
\end{equation*}and $2^{\mathbf{j } } \mathbb{Z}^{2\mathbf{n} } \Box_{ \mathbf{j }}^+ $ is obviously a partition of  $\mathbb{C}^{N}$. On the other hand,  an element $(\mathbf{z},0,0)$ in $2^{\mathbf{j } } \mathbb{Z}^{2\mathbf{n} } \times (0,0) $ maps  $\Box_{ \mathbf{j }}^+\times\mathbb{R}^2$ bijectively to
$\mathbf{z}\Box_{ \mathbf{j }}^+\times\mathbb{R}^2$: the slice $\{\mathbf{z}' \}\times\mathbb{R}^2$ is mapped to $\{\mathbf{z}+\mathbf{z}' \}\times\mathbb{R}^2$ by a translate. So $\mathbf{z}\Box_{ \mathbf{j }}^+\times\mathbb{R}^2$ is disjoint for different such $\mathbf{z}$'s, and their union   over such $\mathbf{z}$ is obviously $\mathbb{C}^{N}\times\mathbb{R}^2$. In summary, we have proved that $\mathbf{g}{ \mathcal{R}}^{\mathbf{j} }_o$'s for $\mathbf{g}\in
  G_{\mathbf{j}}$ are   mutually disjoint and their union is $ {\mathscr N} $.

(2) It is sufficient to show the inclusion for $\mathbf{g}=\mathbf{0}$. For $(\mathbf{z},\mathbf{t})\in T( \zeta_{\mathbf{j}}, 2^{ \mathbf{j} -\sigma }   ) $, we
have
$ \mathbf{z}_{\mu  }\in [0,2^{  {j}_\mu   })^{2n_\mu}$ since
$| {x}_{\mu a}-2^{  {j}_\mu -1 } |<2^{  {j}_\mu -\sigma }$ for  $\sigma>1$, $a=1,\ldots,2n_\mu$. Namely, $ \mathbf{z}\in \Box_{ \mathbf{j }}^+$.
On the other hand, by $(\mathbf{z},\mathbf{t})\in T( \zeta_{\mathbf{j}}, 2^{ \mathbf{j} -\sigma }   ) = \zeta_{\mathbf{j}} T(\mathbf{0}, 2^{ \mathbf{j} -\sigma }
) $, we have
\begin{equation*}
   t_1=t_1'+\Phi_1(\zeta_{\mathbf{j}1},\mathbf{z}_1')+\Phi_3(\zeta_{\mathbf{j}3},\mathbf{z}_3'),\qquad
   t_2=t_2'+\Phi_2(\zeta_{\mathbf{j}2},\mathbf{z}_2')+\Phi_3(\zeta_{\mathbf{j}3},\mathbf{z}_3')
\end{equation*}
for some
$(t_1',t_2')\in P_{2^{  2({j}_1 -\sigma) }, 2^{ 2( {j}_3 -\sigma) }  }$ and $\mathbf{z}_\mu '\in [0, 2^{  j_\mu -\sigma }   )^{2n_\mu}$. Thus, we have
\begin{equation*}\begin{split}
  | t_1- t_2|&\leq |t_1'-t_2'|+|\Phi_1( \zeta_{\mathbf{j}1},\mathbf{z}_1')|  +|\Phi_2( \zeta_{\mathbf{j}2},\mathbf{z}_2')|\\
  &\leq 2^{  2{j}_1 -2\sigma }   + n_12^{  2{j}_1 - \sigma } + n_22^{  2{j}_2 - \sigma }   <2^{  2{j}_1 },
 \end{split} \end{equation*}
 for large $\sigma$, and similarly, $|  t_2| \leq | t_2'|+|\Phi_2( \zeta_{\mathbf{j}2},  \mathbf{z}_2')|+|\Phi_1( \zeta_{\mathbf{j}3}, \mathbf{z}_3')|    <2^{  2{j}_3 }$. Consequently, $(\mathbf{z},\mathbf{t})\in \Box_{ \mathbf{j }}^+\times  P_{2^{2j_1  } ,2^{2j_3  } }$ and so the first inclusion follows by
 using Lemma \ref{prop:shard-P}.

The second inclusion follows from
     $ \mathcal{R}^{\mathbf{j} }_o\subset \Box_{ \mathbf{j }}^+\times  P_{ 2^{2j_1+\kappa+3}, 2^{2j_3+\kappa+3 }}$ in Lemma \ref{prop:shard-P} and
     $\Box_{ \mathbf{j }}^+\subset \Box_{ \mathbf{j }}$.
 \end{proof}

  \section{Atomic decomposition of the Hardy space $H_{area,\boldsymbol \varphi}^
1(\mathscr N)$}
\subsection{The tent of a shard}

For each shard $R $ in $\mathfrak{R }_{\mathbf{j}}$,      the {\it tent $T(R)$ over  $R $}   is defined as
\begin{equation}\label{eq:tent-2} T(R):=  R\times  \left[  2^{ j_1 },
   2^{ j_1+1 }  \right)\times \left[  2^{ j_2 },
   2^{ j_2+1 }  \right)\times \left[  2^{ j_3  }  ,
   2^{ j_3+1 }  \right) \subset \mathscr N \times\mathbb{R}_+^3,
\end{equation}as in the flag case.

\begin{prop}\label{prop:tent-decomposition}
 We have the disjoint decomposition
 \begin{equation}\label{eq:tent-decomposition}
   \mathscr N \times\mathbb{R}_+^3= \bigcup_{R\in {\mathfrak{R}}} T(R).
 \end{equation}
\end{prop}
 \begin{proof} Since for given $\mathbf{j}\in \mathbb{Z}^3$, $\mathfrak{R }_{\mathbf{j}}$
      is a partition of $ {\mathscr N} $ by Proposition \ref{prop:shard}, and so
      \begin{equation}\label{eq:tent-3}
         \bigcup_{R\in \mathfrak{R }_{\mathbf{j}} }   T(R)= {\mathscr N} \times \left[  2^{ j_1 },
   2^{ j_1+1}  \right)\times \left[  2^{ j_2 },
   2^{ j_2 +1}  \right)\times \left[  2^{ j_3  }  ,
   2^{ j_3+1 }  \right) ,
      \end{equation}
is a partition.
      On the other hand, if $R\in \mathfrak{R }_{\mathbf{j}}$ and $R'\in \mathfrak{R }_{\mathbf{j}'}$ with $\mathbf{j}\neq\mathbf{j}'$, there exists some $\mu$
      such that $j_\mu\neq j_\mu'$.  Hence, $\left[ 2^{ j_\mu },
   2^{ j_\mu +1}  \right) \cap   [  2^{ j_\mu'  }  ,
   2^{ j_\mu '+1}   )=\emptyset$, and so $T(R)\cap T(R')=\emptyset$.
  Therefore,  $T(R)$ are disjoint for any two different $R$'s in ${ \mathcal{R}} $. By taking union of \eqref{eq:tent-3} over $\mathbf{j}$, we get the decomposition
  \eqref{eq:tent-decomposition}.  \end{proof}

Define the ($\sigma$-){\it enlargement} $ R^{* }$  of a shard $R=\mathbf{g}\mathcal{R}^{\mathbf{j} }_o$ with $\mathbf{g}\in
G_{\mathbf{j}}$ to be the tube $T(\mathbf{g},  2^ {\mathbf{j} + 2\sigma} )$. Then $R \subset R^{* }  $ by Proposition \ref{prop:shard}.

\begin{prop}For any shard $R $, we have $M(\mathbbm 1_{R})(\mathbf{g}')\geq 2^{ - 4\sigma Q }$ for $\mathbf{g}'\in  R^{* }$.
\end{prop}
\begin{proof} Let $R=\mathbf{g}\mathcal{R}^{\mathbf{j} }_o$ for some $\mathbf{g}\in
G_{\mathbf{j}}$. Then for $\mathbf{g}' \in R^{* }=\mathbf{g}T(\mathbf{0},  2^ {\mathbf{j}  +2\sigma  } )$, we have $\mathbf{g} \in  \mathbf{g}' T(\mathbf{0},
2^ {\mathbf{j}  +2\sigma } )$, since the tube is invariant under the transformation $\mathbf{g}\rightarrow \mathbf{g}^{-1}$ by definition.  On the other hand,
$R\subset  \mathbf{g}T\left(\mathbf{0},    2^ {\mathbf{j} +\sigma  }\right )$ by Proposition \ref{prop:shard}. Thus,
\begin{equation}\label{eq:gT-g'T}\begin{split}& R\subset    \mathbf{g}'T\left(\mathbf{0},  2^ {\mathbf{j} +2\sigma   }\right )T\left(\mathbf{0}, 2^
{\mathbf{j}+ \sigma } \right) \subset \mathbf{g}'T\left(\mathbf{0},  2^ {\mathbf{j} +3\sigma  } \right).
 \end{split}\end{equation}This is because  $\tilde{{B}}_\mu (\mathbf{0} ,  {r}_\mu)  \tilde{{B}}_\mu (\mathbf{0} ,  {r}_\mu') \subset \tilde{{B}}_\mu (\mathbf{0} ,
 \sqrt{ n_\mu} ({r}_\mu +{r}_\mu'))$ by the multiplication law \eqref{eq:multiplication3}, and so
 \begin{equation}\label{eq:T-T}\begin{split}
   T(\mathbf{0}, \mathbf{r}) T(\mathbf{0},\mathbf{r}')&\subset \pi \left(\tilde{{B}} (\mathbf{0} , 2\mathbf{r})\right)\pi \left(\tilde{{B}} (\mathbf{0}
   ,2\mathbf{r}') \right)=\pi \left(\tilde{{B}} (\mathbf{0} , 2\mathbf{r} )  \tilde{{B}} (\mathbf{0} , 2\mathbf{r}')\right)\\
   &\subset \pi\left (   \tilde{{B}} (\mathbf{0} ,2^{\sigma-2} (\mathbf{r}+\mathbf{r}'))\right)\subset  T(\mathbf{0},2^{\sigma}(\mathbf{r}+\mathbf{r}')/2)
 \end{split}\end{equation}for   $2^{\sigma }>8 \sqrt{n_\mu}$, by using Proposition \ref{prop:tube-pi}  and $\pi$  being a homomorphism. Thus,
 \begin{equation}\label{eq:g'-Tg}\begin{split} M(\mathbbm 1_{R})(\mathbf{g}')&\geq \frac 1{\left|T\left(\mathbf{g}', 2^ {\mathbf{j} +3\sigma  }
 \right)\right|}
 \int_{T\left(\mathbf{g}', 2^ {\mathbf{j} +3\sigma }  \right)}
 \mathbbm 1_{R} (\mathbf{g}'')d\mathbf{g}''\\
 &=\frac {|R|}{\left|T\left(\mathbf{0},2^ {\mathbf{j} +3\sigma  } \right)\right|} \geq \frac {\left|T\left(\mathbf{0},  2^ {\mathbf{j}-\sigma }    \right)\right|
 }{\left|T\left(\mathbf{0}, 2^ {\mathbf{j} +3\sigma}  \right)\right|} = \frac 1 {
 2^{  4\sigma Q}}, \end{split}\end{equation} by Proposition \ref{prop:shard} and \eqref{eq:tube-volume}, where $Q$ is the homogeneous dimension of $ \mathscr N$.
 \end{proof}

\subsection{Atomic decomposition}
Suppose that $E$ is an open set of $\mathscr N $
of finite measure. We write $\mathfrak{R}(E)$ for the
set of all shards $R$ in $\mathfrak{R}$ whose interior is a subset of $E$, and $\mathfrak{M}(E)$ for the collection of all maximal
such $R$ in $\mathfrak{R}(E)$.

Fix a positive integer  $M $. An {\it atom} $a$ is  a  function in
 $L^2
(\mathscr N)$ such that there exists an open subset $E$ of  $
 \mathscr N $
of finite measure   and functions $a_R$ in $L^2
(\mathscr N)$,  and $b_R$ in ${\rm Dom}( {\triangle}_1^M {\triangle}_2^M {\triangle}_3^M)$
   for all $R \in  \mathfrak{ M}(E)$ such that
\\
(i) $a_R=  {\triangle}_1^M {\triangle}_2^M {\triangle}_3^Mb_R$
  and supp $b_R\subset R^*$,
  where $R^*$
is the  enlargement of $R$;
\\
(ii)
  for all sign sequences $\tau:  \mathfrak{ M}(E)\rightarrow \{\pm 1\}$, the sum $\sum_{R\in \mathcal{M}(E)} \tau_Ra_R$ converges in $L^2
(\mathscr N)$ to $ a_\tau$ say,
and $\|a_\tau\|_{L^2
(\mathscr N)}\leq |E|^{-\frac 12}$;
\\
(iii)
$ a =\sum_{R\in \mathcal{M}(E)}  a_R$.

Consider $f \in L^2
(\mathscr N) \cap  H^1_{area,\boldsymbol\varphi}(\mathscr N)$. For each $\ell \in \mathbb{Z}$, set
\begin{equation}\label{eq:E-ell}\begin{split}
   E_\ell:&=\left\{\mathbf{g}\in \mathscr N; S_{area,\boldsymbol \varphi}(f)(\mathbf{g})>2^{\ell}\right\},\\
  { \mathfrak{R }}_\ell:&=\left\{R\in  {\mathfrak{R }}; |R^* \cap E_{\ell+1}|\leq\frac 1 {  2^{  2\sigma  Q+1}}  |R^*  |<|R^* \cap E_{\ell  }|\right\},\\
 \widetilde{  E }_\ell:&=\left\{\mathbf{g}\in \mathscr N;M(\mathbbm 1_{ E_\ell})(\mathbf{g})> \frac 1 {
2^{  3\sigma  Q+1}} \right\}.
 \end{split}\end{equation}
Note that for any shard $R $,
there exists a unique $\ell$ such that $ R\in { \mathfrak{R }}_\ell$ by definition \eqref{eq:E-ell}, since $|R^* \cap E_{\ell  }|$   decreases from $|R^*   |$ to
$0$ when $l$ increases from $-\infty$ to $+\infty$. We claim   $R^*\subset \tilde{ {E}}_\ell$. This is
because for any $\mathbf{g}'\in R^*= T\left(\mathbf{g},2^ {\mathbf{j} + 2\sigma } \right )$, if  argue as  \eqref{eq:gT-g'T},  we get
\begin{equation*}
   T\left(\mathbf{g}',2^ {\mathbf{j} +3\sigma  }  \right )\supset  \mathbf{g}' T\left(\mathbf{0} ,2^ {\mathbf{j} +2 \sigma } \right ) T\left(\mathbf{0} ,2^
   {\mathbf{j} +2 \sigma } \right )\supset  \mathbf{g} T\left(\mathbf{0} ,2^ {\mathbf{j} +2 \sigma } \right )= R^*,
\end{equation*}
by using \eqref{eq:T-T}, and so
\begin{equation*}\begin{split}
    M(\mathbbm 1_{ E_\ell})(\mathbf{g}')&\geq \frac 1{\left|T\left(\mathbf{g}',2^ {\mathbf{j} +3\sigma }\right )\right|}\int_{T\left(\mathbf{g}',2^ {\mathbf{j}
    +3\sigma }\right )}\mathbbm 1_{ E_\ell} (\mathbf{g}  '')d\mathbf{g}''
   \geq \frac {|R^* \cap E_{\ell  }|}{ 2^{ \sigma  Q}\left|R^* \right|}>\frac 1 {
2^{   3\sigma   Q+1}}  .
 \end{split}\end{equation*}Thus, $R^*\subset \tilde{ {E}}_\ell$ by definition \eqref{eq:E-ell}. The claim is proved.

If we  write ${}_{\mathbf{g}}\psi_ {\mathbf{r}}(\mathbf{h})$  for the function $\mathbf{h}\mapsto \psi_ {\mathbf{r}}(\mathbf{g}^{-1}\mathbf{h})$, which is a   left translate of $\psi_ {\mathbf{r}}$,
and apply the Calder\'on
reproducing formula  \eqref{eq:reproducing} to $f \in L^2
(\mathscr N) \cap  H^1_{area}(\mathscr N)$, we  get
\begin{equation}\begin{split}
 f(\mathbf{h})&=\int_{\mathscr N\times \mathbb{R}_+^3} f*\varphi_ {\mathbf{r}}( \mathbf{g} )\cdot {}_{\mathbf{g}}\psi_ {\mathbf{r}}(\mathbf{h})\frac {
 d\mathbf{r}}{\mathbf{r}}d\mathbf{g}\\
&=\sum_{\ell\in \mathbb{Z}}\sum_{R\in  {\mathfrak{R }}_\ell}\int_{  T(R)  }
  f*\varphi_ {\mathbf{r}}( \mathbf{g} )\cdot {}_{\mathbf{g}}\psi_ {\mathbf{r}}(\mathbf{h})\frac { d\mathbf{r}}{\mathbf{r}}d\mathbf{g} \\
  &=\sum_{\ell\in \mathbb{Z}}\lambda_\ell \sum_{R\in  {\mathfrak{R }}_\ell} a_{\ell, R}(\mathbf{h})
 \end{split} \end{equation} by the partition of $ \mathscr N \times\mathbb{R}_+^3$ into tents in Proposition \ref{prop:tent-decomposition}, where
 \begin{equation}\label{eq:decomposition-R} \begin{split}
 a_{\ell, R}(\mathbf{h})&= \frac 1{\lambda_\ell}\int_{  T(R)  }
  f*\varphi_ {\mathbf{r}}( \mathbf{g} )\cdot {}_{\mathbf{g}}\psi_ {\mathbf{r}}(\mathbf{h})\frac { d\mathbf{r}}{\mathbf{r}}d\mathbf{g} ,\\
  \lambda_\ell &=\left\|\left(\sum_{R\in  {\mathfrak{R }}_\ell}\int_{\mathbb{R}^3_+}
 | f*\varphi_ {\mathbf{r}}(\cdot ) |^2\chi_{ T(R)}(\cdot,\mathbf{r})\frac { d\mathbf{r}}{\mathbf{r}} \right)^{\frac 12} \right\|_{L^2
(\mathscr N)} \left| {E}_\ell\right|^{\frac 12} .
 \end{split} \end{equation}

 For fixed $\ell\in \mathbb{Z}$,  we claim that $\sum_{R\in  {\mathfrak{R }}_\ell}   a_{\ell, R}$  is a   atom associated to $   \tilde{ {E}}_\ell$. Since the
 w-inverses $\psi^{
(\mu)}=\triangle_\mu^M \dot{ \psi}^{
(\mu)}$
on $ {\mathscr
  {H}}_\mu$ by the assumption of the theorem, we have
$
     a_{\ell, R}=  {\triangle}_1^M {\triangle}_2^M {\triangle}_3^Mb_{\ell, R}
$
  with
  \begin{equation}\begin{split}
 b_{\ell, R}(\mathbf{h})&= \frac 1{\lambda_\ell}\int_{  T(R)  }(r_1r_2 r_3)^{-2M}
  f*\varphi_ {\mathbf{r}}( \mathbf{g}' )\cdot {}_{\mathbf{g}'}\dot{\psi}_ {\mathbf{r}}(\mathbf{h})\frac { d\mathbf{r}}{\mathbf{r}}d\mathbf{g}',
 \end{split} \end{equation}
 since $\triangle_1$, $\triangle_2$ and $\triangle_3$   consist of left invariant vector fields on $\mathscr N$.
 For $R\in  {\mathfrak{R }}_\ell$ and $(\mathbf{g}',\mathbf{r})\in T(R)$,
   we have
$\mathbf{g}'\in  R \subset T\left(\mathbf{g},2^ {\mathbf{j} +  \sigma } \right )$  by Proposition \ref{prop:shard} (2) and $   2^{ j_1 }\leq r_\mu<
   2^{ j_1+1 } $, and so
\begin{equation}\label{eq:T-g'-r-R*}
   T(\mathbf{g}',\mathbf{r})\subset \mathbf{g}'T(\mathbf{0} ,\mathbf{r})\subset \mathbf{g} T\left(\mathbf{0},2^ {\mathbf{j} +  \sigma } \right )T(\mathbf{0}
   ,\mathbf{r})\subset \mathbf{g}T\left(\mathbf{0},2^ {\mathbf{j} + 2\sigma } \right )=R^*
\end{equation}
   by using \eqref{eq:T-T}. Therefore,
 $
{\rm supp} \, b_{\ell, R} \subset R^*.
 $

The decomposition \eqref{eq:decomposition-R} is taken over  not necessarily maximal  shards, but we may group together  to a sum  over maximal ones by the
following lemma.
\begin{lem}
  Fix integer  $M \geq 1 $. Let $E$ be an open subset of
$\mathscr N$
of finite measure. Suppose that there exist functions $a_R$ in
$L^2(\mathscr N)$ and $b_R$ in   ${\rm Dom}( {\triangle}_1^M {\triangle}_2^M {\triangle}_3^M)$
   for all $R \in  \mathfrak{ R}(E)$ such that
 \\
(A1) $a_R=  {\triangle}_1^M {\triangle}_2^M {\triangle}_3^Mb_R$
  and supp $b_R\subset R^*$,
  where $R^*$
is  the $\sigma$-enlargement of $R$;
\\
(A2)
  for all sign sequences $\tau:  \mathfrak{ R}(E)\rightarrow \{\pm 1\}$, the sum $\sum_{R\in \mathcal{R}(E)} \tau_Ra_R$ converges in $L^2
(\mathscr N)$ to $ a_\tau$ say,
and $\|a_\tau\|_{L^2
(\mathscr N)}\leq |E|^{-\frac 12}$;
\\
(A3)
$ a =\sum_{R\in \mathcal{R}(E)}  a_R$.

Suppose that $\mathfrak{S}$ is a subcollection of $\mathfrak{R}(E)$ and  $\dag :\mathfrak{R}(E)\rightarrow \mathfrak{S}$   is a mapping such that $R\subset
R^\dag$. Then for
each $ S\in \mathfrak{S}$, the sum $ \sum_{R\in  \mathfrak{R}(E),R^\dag=S} a_R $ converges in $L^2
(\mathscr N)$ to $\tilde{a}_S$, say, and  $ \sum_{R\in  \mathfrak{R}(E),R^\dag=S} b_R $ converges in $L^2
(\mathscr N)$ to $\tilde{b}_S$, say. Further,
\\
(B1) $\tilde{a}_S=  {\triangle}_1^M {\triangle}_2^M {\triangle}_3^M\tilde{b}_S$
  and supp $\tilde{b}_S\subset S^*$;
\\
(B2)
  for all sign sequences $\tau :  \mathfrak{S}\rightarrow \{\pm 1\}$, the sum $\sum_{S\in \mathfrak{S}} \tau_S\tilde{a}_S$ converges in $L^2
(\mathscr N)$ to $ \tilde{a}_\tau$ say,
and $\|\tilde{a}_\tau\|_{L^2
(\mathscr N)}\leq |E|^{-\frac 12}$;
\\
(B3)
$ a =\sum_{R\in \mathfrak{S}}  a_S$.

In particular, any function a for which (A1) to (A3) hold is an atom.
 \end{lem}
This is \cite[Lemma 3.4]{CCLLO} for flag atoms on the Heisenberg group. The proof is exactly the same for our case. We omit the details.

It remains
  to show for  any sign sequence $\tau :  {\mathfrak{R }}_\ell\rightarrow\{\pm 1\}$,   and $\sum_{R\in  {\mathfrak{R }}_\ell}
\tau_{\ell, R} a_{\ell, R}$ converges in $L^2(
\mathscr N)$,  say $a_\tau$, and satisfies the estimate  $\|a_\tau \|_{L^2
(\mathscr N)}\lesssim \left| {E}_\ell\right|^{-\frac 12}$.

For
all smooth compactly supported $ h$ on $\mathscr N$
such that  $\|h\|_{L^2
(\mathscr N)}\leq 1$,
  \begin{equation*}\begin{split}
\left|\int_{\mathscr N} \right.&\sum_{R\in  {\mathfrak{R }}_\ell} \tau_{\ell, R} a_{\ell, R}(\mathbf{g}') h(\mathbf{g}')d\mathbf{g}'\left|=\left| \frac
1{\lambda_\ell} \sum_{R\in  {\mathfrak{R }}_\ell}\tau_{\ell, R} \int_{  T(R)  }\int_{\mathscr N}
 \right. f*\varphi_ {\mathbf{r}}( \mathbf{g} ) \cdot {}_{\mathbf{g}}\psi_ {\mathbf{r}}(\mathbf{g}') h(\mathbf{g}') d\mathbf{g}'\frac {
  d\mathbf{r}}{\mathbf{r}}d\mathbf{g}\right|\\&=\left| \frac 1{\lambda_\ell} \sum_{R\in  {\mathfrak{R }}_\ell}\tau_{\ell, R} \int_{  T(R)  }
h* \breve{\psi}_ {\mathbf{r}}(\mathbf{g} ) \cdot  f*\varphi_ {\mathbf{r}}( \mathbf{g} )  \frac { d\mathbf{r}}{\mathbf{r}}d\mathbf{g}\right|\\&=\left| \frac
1{\lambda_\ell} \sum_{R\in  {\mathfrak{R }}_\ell}\tau_{\ell, R} \int_{ \mathscr N\times \mathbb{R}^3 _+ }
h* \breve{\psi}_ {\mathbf{r}}(\mathbf{g} ) \cdot  f*\varphi_ {\mathbf{r}}( \mathbf{g} ) \chi_{ T(R)}(\mathbf{g} ,\mathbf{r})\frac {
d\mathbf{r}}{\mathbf{r}}d\mathbf{g}\right| \\&
\leq \frac 1{\lambda_\ell} \left(
\sum_{R\in  {\mathfrak{R }}_\ell}  \int_{ \mathscr N\times \mathbb{R}^3_+  }
|h* \breve{\psi}_ {\mathbf{r}}(\mathbf{g} ) |^2 \chi_{ T(R)}(\mathbf{g} ,\mathbf{r})\frac { d\mathbf{r}}{\mathbf{r}}d\mathbf{g}\right)^{\frac 12}
 \left(
\sum_{R\in  {\mathfrak{R }}_\ell}  \int_{ \mathscr N\times \mathbb{R}^3_+  }
|f*\varphi_ {\mathbf{r}}( \mathbf{g} ) |^2 \chi_{ T(R)}(\mathbf{g} ,\mathbf{r})\frac { d\mathbf{r}}{\mathbf{r}}d\mathbf{g}\right)^{\frac 12},\end{split} \end{equation*}by the Cauchy-Schwarz inequality. Then it
 \begin{equation*}\begin{split}&
\leq \frac 1{\lambda_\ell} \left(
  \int_{ \mathscr N\times \mathbb{R}^3_+  }
|h* \breve{\psi}_ {\mathbf{r}}(\mathbf{g} ) |^2  \frac { d\mathbf{r}}{\mathbf{r}}d\mathbf{g}\right)^{\frac 12}  \left(
\int_{ \mathscr N   } \left( \sum_{R\in  {\mathfrak{R }}_\ell}  \int_{  \mathbb{R}^3_+  }
|f*\varphi_ {\mathbf{r}}( \mathbf{g} ) |^2 \chi_{ T(R)}(\mathbf{g} ,\mathbf{r})\frac { d\mathbf{r}}{\mathbf{r}}\right)d\mathbf{g}\right)^{\frac 12}\\
&
\lesssim \frac 1{\lambda_\ell}  \|h\|_{L^2
(\mathscr N)}  \left\|  \left(
\sum_{R\in  {\mathfrak{R }}_\ell}  \int_{  \mathbb{R}^3_+  }
|f*\varphi_ {\mathbf{r}}(\cdot ) |^2 \chi_{ T(R)}(\cdot ,\mathbf{r})\frac { d\mathbf{r}}{\mathbf{r}} \right)^{\frac 12}\right\|_{L^2
(\mathscr N)}  \leq \left| {E}_\ell\right|^{-\frac 12}
 \end{split} \end{equation*}by the square function estimate in Theorem \ref{thm:g-function} and the definition
of $\lambda_\ell$.
Here, we choose $\psi^{(\mu)}$  to be Poisson bounded and   have mean value zero
 by Remark \ref{rem:compact-support}.
     It follows that
\begin{equation*}
   \left\|  \sum_{R\in  {\mathfrak{R }}_\ell} \sigma_{\ell, R} a_{\ell, R}\right\|_{L^2
(\mathscr N)}  \lesssim \left| {E}_\ell\right|^{-\frac 12}.
\end{equation*}

To see  the convergence of the series $\sum_\ell|\lambda_\ell|$, recall that for $R\in  {\mathfrak{R }}_\ell$ and $(\mathbf{g}',\mathbf{r})\in T(R)$, we have $T(\mathbf{g}',\mathbf{r})\subset  R^*$
by \eqref{eq:T-g'-r-R*}.
It follows from our previous claim $R^*\subset \tilde{ {E}}_\ell$ that  $  T(\mathbf{g}',\mathbf{r})\subset \widetilde{ {E}}_\ell$.
On the other hand,
\begin{equation*}
   \left|   {E}_{\ell +1}  \cap T(\mathbf{g}',\mathbf{r})\right|\leq  \left|   {E}_{\ell +1}  \cap  R^*\right|\leq\frac 1 {  2^{  2\sigma  Q+1}} |R^*  |\leq\frac 1
 {  2 }
   \left|
   T(\mathbf{0} ,2^{\mathbf{j} })\right|\leq\frac 1    2
   \left|
   T(\mathbf{0} ,\mathbf{r})\right|.
\end{equation*}
Therefore
\begin{equation*}
   \frac {\left|( \widetilde{{E}}_\ell\setminus  {E}_{\ell +1}) \cap T(\mathbf{g}',\mathbf{r})\right|}{|T(\mathbf{g}',\mathbf{r})|}\geq \frac 1 {  2 } .
\end{equation*}
Thus, we have
 \begin{equation*}\begin{split}
  \sum_{R\in  {\mathfrak{R }}_\ell}\int_{ T(R)}
 | f*\varphi_ {\mathbf{r}}(\mathbf{g}')|^2 \frac { d\mathbf{r}}{\mathbf{r}}d\mathbf{g}' &\leq  2 \sum_{R\in  {\mathfrak{R }}_\ell}\int_{ T(R)}
 | f*\varphi_ {\mathbf{r}}(\mathbf{g}')|^2\frac {\left|(\widetilde{ {E}}_\ell\setminus  {E}_{\ell +1}) \cap
 T(\mathbf{g}',\mathbf{r})\right|}{|T(\mathbf{g}',\mathbf{r})|} \frac { d\mathbf{r}}{\mathbf{r}}d\mathbf{g}'\\
 &\leq  2 \int_{  \mathscr N\times \mathbb{R}^3_+ }
 | f*\varphi_ {\mathbf{r}}(\mathbf{g}')|^2\frac {\left|( \widetilde{{E}}_\ell\setminus  {E}_{\ell +1}) \cap
 T(\mathbf{g}',\mathbf{r})\right|}{|T(\mathbf{g}',\mathbf{r})|} \frac { d\mathbf{r}}{\mathbf{r}}d\mathbf{g}'\\
 &=  2 \int_{\widetilde{ {E}}_\ell\setminus  {E}_{\ell +1}}\int_{  \mathscr N\times \mathbb{R}^3_+ }
 | f*\varphi_ {\mathbf{r}}(\mathbf{g}')|^2\frac {\mathbbm 1_{ T(\mathbf{g}',\mathbf{r})}( \mathbf{g} )}{|T(\mathbf{g}',\mathbf{r})|} \frac {
 d\mathbf{r}}{\mathbf{r}}d\mathbf{g}'d\mathbf{g}\\
 &=  2 \int_{ \widetilde{{E}}_\ell\setminus  {E}_{\ell +1}}\int_{  \Gamma(\mathbf{g}) }
 | f*\varphi_ {\mathbf{r}}(\mathbf{g}')|^2\frac { 1}{|T(\mathbf{g}',\mathbf{r})|} \frac { d\mathbf{r}}{\mathbf{r}}d\mathbf{g}'d\mathbf{g}
 \\
 &=  2 \int_{\widetilde{ {E}}_\ell\setminus  {E}_{\ell +1}}|S_{area,\boldsymbol\varphi}(f)(\mathbf{g}) |^2d\mathbf{g}\\
 &\leq  2^{2\ell+3 }\left| \widetilde{{E}}_\ell\right|\lesssim 2^{2\ell  }\left|  {{E}}_\ell\right|,
 \end{split} \end{equation*}where we use the $L^2$-boundedness of    the  tube  maximal function  $ M  $ in Theorem \ref{thm:maximal} in the last inequality.
 Consequently, we get
  \begin{equation*}\begin{split}\sum_\ell|\lambda_\ell|&\leq\sum_\ell\left(  \sum_{R\in  {\mathfrak{R }}_\ell}\int_{ T(R)}
 | f*\varphi_ {\mathbf{r}}(\mathbf{g}')|^2 \frac { d\mathbf{r}}{\mathbf{r}}d\mathbf{g}'   \right)^{\frac 12}\left|  {{E}}_\ell\right| ^{\frac 12}\\
 &\leq \sum_\ell2^{\ell }\left|  {{E}}_\ell\right|\lesssim \left\|S_{area,\boldsymbol\varphi}(f)\right\|_{L^1
(\mathscr N)}.
 \end{split} \end{equation*}
The Theorem is proved.

\appendix

\section{Calculation of the  Cauchy-Szeg\H o kernels on some  Siegel domains }

As in   \cite[Section 3.1]{WW}, we    calculate the  Cauchy-Szeg\H o kernel by using general Gindikin's formula.

For a regular cone $\Omega\subset    \mathbb{R}^{m}$ (i.e. it is a nonempty open convex  with vertex at $0$ and
containing no entire straight line),   an
 Hermitian form
 $\Psi: \mathbb{C}^{N }\times \mathbb{C}^{N }\rightarrow \mathbb{C}^{m}$ is said to be  {\it $\Omega$-positive} if $\Psi(z ,z )\in
 \overline{\Omega}$ for any $z
 \in\mathbb{C}^{N }$ and $\Psi(z ,z )=0$ only if $z =0$. For the Siegel domain $\mathcal D $     defined by
 \eqref{eq:Siegel-domain}, its {\it Shilov boundary} ${\mathcal S }$ is the CR submanifold defined by the equation
$
  \operatorname{Im} \zeta'' - \Psi(\zeta' ,\zeta'  )=0 ,
$ which has the structure of   a nilpotent Lie group  of step two.

 For $f\in H^2(\mathcal D)$, it is well known \cite{KS} that its boundary value 
 \begin{equation*}
     \lim_{\substack {y\in \Omega\\y\rightarrow 0}}f(\zeta',  x +\mathbf{i}y
   +\mathbf{i}\Psi(\zeta',\zeta' )  )
 \end{equation*}exists in $L^2({\mathcal S})$
   and its $L^2({\mathcal S})$ norm equals to $\|f\|_{H^2(\mathcal D)}$. Thus,
   $H^2(\mathcal D )$ can be identified with a closed subspace of $L^2({\mathcal S})$.
The {\it  Cauchy-Szeg\H o projection} $\mathcal P  :L^2(\mathcal{S})\rightarrow H^2(\mathcal D )$ is the orthogonal projection to this closed subspace. It  has a reproducing kernel $S(\zeta, \eta)$, the {\it Cauchy-Szeg\H o kernel} \cite{Gin}, which  is holomorphic
in $\zeta\in \mathcal D $.
 Namely, for $f\in H^2(\mathcal D )$, we have the reproducing formula:
\begin{equation}\label{eq:Szego-reproducing}
   f(\zeta)=\int_{{\mathcal S}} S(\zeta,\eta)f(\eta)d\beta(\eta),
\end{equation}
where $d\beta$ is the measure corresponding to $dx
   d\zeta''$ in \eqref{eq:H2}. For $f\in L^2({\mathcal S} )$,
   \begin{equation}\label{eq:Szego-projection}
      \mathcal Pf(\zeta',  x   +\mathbf{i}\Psi(\zeta',\zeta' ))= \lim_{\substack {y\in \Omega\\y\rightarrow 0}}\int_{{\mathcal S}} S(\zeta',  x +\mathbf{i}y
   +\mathbf{i}\Psi(\zeta',\zeta' ) ,\eta)f(\eta)d\beta(\eta).
   \end{equation}

The {\it dual cone}  $\Omega^*$ is the set of all $\lambda\in(\mathbb{R}^{m})^*$  such that $\langle\lambda, x\rangle > 0$ for all
 $x\in \overline{\Omega} \setminus\{0\}$.
For $\lambda\in \Omega^*$, denote $B_\lambda(\zeta' ,\eta' ) :=4\langle\lambda,\Psi(\zeta' ,\eta' )\rangle$,  an
Hermitian form on $\mathbb{C}^{N }$, whose
associated Hermitian matrix is also denoted by  $B_\lambda$. The explicit formula for $S(\zeta, \eta)$ \cite[Theorem
5.1]{KS} is known
 as
\begin{equation}\label{eq:Szego}
   S(\zeta, \eta)=\int_{\Omega^*}e^{-2\pi\langle\lambda,\rho(\zeta,\eta)\rangle}\det B_\lambda\, d\lambda,
\end{equation}
for $\zeta=(\zeta',\zeta'')\in \mathcal{D}, \eta=(\eta',\eta'')\in  \mathcal{S}$, where
$
   \rho(\zeta,\eta)=\frac {\zeta''-\overline{\eta''} }{\mathbf{i}}-2\Psi(\zeta'  ,\eta' )
$ is polarized form of $
  2 \operatorname{Im} \zeta'' - 2\Psi(\zeta' ,\zeta'  )
$.

The domain  \eqref{eq:U} is a  Siegel domain defined by  the regular  cone $\Omega=
\mathbb{R}^{2}_+\subset    \mathbb{R}^{2}$  with $\Omega^*=
\mathbb{R}^{2}_+ $, $m=2$,    $N:=n_1+ n_2+ n_3$,   and the Hermitian form given by
 \begin{equation*}
    \Psi(   {\mathbf z} ,   {\mathbf z}')=\Big( \left\langle   {\mathbf z
    }_1 ,   {\mathbf z
    }_1'\right\rangle+ \langle  {\mathbf z
    }_3 ,   {\mathbf z
    }_3'\rangle, \langle   {\mathbf z
    }_2 ,   {\mathbf z
    }_2'\rangle+ \langle  {\mathbf z
    }_3 ,   {\mathbf z
    }_3'\rangle  \Big),
 \end{equation*}
  where $\langle\cdot ,\cdot\rangle $ is the standard Hermitian inner product in $\mathbb{C}^{n_\mu}$. $\Psi$ is  $
\mathbb{R}^{2}_+$-positive,
   and $\rho =(\rho_1 ,\rho_2 )$  with
   \begin{equation}\label{eq:rho}
   \rho_\alpha(\zeta,\eta)=\frac { w_\alpha-\overline{   w'_\alpha }
}{\mathbf{i}} -2\langle   {\mathbf z}_\alpha ,
   {\mathbf
   z}_\alpha'\rangle-2\langle   {\mathbf z}_3,
   {\mathbf
   z}_3'\rangle,
\end{equation}for   $\zeta=(\zeta',\zeta'')=( {\mathbf z}, {\mathbf w} )\in\mathcal{D}$,
$\eta=(\eta',\eta'')=( {\mathbf
z}',{\mathbf w}' )\in   {\mathcal S}$.
  Then,
  $
     B_\lambda(   {\mathbf z }  ,  {\mathbf z } ')=4\lambda_1\langle   {\mathbf z }_1
  ,   {\mathbf z
    }_1'\rangle+ 4\lambda_2\langle  {\mathbf z }_2 ,   {\mathbf z
    }_2'\rangle+ 4(\lambda_1+\lambda_2)\langle  {\mathbf z }_3 ,   {\mathbf z
    }_3'\rangle
 $,
    i.e.
\begin{equation*}
   B_\lambda=4
   \left(\begin{matrix} \lambda_1I_{n_1}&0&0\\
   0&\lambda_2I_{n_2}&0\\
   0&0& (\lambda_1+\lambda_2)I_{n_3}
  \end{matrix}\right).
\end{equation*}Thus,
 $\det B_\lambda=4^{N} \lambda_1^{n_1} \lambda_2^{n_2} (\lambda_1+\lambda_2)^{n_3}  $. Recall that for $\theta\in \mathbb{C}$ with ${\rm Re}\,\theta>0$,
\begin{equation*}
   \int_0^{+\infty} e^{-2\pi  s\theta }s^{m}  ds= \frac {m!}{(2\pi \theta)^{m+1}}.
\end{equation*}
 It follows   that
\begin{equation} \label{eq:Szego2}\begin{split}
   S(\zeta, \eta)&= \int_{\mathbb{R}^{2}_+ }e^{-2\pi\sum_{\alpha=1}^{2} \lambda_\alpha\rho_\alpha(\zeta,\eta)
   }4^{N} \lambda_1^{n_1} \lambda_2^{n_2} (\lambda_1+\lambda_2)^{n_3}  d\lambda\\&=\sum_{k=0}^{n_3}\binom {n_3}{k}\int_{\mathbb{R}^{2}_+
   }e^{-2\pi\sum_{\alpha=1}^{2} \lambda_\alpha\rho_\alpha(\zeta,\eta)
   }4^{N} \lambda_1^{n_1+k} \lambda_2^{n_2+n_3-k}  d\lambda
   \\&
   =\frac { 1 }{ 4^2(\frac \pi
 2)^{N+2}}\sum_{k=0}^{n_3}\binom {n_3}{k}
 \frac { (n_1+k )!}{  \rho_1(\zeta,\eta) ^{n_1+k+1}  } \frac { (n_2+n_3-k)!}{  \rho_2(\zeta,\eta) ^{n_2+n_3-k+1}  }.
\end{split}\end{equation}

We need to transform the  Cauchy-Szeg\H o kernel \eqref{eq:Szego2} on $  \mathcal{ D }$  to the one on the flat model $   \mathscr U := \mathscr N\times
\mathbb{R}^2_+$.
The identification $\iota: \mathscr U \rightarrow \mathcal{ D }$ is given by the quadratic mapping
\begin{equation*}
   (\mathbf{z},\mathbf{w})\mapsto (\mathbf{z}, {w}_1+\mathbf{i}(|{\mathbf
z}_1 |^2+|{\mathbf z}_3 |^2), {w}_2+\mathbf{i}(|{\mathbf z}_2 |^2+|{\mathbf z}_3 |^2)).
\end{equation*}

\begin{cor} \label{cor:flat-Szego1} The  Cauchy-Szeg\H o kernel   on the domain  $
\mathscr U $ is $S(( \boldsymbol\varepsilon ,\mathbf{g}),\mathbf{g}')=S_{\boldsymbol\varepsilon}( (\mathbf{g}')^{-1}\mathbf{g})  )$ with
 \begin{equation}\label{eq:flat-Szego1}
   S_{\boldsymbol\varepsilon}( \mathbf{h} ):= \frac { 1 }{ 4^2(\frac \pi
 2)^{N+2}}\sum_{k=0}^{n_3}\binom {n_3}{k}
 \frac { (n_1+k )!}{  (|{\mathbf z}_1 |^2 +  | {\mathbf z}_3 |^2+\varepsilon_1- \mathbf{i} t_1) ^{n_1+k+1}  } \frac { (n_2+n_3-k)!}{  (|{\mathbf z}_2 |^2 +  |
 {\mathbf z}_3 |^2+\varepsilon_2- \mathbf{i} t_2)^{n_2+n_3-k+1}  },
   \end{equation}
 where   $\mathbf{h}=(\mathbf{z} , \mathbf{t}) \in\mathscr  N $, $\boldsymbol\varepsilon=(\varepsilon_1,\varepsilon_2)\in \mathbb{R}^2_+$ and   $\mathbf{z}=(\mathbf{z}_1,\mathbf{z}_2,\mathbf{z}_3) \in \mathbb{C}^N$.
   \end{cor}
\begin{proof} For   $\zeta=(\zeta',\zeta'')=( {\mathbf z}, {\mathbf w} )\in\mathcal{D}$,
$\eta=(\eta',\eta'')=( {\mathbf
z}', {\mathbf w}' )\in   {\mathcal S}$, under the   identification $\iota$, we   write
\begin{equation}\label{eq:flat-Szego4}
     w_\alpha =   t_\alpha+\mathbf{i}(\varepsilon_\alpha+ |{\mathbf z}_\alpha |^2+|{\mathbf z}_3 |^2) , \qquad
   w_\alpha '=   t_\alpha' + \mathbf{i}( | {\mathbf
   z}_\alpha '|^2+|{\mathbf z}_3' |^2) ,
\end{equation} for  $\varepsilon_\alpha>0$.
  By definition \eqref{eq:rho}, we have
\begin{equation}\label{eq:flat-Szego5}\begin{split}
   \rho_\alpha(\zeta,\eta)& =\frac {
   w_\alpha- \overline{w'_\alpha}}{\mathbf{i}}  -2\langle  {\mathbf z}_\alpha ,   {\mathbf
   z}_\alpha'\rangle-2\langle  {\mathbf z}_3 ,   {\mathbf
   z}_3'\rangle\\&= \frac {t_\alpha-t_\alpha'}{\mathbf{i}}+\operatorname{Im }
   w_\alpha +\operatorname{Im }
   {w }_\alpha' -2\langle  {\mathbf z}_\alpha ,   {\mathbf
   z}_\alpha'\rangle-2\langle  {\mathbf z}_3 ,   {\mathbf
   z}_3'\rangle\\&=
    -\mathbf{i}(t_\alpha-t_\alpha')+ \varepsilon_\alpha+ | {\mathbf z}_\alpha |^2+ | {\mathbf z}_3 |^2+ | {\mathbf z}_\alpha '|^2+ | {\mathbf z}_3 '|^2-2\langle
    {\mathbf
    z}_\alpha , {\mathbf z}_\alpha'\rangle-2\langle  {\mathbf z}_3 ,   {\mathbf
   z}_3'\rangle\\
   &= - \mathbf{i}\Big(t_\alpha-t_\alpha'-2\operatorname{Im} \langle{{\mathbf z}}_\alpha',    {\mathbf
   z}_\alpha\rangle-2\operatorname{Im}\langle{{\mathbf z}}_3',    {\mathbf
   z}_3\rangle\Big)+ \varepsilon_\alpha+ | {\mathbf z}_\alpha -
   {\mathbf z}_\alpha'|^2+ | {\mathbf z}_3 -
   {\mathbf z}_3'|^2 .
\end{split}\end{equation}The result follows from \eqref{eq:Szego2} and the multiplication law  \eqref{eq:multiplication1}.
  \end{proof}

Now  consider
the product of three Siegel upper half spaces:
\begin{equation}\label{eq:tilde-U}
    \tilde{{\mathcal U} }=\Big\{(  \tilde{{\mathbf{z}}},  \tilde{ \mathbf  w} )\in
\mathbb{C}^{n_1+ n_2+ n_3} \times\mathbb{C}^3;\rho_\mu( \tilde
    { \mathbf z}_\mu, {\tilde
    w_\mu}):=\operatorname{Im}\tilde  {w}_\mu-|\tilde  {\mathbf z}_\mu|^2 >0, \mu=1,2,3\Big\},
\end{equation} where $  \tilde{ {\mathbf{z}}}=(  \tilde{{\mathbf{z}}}_1, \tilde{ {\mathbf{z}}}_2,  \tilde{{\mathbf{z}}}_3)\in \mathbb{C}^{n_1
}\times \mathbb{C}^{ n_2 }\times \mathbb{C}^{  n_3}$,
$   \tilde{\mathbf{w}}=(\tilde{w}_1,\tilde{w}_2,\tilde{w}_3)\in \mathbb{C}^3$.
The Shilov boundary of $ \tilde{ {\mathcal U}}$    is    defined by $\rho_1=\rho_2=\rho_3=0$, and has the structure of  the product of three
Heisenberg  groups. It is a  Siegel domain   with the regular cone $\Omega=
\mathbb{R}^{3}_+\subset    \mathbb{R}^{3}$ ($\Omega^*=
\mathbb{R}^{3}_+ $), $m=3$,    $N:=n_1+ n_2+ n_3$,  and
 \begin{equation*}
    \Psi(  \tilde{ {\mathbf z}} ,  \tilde{ {\mathbf z}}')=\Big(\langle  \tilde{ {\mathbf z
    }}_1 ,  \tilde{ {\mathbf z
    }}_1'\rangle,\langle  \tilde{ {\mathbf z
    }}_2 ,  \tilde{ {\mathbf z
    }}_2'\rangle,\langle  \tilde{ {\mathbf z
    }}_3 ,  \tilde{ {\mathbf z
    }}_3'\rangle\Big ),
 \end{equation*}
    and $\rho =(\rho_1 ,\rho_2,\rho_3)$  with
   \begin{equation}\label{eq:rho2}
   \rho_\mu(\zeta,\eta)=\frac {\tilde{ w}_\mu-\overline{ \tilde{  w}'_\mu }
}{\mathbf{i}} -2\langle  \tilde{ {\mathbf z}}_\mu ,
  \tilde{ {\mathbf
   z}}_\mu'\rangle ,
\end{equation}for   $\zeta=(\zeta',\zeta'')=(  \tilde{{\mathbf w}},  \tilde{{\mathbf z}})\in\mathcal{D}$,
$\eta=(\eta',\eta'')=( \tilde{{\mathbf w}}', \tilde{ {\mathbf
z}}')\in   {\mathcal S}$.
  Then,
  $
     B_\lambda(   \tilde{{\mathbf z }}  , \tilde{ {\mathbf z } }')=4\lambda_1\langle  \tilde{ {\mathbf z }}_1
  ,   \tilde{{\mathbf z
    }}_1'\rangle+ 4\lambda_2\langle  \tilde{{\mathbf z }}_2 ,   \tilde{{\mathbf z
    }}_2'\rangle+ 4 \lambda_3\langle \tilde{ {\mathbf z }}_3 ,  \tilde{ {\mathbf z
    }}_3'\rangle
$,
    i.e.
\begin{equation*}
   B_\lambda=4
   \left(\begin{matrix} \lambda_1I_{n_1}&0&0\\
   0&\lambda_2I_{n_2}&0\\
   0&0&  \lambda_3I_{n_3}
  \end{matrix}\right),
\end{equation*}and so
 $\det B_\lambda=4^{N} \lambda_1^{n_1} \lambda_2^{n_2}  \lambda_3^{n_3}  $.
It follows from \eqref{eq:Szego} that
\begin{equation} \label{eq:Szego2'}\begin{split}
 \tilde   S(\zeta, \eta)&= \int_{\mathbb{R}^{3}_+ }e^{-2\pi\sum_{\mu=1}^{3} \lambda_\mu\rho_\mu(\zeta,\eta)
   }4^{N} \lambda_1^{n_1} \lambda_2^{n_2}  \lambda_3^{n_3}  d\lambda
     =\prod_{\mu=1}^3
 \frac {c_\mu}{\rho_\mu(\zeta,\eta) ^{n_\mu+1}} ,
\end{split}\end{equation}where $c_\mu$'s are given by \eqref{eq:flat-Szego2}.

Similarly, the  Cauchy-Szeg\H o kernel \eqref{eq:Szego2'} on $   \tilde{{\mathcal U} }$  can also be transformed to the one in
\eqref{eq:flat-Szego-tilde}-\eqref{eq:flat-Szego2} on the flat model $    \tilde{\mathscr U }:= \tilde{\mathscr N}\times
\mathbb{R}^3_+$ by
the identification $ \tilde\iota:  \tilde{\mathscr U} \rightarrow   \tilde{{\mathcal U} }$  given by the quadratic mapping
\begin{equation}\label{eq:tilde-ota}
  (\mathbf{z},\mathbf{w})\mapsto (\mathbf{z}, {w}_1+\mathbf{i}|{\mathbf
z}_1 |^2, {w}_2+\mathbf{i}|{\mathbf z}_2 |^2 , {w}_3+|{\mathbf z}_3 |^2) .
\end{equation}

\end{document}